\documentclass[reqno, a4paper]{amsart}
\usepackage[dvipsnames]{xcolor}
\usepackage{tikz-cd}
\usepackage[utf8]{inputenc}
\usepackage[T1]{fontenc}
\usepackage{amsmath}
\usepackage{amsfonts,amssymb}
\usetikzlibrary{matrix}

\makeatletter
\newcommand{\vast}{\bBigg@{4.030}}
\newcommand{\Vast}{\bBigg@{12.30}}
\makeatother
\usepackage{footmisc}
\usepackage{tabularx}

\newcommand{\gris}[1]{\textcolor{gray}{#1}}
\newcommand{\azul}[1]{\textcolor{azulESI}{#1}}

\definecolor{azulESI}{HTML}{1266AE}

\definecolor{AZULESI}{HTML}{1266AE}
\usepackage{enumerate}
\usepackage{enumitem}
\usepackage{nicematrix}
\usepackage{soul}
\usepackage[utf8]{inputenc}

\newcommand{\includegraphicsc}[2]{\raisebox{-.35\height}{\includegraphics[width=#1cm]{#2} }  }

\makeatletter
\ifcase \@ptsize \relax% 10pt
  \newcommand{\nano}{\@setfontsize\miniscule{3.5}{4.5}}%\newcommand{\Dcat}{\mathscr{D}}
\or% 11pt
  \newcommand{\nano}{\@setfontsize\miniscule{4.5}{5.5}}%
\or% 12pt
  \newcommand{\nano}{\@setfontsize\miniscule{4.5}{5.5}}%
\fi
\makeatother

\newcommand{\balita}{\raisebox{1.9pt}{\text{\nano$\bullet$\hspace{.7pt}}}}
\newcommand{\K}{\mathcal{C}}

\newcommand{\balitabaja}{\raisebox{.19pt}{\text{\nano$\bullet$\hspace{.7pt}}}}
\usepackage{tikz}
\usetikzlibrary{automata,arrows,topaths}
\usepackage{caption}

\usepackage[bookmarks=false]{hyperref}
\hypersetup{
         colorlinks   = true,
         citecolor   =azulESI,
        urlcolor = azulESI
}
\hypersetup{linkcolor=azulESI}

\newcommand{\pS}{{\mathscr{pS}}}
\newcommand{\ppS}{{\tilde{\mathscr{p}}\mathscr{S}}}
\newcommand{\inv}{^{-1}}
\newcommand{\sour}{_{s(e)}}
\newcommand{\targ}{_{t(e)}}
\newcommand{\starcirc}{^{\star\!\!\!\circ}}
\newcommand{\kol}{\mathring} %\kółko faktycznie

\newcommand{\Wils}{\mathcal{W}}

\newcommand{\itemb}{\item[\balita]}

\usepackage[cal=euler, scr=boondoxo]{mathalfa}

\usepackage{float}
\usepackage{subcaption}
\usepackage{amsthm}
\numberwithin{equation}{section}
\newtheoremstyle{mytheoremstyle} % name
    {10pt}                    % Space above
    {8pt}                    % Space below
    {\itshape}                   % Body font
    {}                           % Indent amount
    {\scshape}                   % Theorem head font
    {.}                          % Punctuation after theorem head
    {.5em}                       % Space after theorem head
    {}  % Theorem head spec (can be left empty, meaning ‘normal’)

\makeatletter
\newcommand{\leqnomode}{\tagsleft@true}
\newcommand{\reqnomode}{\tagsleft@false}
\makeatother
\theoremstyle{mytheoremstyle}

\newtheorem{theorem}{Theorem}[section]
 \newtheorem{corollary}[theorem]{Corollary}
 \newtheorem{lemma}[theorem]{Lemma}
 \newtheorem{claim}[theorem]{Claim}
 \newtheorem{proposition}[theorem]{Proposition}
 \newtheoremstyle{definition} % name
    {8pt}                    % Space above
    {5pt}                    % Space below
    {}                   % Body font
    {}                           % Indent amount
    {\scshape}                   % Theorem head font
    {.}                          % Punctuation after theorem head
    {.5em}                       % Space after theorem head
    {}  % Theorem head spec (can be left empty, meaning ‘norma13pt]{holl’)

 \theoremstyle{definition}
 \newtheorem{definition}[theorem]{Definition}
 \newtheorem{example}[theorem]{Example}
 
 \newtheorem{remark}[theorem]{Remark}
 \newtheorem{notation}[theorem]{Notation}

\usepackage{mathabx}

\newcommand{\N}{\mathbb{Z}_{>0}}
\newcommand{\Z}{\mathbb{Z}}
\newcommand{\Zpos}{\mathbb{Z}_{>0}}
\newcommand{\C}{\mathbb{C}}
\newcommand{\R}{\mathbb{R}}
\newcommand{\runter}[1] {\raisebox{-.45\height}{#1}}
\newcommand{\runterhalb}[1] {\raisebox{-.1\height}{#1}}

\newcommand{\sumsub}[1]{\sum_{\substack{#1}}}
\newcommand{\coprodsub}[1]{\coprod_{\substack{#1}}}
\newcommand{\dif}{{\mathrm{d}}}
\newcommand{\symmetr}{^\mathrm{sym}}
\newcommand{\M}[1]{M_{#1}(\mathbb{C})}
\newcommand{\uni}{\mathrm{U}}

\newcommand{\cC}{\mathcal{C}}
\newcommand{\bn}{\mathbf{n}}
\newcommand{\be}{\mathbf{e}}
\newcommand{\bmu}{\boldsymbol{\mu}}
\newcommand{\bv}{\mathbf{v}}
\newcommand{\ba}{\mathbf{a}}
\newcommand{\bm}{\mathbf{m}}

\newcommand{\bq}{\mathbf{q}}
\newcommand{\br}{\mathbf{r}}
\newcommand{\trans}{^{\text{\tiny{$\mathrm T$}}}}

\newcommand{\latt}{^{\text{\tiny{lattice}}}}
\newcommand{\ii}{\mathrm{i}}
\newcommand{\ee}{\mathrm{e}}
\newcommand{\Brat}{\mathscr{B}}
\newcommand{\se}{_{s(e)}}
\newcommand{\te}{_{t(e)}}

\DeclareMathOperator{\Rep}{\mathrm{Rep}}
\DeclareMathOperator{\vol}{\mathrm{vol}}
\DeclareMathOperator{\id}{\mathrm{id}}
\DeclareMathOperator{\hol}{\mathrm{hol}}
\DeclareMathOperator{\ReppS}{\mathrm{Rep}_{\pS}}
\DeclareMathOperator{\modpS}{\text{-}\mathrm{mod}_{\pS}}
\DeclareMathOperator{\End}{\mathrm{End}}

\DeclareMathOperator{\diag}{\mathrm{diag}}
\DeclareMathOperator{\Sym}{\mathrm{Sym}}

\DeclareMathOperator{\Adj}{Ad}

\DeclareMathOperator{\Aut}{Aut}
\DeclareMathOperator{\Tr}{Tr}

\newcommand{\hp}[1]{^{(#1)}}
\usepackage[margin=.850in]{geometry}
\tikzcdset{arrow style=tikz, diagrams={>={Stealth[round,length=4pt,width=4.95pt,inset=2.75pt]}}}

 \title[The Spectral Action on Quivers]{Bratteli networks and \\ the Spectral Action on Quivers}
 \author[C. I. Perez-Sanchez]{Carlos I. P\'erez S\'anchez}

  \address{University of Heidelberg, Institute for Theoretical Physics,\newline \indent
  Philosophenweg 19, 69120 Heidelberg, Germany   \newline \indent
 \hspace{.0cm}\& \newline \indent Erwin Schrödinger International
 Institute for Mathematics and Physics, \newline \indent University of
 Vienna, Boltzmanngasse 9 1090 Wien, Austria }
  \email{perez@thphys.uni-heidelberg.de}

\usepackage{stackrel}

\makeatletter

\newcommand*\notocchapter[1]{%
  \if@openright\cleardoublepage\else\clearpage\fi
  \thispagestyle{empty}\global\@topnum\z@
  \@afterindenttrue
  \let\@secnumber\@empty
  \@makeschapterhead{#1}\@afterheading
}

\makeatother

\begin{document}
\begin{abstract}
 In the context of noncommutative geometry, we consider quiver representations---not on vector spaces, as traditional, but on finite-dimensional prespectral triples (`discrete topological noncommutative spaces').  The similar idea developed first by Marcolli-van Suijlekom of representing quivers in spectral triples (`discrete noncommutative geometries') paved the way for some of the next results.  We introduce Bratteli networks, a structure that yields a neat combinatorial characterisation of the space $\mathrm{Rep}\, Q $ of prespectral-triple-representations of a quiver $Q$, as well as of the gauge group and of their quotient.  Not only these claims that make it possible to `integrate over $\mathrm{Rep}\, Q$' are, as we now argue, in line with the spirit of random noncommutative geometry---formulating path integrals over Dirac operators---but they also contain a physically relevant case. Namely, the equivalence between quiver representations and path algebra modules, established here for the new category, inspired the following construction: Only from representation theory data, we build a spectral triple for the quiver and evaluate the spectral action functional from a general formula over closed paths. When we apply this construction to lattice-quivers, we obtain not only Wilsonian Yang-Mills lattice gauge theory, but also the Weisz-Wohlert-cells in the context of Symanzik's improved gauge theory.  We show that a hermitian (`Higgs') matrix field emerges from the self-loops of the quiver and derive the Yang-Mills--Higgs theory on flat space as a smooth limit.
\end{abstract}
 \maketitle%
% \vspace{-2ex}
 \fontsize{11.35}{14}\selectfont%
\section{Motivation and Introduction}\label{sec:motivation}
\captionsetup{font=small} Quiver representations is a discipline of
relevance in algebraic geometry, invariant theory, representation of
algebraic groups \cite{bookQuivRep} and several other fields of
mathematics and physics.  From time to time, new applications of
quiver representations are discovered: they compute Donaldson-Thomas
invariants \cite{DT}, they yield \textsc{homfly-pt} polynomials in
knot theory \cite{Kucharski:2017ogk,Ekholm:2018eee}, which harmonises
with topological recursion \cite{QuiversFUW}, just to mention a
non-comprehensive list on contemporary developments.  Quiver
representation theory often builds unexpected bridges among topics one
initially thinks to lie far apart. Yet another example of this is
\cite{MvS}, which connects spin networks with noncommutative geometry
and lattice gauge theory.  In the present article we report, in a
self-contained way, progress on the relation between the latter two
topics, adopting a quiver representation viewpoint (we will not use spin networks at all). In this section we
motivate our investigations in informal style, prior to the technical
part that starts in Sec.~\ref{sec:geometry}. \par
% \par
% \begin{figure}[b!]%
% \runter{\includegraphics[width=3cm]{regions_annotated}} $\qquad\rightsquigarrow$ \qquad
% \runter{\includegraphics[width=1.95cm]{regions_quiver}}
%  \caption{Some regions $U_1,\ldots, U_5$ are neighbouring in some way
%    that is encoded in the dual graph of the right.  In the left panel
%    each $U_i$ illustrates an open region of Euclidean space, but each
%    $U_i$ can be for instance a simplex (without reference to
%    embedding).\label{fig:regions}}%
% \end{figure}%

A quiver is a directed multi-graph like
$\runterhalb{\includegraphics[width=1.7cm]{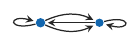}}\!\!$ (this one was
randomly picked).  At least at the heuristic level we pursue in this
introduction, it is not difficult to understand why this kind of
graphs can be used in physics (cf. Figs. \ref{fig:regions} and
\ref{fig:nonsymm} for two naive examples) even though a quiver by
itself cannot have physical information; it is rather the `shadow' of
the actual interaction between subsystems, points or regions (for
which more precise algebraic words exist), which can be expressed as a
commutative diagram in a suitable category. \par
% (We postpone the
% serious list of applications to Section \ref{sec:perspectives},
% since they require the technical part.) \par
An ordinary representation of a quiver
labels its vertices with vector spaces and assigns
linear maps to its arrows. This means that a representation is a
functor from the free category of the quiver (whose objects are the
vertices and the morphisms from a vertex $v$ to other $w$
are all directed paths from $v$ to $w$) to the category of vector
spaces.
However, depending on the problem, vector spaces
might not retain the whole information and this target category has to be replaced.
For the case at hand, gauge theory, we use a suitable
target category that has its roots in noncommutative geometry.\\

\begin{figure}
\begin{subfigure}{0.36\textwidth} \centering
\runter{\includegraphics[width=2.5cm]{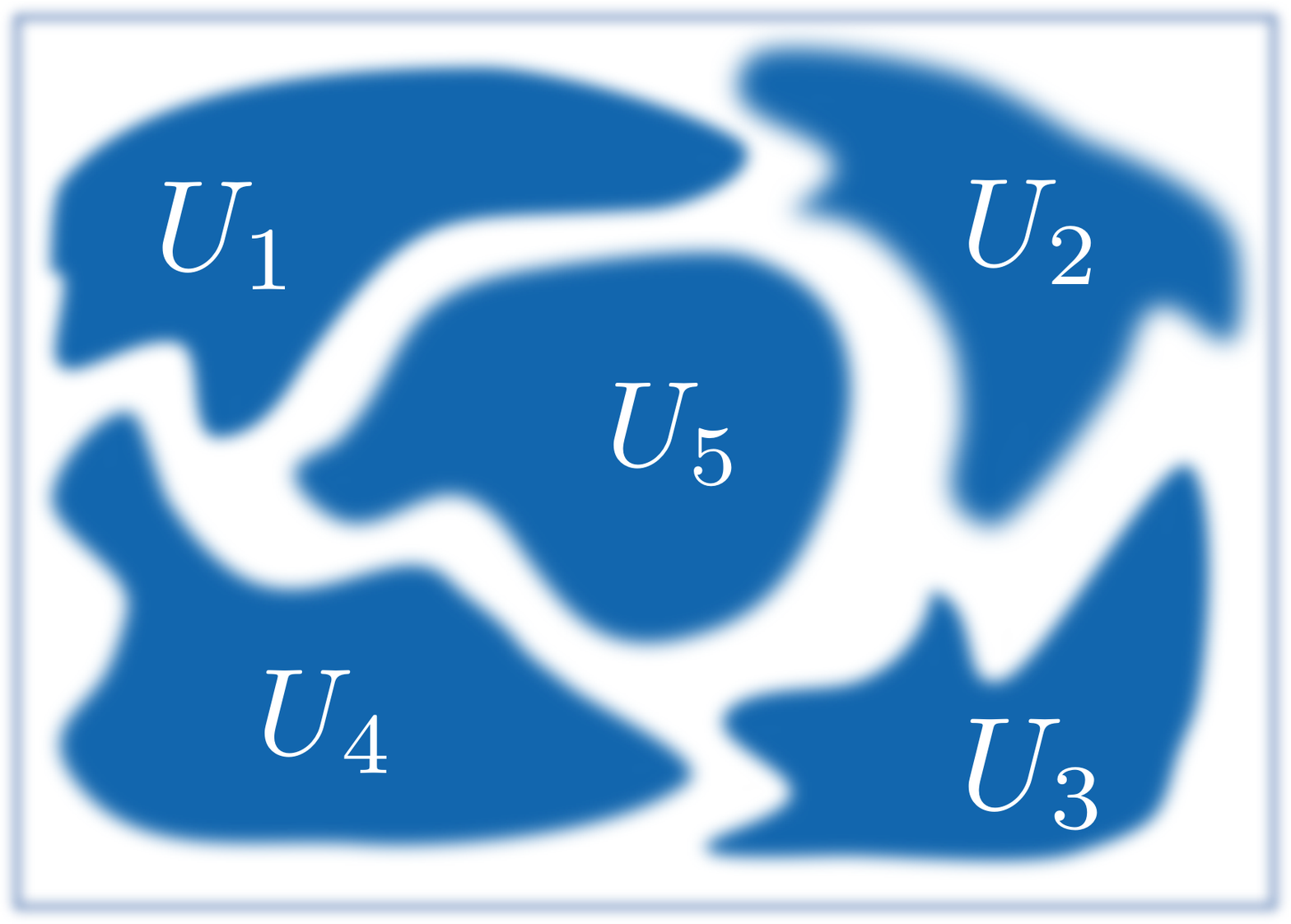}}\,\,
\,\,$\rightsquigarrow$ \,\,
\runter{\includegraphics[width=1.5cm]{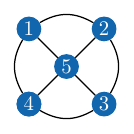}} \\[2.35ex]
 \caption{Some regions $U_1,\ldots, U_5$ are neighbouring in some way
   that is encoded in the dual graph of the right.  In the left panel
   each $U_i$ illustrates an open region of Euclidean space.\label{fig:regions}}%\runter{\includegraphics[width=1.95cm]{triangulo_quiver}}
%  \caption{Each $U_i$ region is filled of particles of certain type $i=1,2,3$. A membrane
%  (with `chemical potential' $\mu_i$ as depicted) at
%  $\partial U_j$ lets pass $i$-type particles,
%  which we represent by an arrow
% $i\to j$.
%  For this naive asymmetric interaction,
% the quiver is more suitable than the ordinary
%    graph $\stackbin[1]{}{\azul{\bullet}}\!\!-\!\!\stackbin[3]{}{\azul{\bullet}}\!\!-\!\!\stackbin[2]{}{\azul{\bullet}}{\vphantom -}\,$.
%  \label{fig:nonsymm}}
\end{subfigure}
\hfill
\begin{subfigure}{0.61\textwidth} \centering
\runter{\includegraphics[width=5cm]{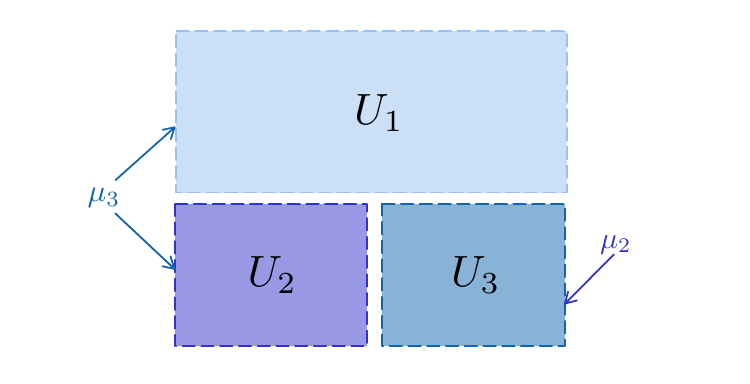}} \hspace{-4ex} $\rightsquigarrow$ \quad
\runter{\includegraphics[width=1.5cm]{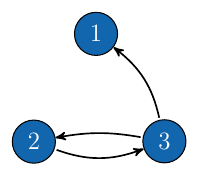}} \\[-1.85ex]
 \caption{Each $U_i$ region is filled of particles of certain type $i=1,2,3$. A membrane
  at
 $\partial U_j$ (labeled with a `chemical potential' $\mu_i$) lets pass $i$-type particles,
 which we encode by an arrow
$i\to j$ between nodes.
 For this naive asymmetric interaction,
the quiver is more suitable than the ordinary
   graph $\stackbin[1]{}{\azul{\bullet}}\!\!-\!\!\stackbin[3]{}{\azul{\bullet}}\!\!-\!\!\stackbin[2]{}{\azul{\bullet}}{\vphantom -}\,$.
 \label{fig:nonsymm}}
\end{subfigure} \caption{Naive examples with graphs and quivers}
\end{figure}
% \begin{figure}[t!]
% \runter{\includegraphics[width=5.85cm]{chempot2}} \hspace{-4ex} $\rightsquigarrow$ \quad
% \runter{\includegraphics[width=1.95cm]{triangulo_quiver}}
%  \caption{Each $U_i$ region is filled of particles of certain type $i=1,2,3$. A membrane
%  (with `chemical potential' $\mu_i$ as depicted) at
%  $\partial U_j$ lets pass $i$-type particles,
%  which we represent by an arrow
% $i\to j$.
%  For this naive asymmetric interaction,
% the quiver is more suitable than the ordinary
%    graph $\stackbin[1]{}{\azul{\bullet}}\!\!-\!\!\stackbin[3]{}{\azul{\bullet}}\!\!-\!\!\stackbin[2]{}{\azul{\bullet}}{\vphantom -}\,$.
%  \label{fig:nonsymm}}
% \end{figure}

From amidst the largely diversified landscape of noncommutative
geometry \cite{ConnesBook} exclusively spectral triples are treated
here. These are the noncommutative generalisation of Riemannian
manifolds in the sense in which Gelfand-Naimark theory generalises
locally compact Hausdorff spaces to operator algebras. This is
reflected in the first item of a spectral triple $(A,H,D)$, a unital
involutive algebra $A$ that might be noncommutative (which often is a
C$^*$-algebra too and, in the above spirit, is the noncommutative
`space of functions'); secondly, a Hilbert space $H$ is required, with
an involution-preserving action $A \curvearrowright H$; and finally,
there is a self-adjoint operator $D$ on $H$, referred to as the Dirac
operator.  This nomenclature, and the axioms that spectral triples
satisfy (cf. Rem.~\ref{rem:infdims} for the omitted conditions) come
from Riemannian geometry.  The extent to which the action of the
algebra and the Dirac operator commute determines an abstract distance
on the spectral triple; the prescription is known as Connes' geodesic
distance formula \cite[Sec. 1]{ConnesBook}, since it reduces to the
geodesic distance for Riemannian manifolds, when a spectral triple
has a commutative algebra, provided that also more axioms are
verified \cite{Reconstr}.  \par

The dynamics of spectral triples was described first by
Chamseddine-Connes in what is referred to as the spectral action
\cite{Chamseddine:1996zu}.  This action reduces to Einstein's gravity
(plus mild terms) if the spectral triple is the canonical dual to a
(compact) Riemannian manifold, and to the Einstein--Yang-Mills--Higgs
theory if one allows spectral triples\footnote{Allowing the very same
noncommutative spectral triples, one can obtain the full Standard
Model of Particle Physics by adding a `fermionic Spectral Action',
which however, is outside this scope of this article. We refer to the
textbook \cite{WvSbook} and also to the original article
\cite{Chamseddine:2006ep}.} to have a noncommutative algebra of
a certain type (essentially the algebra of functions on the manifold
tensored with a matrix algebra).

\subsection{Preliminary work and two observations}\label{sec:Preliminarywork}

To our knowledge, the first to represent quivers in the above noncommutative
geometrical context were Marcolli and van Suijlekom. Because their work
had a strong influence in our constructions, it is convenient to
briefly mention already in this introduction what they achieved in
\cite{MvS}:

\begin{itemize}
 \item[({I})]
The category of spectral triples $\mathcal S$ was defined, as well as
$\mathcal{S}_0$,  its subcategory characterised by having a vanishing Dirac operator.
The space $\Rep _{\mathcal{S}_0}\!Q$ of
representations of a quiver $Q$ in
$\mathcal{S}_0$ was expressed in terms of unitary groups and homogeneous spaces;
the invertible natural transformations (that form the gauge group $\mathsf G$) between the
elements of $\Rep _{\mathcal{S}_0}\!Q$ was provided.

\vspace{1ex}
\item[(II)] From an orthogonal basis decomposition of the space $L^2
  (\Rep_{\mathcal{S}_0}\!Q)^{\mathsf G}=L^2(\Rep _{\mathcal
  S_0}(Q)/\mathsf G)$ the main object in \textit{op. cit.}, gauge
  networks, was constructed in terms of intertwiners at the vertices
  of the quiver. This structure generalises spin networks to a
  noncommutative geometrical setting. (Gauge networks will not
  reappear here, cf. Sec. 3 of \textit{op. cit.}  for the omitted
  results.)

\vspace{1ex}
\item[(III)] Starting with a background spin manifold to embed the
  quiver in, a spectral triple was defined from the geometry that the
  spin connection induces on the quiver.  After twisting its Dirac
  operator with quiver representation data, a lattice gauge theory on
  $Q=\Z^4$ was obtained from the spectral action.  This theory has the
  Yang-Mills--Higgs action on $\R^4$ as smooth limit.
% In particular, certain self-adjoint matrices $D_v$ that
% the quiver representation associates to the
% vertices of $\Z^4$ model the Higgs in the continuum.
\end{itemize}

\vspace{0ex}
\subsubsection{A remark about the Higgs field}\label{sec:Higgs_to_stała}
In Section
\ref{sec:theMvScategory} below, it will be shown that
the morphism structure given in (I) to
the category $\mathcal S$ of spectral triples, which has been used in (III), implies that the
matrices $D_v$ and $D_w$ that a quiver representation in $\mathcal S$  attaches to
two arbitrary vertices, $ v,w\in
\Z^4$,  have the same spectrum\footnote{The author is in debt
with Sebastian Steinhaus, without whose remark the proof given Sec.~
\ref{sec:theMvScategory} would not exist.}. Since the Higgs action functional in \cite[Prop. 29]{MvS} depends
only on the traces of those matrices at the vertices, such Higgs
is a constant field over $\R^4$, in the smooth limit.  \\

In order to get a non-constant Higgs field,
we change from (I)
the target category (dropping
said Dirac operators at the vertices; see also Rem.~\ref{rem:kernels} and Sec.~\ref{sec:kernels}
for more technical details)
and look for another mechanism
for the Higgs field to be generated.
Since the objects of the new  category are simpler (`simple' in a colloquial sense),
some of the results grouped in (I) can be obtained for the new category
without starting from scratch, but since they do not follow from (I) automatically,
they are spelled out in this article; this explains its length. However,
in view of  the  following
observation, there is an additional subtlety that,
to the knowledge of the author, has not been addressed before.

\subsubsection{A remark about the gauge group of the quiver}
Some facts about quiver representations in (say, complex) vector
spaces, $\mathsf{Vect}$, enjoy properties coming from the
$\hom_{\mathsf{Vect}}$-sets not being empty, regardless of the object
pair taken as input.  The space of functors $\Rep_{\mathsf{Vect}}Q:=\{ \mathcal P Q \to
\mathsf{Vect} \} $ from the free or path category $\mathcal P Q$ to $\mathsf{Vect}$,
as well as the invertible natural
transformations $\mathcal G_{\mathsf{Vect}}(Q) $ between these
functors read as follows (where $Q_0$ denotes the vertices of $Q$ and $Q_1$ its
oriented edges):
 \[
\Rep_{\mathsf{Vect}}Q = \coprod_{ m : {Q_0}\to \Z_{\geq 0} }
\prod_{\substack{e \in Q_1 \\
e : v  \to w }}
\C^{m_v\times m_w}
\quad \text{ and } \quad \mathcal G_{\mathsf{Vect}}(Q) =
\coprod_{ m : {Q_0}\to \Z_{\geq 0} } \prod_{v\in Q_0} \mathrm{GL}(m_v, \C).
\]
 In the leftmost expression, each $
\C^{m_v\times m_w}$  parametrises the morphisms
 from the vector space at $v$  to that at  $w$
 ($v,w\in Q_0$, and observe that there is one copy of that space
 for each edge  from $v$ to $w$);
the disjoint union over the dimension-labels $m$
lists all such assignments. The invertible natural
transformations $\mathcal G_{\mathsf{Vect}}(Q) $ can be written  as above
over all dimension-assignments $m : {Q_0}\to \Z_{\geq 0}$, since  $\mathsf{Vect}$ knows
that, for all $k,l\in \Z_{\geq 0}$, $\hom_{\mathsf{Vect}}(\C^{k},\C^{l})$ is never empty---but what if it were?
\par
For the target
categories $\mathcal{C}$ of our interest---think of $\K$
as the category of involutive algebras for the time being---the sets
$\hom_{\mathcal{C}}(Y,Z)$ can be empty for some pair of objects $(Y,Z)$,
demanding  special care.
% This is not a problem when describing (the objects of)
% the representation category of a quiver $Q$  (with vertex-set $Q_0$ and edge-set  $ Q_1$)
% by the expression on the right,
While the functors $\mathcal P Q \to \K$ are, even if not economically\footnote{The proof of this claim is presented
 in Rem.~\ref{rem:GaugeGroups_differences} below,
 but the reason for the said non-economic expression is that
 not all maps $Y: Q_0\to \K$ lead to a $\K$-representation of $Q$.
It suffices, e.g. that a certain $Y^\diamond :Q_0 \to \K $
yields no morphisms $\{Y^\diamond _v \to Y^\diamond_w\} $ for
some edge $(v,w)$;  the reason for `well-described' is that for those
 labels the product vanish. But the  gauge group is subtle.}, at least well-described by
\begin{align}\label{redundantlabelsRepQ}
% \Rep_\K(Q)  = \{\text{functors} : \text{free cat. of $Q$} \to \K \}=
\coprod_{Y: Q_0\to \K} \,\prod_{(v,w)\in Q_1} \hom_\K ( Y_{v}, Y_w), \quad
\end{align}
the gauge group (of invertible natural transformations between those functors)  is smaller than
\begin{align}\label{redundantlabelsGauge}
%  \qquad\,\,\,
 \coprod_{Y: Q_0\to \K} \,\,\prod_{v\in Q_0} \Aut_\K (Y_v).
 \end{align}
This is because \eqref{redundantlabelsGauge}  does not detect when
a certain map $Y^\diamond: Q_0\to \K$ in \eqref{redundantlabelsRepQ}
yields an empty product over edges, and the product over vertices in \eqref{redundantlabelsGauge} is not empty for such map $Y^\diamond$.\\[-1ex]

          Expressions like
          \eqref{redundantlabelsGauge} were supposed to describe the quiver's gauge group in an older version of this manuscript and in \cite[Prop. 13]{MvS} (for the respective categories of interest),
          but such description yields a group that is too large, because not all vertex-labels $Y$ lift to a functor $\mathcal P Q\to \K$. \textit{Bratteli networks} are simple combinatorial data introduced
here with a twofold aim:
\begin{itemize}

           \itemb They make it possible to rewrite
           \eqref{redundantlabelsRepQ} as a list of actual
           contributions $Y$, for which the hom-sets $\hom_\K ( Y_{v},
           Y_w)$ are non-empty for each edge $(v,w)$. See
           Example \eqref{ex:addedV2}.
           
\vspace{.5ex} \itemb While the previous point is only about economy,
Bratteli networks become vital to determine the gauge group.  For
categories like $\K$ with the property  described above, the gauge group is properly
contained in \eqref{redundantlabelsGauge} for a general quiver. See
Example \eqref{ex:GGdifferences}.

          \end{itemize}
          Besides the aforementioned examples that rapidly illustrate this,
          Remark \eqref{rem:GaugeGroups_differences} proves these claims.

\subsection{Strategy and Results}

With the physical motivation of the previous sections, we choose
another target category to represent quivers, i.e.~different from
vector spaces---the usual category involved in quiver
representations---but also slightly different from the category used
by Marcolli and van Suijlekom (cf. Rem.~\ref{rem:kernels}).  We call
`prespectral triples' such category, which we now fix
% as the target category of our quiver representations
for the rest of the next summary of results:

\begin{itemize}[leftmargin=12pt]
  \setlength\itemsep{.5ex}

\itemb We use Bratteli networks to determine
the space of quiver representations, and
of this space modulo
natural equivalence. These structures yield
non-redundant expressions (in the sense of the two points at the end of the
last subsection)  like Theorem \ref{thm:RepQ_modG},
which are needed as the integration
domains for our path integrals in the  last point of this list.

  \itemb The name `prespectral' in the target category suggests that
  its objects are awaiting some sort of completion (by adding a Dirac
  operator), after which it becomes `spectral', which is true, but not
  how we shall proceed. Our construction is subtle in the sense that
  the important spectral triple will be one assembled for the whole quiver, not
  requiring to complete the prespectral triples at the
  vertices (recall Sec.~\ref{sec:Higgs_to_stała}).  Given a quiver representation, we construct a
  spectral triple, in particular providing a Dirac operator from
  representation theory data.  Such Dirac operator is a noncommutative
  version of the adjacency matrix of the quiver in the sense of its
  entries being operator-valued, which is in fact an
%   ; the entries of the Dirac operator
%   are, instead of the number of incoming or outgoing arrows, an
  abstraction of the parallel transport along those.  We compute the
  spectral action---essentially a trace of functions of the Dirac
  operator---in terms of closed paths on the quiver.  This is not by a
  coincidence: it reflects the equivalence between quiver
  representations and modules for the path algebra of the quiver for
  our new category (Prop.~\ref{thm:Eq_Cats}).

  \itemb The tools referred to in the previous point allow to derive
  Yang-Mills theory---pure or coupled to a \textit{bona fide}
  hermitian matrix field that can serve as a Higgs---in Sec.~
  \ref{sec:lattice_paths} from representations on lattices, grasped as
  quivers; we present this in arbitrary dimension.  The Higgs scalar
  emerges from the self-loops of the quiver. (A posteriori one could
  think of them as a graph-theoretical Kaluza-Klein picture, but this
  picture is also reminiscent of early Higgs models based on
  a two-point space collapsed
  to one\footnote{The equivalence relation is, classically, a point,
  but well-known to generate the non-diagonal elements of the matrix
  algebra resulting from the groupoid algebra associated to the
  equivalence relation. The two-point space is, historically, one of
  the fundamental steps to generate the Higgs, cf. \cite{ConnesBook},
  or its first French version, where this aspect is discussed in depth
  (even if that is not the most up-to-date version).}.)  The smooth limit
  is obtained in Theorem \ref{thm:YMH_explicitandsmooth}.  To further
  test our theory, we show how representations of certain lattices
  generate the (Lüscher--)Weisz-Wohlert action cells
  \cite[Eq. 2.1]{Weisz:1983bn}, which extend the Wilson action of
  gauge theory to Symanzik's improved gauge theory programme
  \cite{Symanzik:1981hc}. (See Rem.~\ref{rem:uwaga_Higgs} about
  physical improvements.)

\itemb We work towards a formalism of path integral over Dirac
operators. The results of Section \ref{sec:rep_theo} allow for an
integration over the space of representations $\Rep Q$ studied here.
When stated in terms of Dirac operators $D$ (parametrised by data of
$\Rep Q$), one obtains a partition function of the form $\int
\ee^{-S(D)} \dif D$ where $S(D)$ is the spectral action and $\dif D$
can be constructed \cite{MakeenkoMigdalNCG} as a product
Haar measure.  (That this partition
function models gauge-Higgs interactions can be deduced from
Sec.~\ref{sec:physics}.) Hence, despite the different origin of our
Dirac operator, this path integral formulation is in line with `random
noncommutative geometry' \cite{Barrett:2015foa, Hessam:2022gaw}, which
is motivated by quantum gravity (striving for Dirac operator integrals
instead of integrals over metrics).  For the Dirac ensembles from the
present article, Dyson-Schwinger or Makeenko-Migdal equations were
addressed in \cite{MakeenkoMigdalNCG} (and in easy cases, solved via positivity constraints).

\end{itemize}
\noindent
Besides the previous list, unlike \cite{MvS}, we do not assume a
manifold as input.  In our setting, this Dirac operator emerges purely
from representation theory data.  This could appear from the
mathematical perspective as irrelevant, but when gauge
theories---mostly addressed classically here---are eventually
quantised and coupled to gravity, the macroscopic object, or manifold,
is expected to emerge from the microscopic one, embodied in the quiver
[or infrared (or low energy) physics being emergent from ultraviolet
  (or high energy) physics].  It was therefore natural to ask whether
the lattice gauge theory of \textit{op. cit.} can be reconstructed
without reference to a manifold. We answer this question positively,
without the use of a spin structure (to get a fully realistic model,
at least a Clifford module and a chirality structure seem to be
required). Instead, the holonomy is the geometrical variable, which
elsewhere has been grasped as the fundamental variable in the geometry
of physical theories, e.g. Yang-Mills and General Relativity;
cf. e.g. Barrett's PhD thesis or \cite{BarrettPhD}.  A proposal for
quantisation is sketched at the end, as a perspective. We hint at a
list of symbols and conventions in the appendix, along with some
auxiliary facts about combinatorics of closed paths.

\setcounter{tocdepth}{2}

\begin{center}
\tableofcontents
\end{center}

\thispagestyle{empty}

\section{Prespectral triples}\label{sec:geometry}

\subsection{The category of prespectral triples}\label{sec:ferm_systems}

\begin{definition}

A \textit{prespectral triple} is a triple $(A,\lambda, H)$ consisting
of a finite dimensional involutive algebra or $*$-algebra $A$, which
is unital ($1\in A$) and of a finite-dimensional vector space $H$ with
inner product, or Hilbert space, that serves as an $A$-module; we
denote by $\lambda$ the corresponding $*$-action $A \curvearrowright
H$, which we impose to be faithful. Sometimes we leave $\lambda$
implicit, in case that omitting it does not lead to ambiguity; thus
the triple will look like a double $(A,H)$.  The set of prespectral
triples, denoted by $\ppS$, will be given a category
structure\footnote{At the moment $\ppS$ should be thought as a whole
symbol (without attention to the tilde).  This will disappear in
favour of a simpler notation, is reserved for the main category
later.}. Writing $X= (A, \lambda, H )$ and $X'= (A', \lambda',H')$ a
morphism $(\phi,L) \in \hom_\ppS(X,X')$, is a $*$-algebra map
$\phi:A\to A'$ together with a transition map, $L: H \to H'$. By
definition, this is a unitary matrix (or unitarity $L^*L=1_{H}$,
$LL^*=1_{H'}$) obeying
\begin{align*}
 \lambda' [  \phi(a) ] = L \lambda (a) L^*\text { for all $a \in A $}\,.
\end{align*}Given $X_1=(A_1,\lambda_1,H_1),X_2=(A_2,\lambda_2,H_2)
\in \ppS$ one can build their \textit{direct sum} $X_1\oplus X_2$,
whose algebra is given by $A_1\oplus A_2$. This algebra acts on
$H_1\oplus H_2$ by multiplication by the block matrix $( \lambda_1
\oplus \lambda_2 )(a_1,a_2) = \diag( \lambda_1 (a_1 ) \oplus \lambda
_2(a_2)) $, $a_i\in A_i$. This action is then faithful too.
\end{definition}

\begin{remark} \label{rem:kernels}
 A category $ \mathcal{S}_0$ that allows for a non-zero $\ker \lambda=
 \{a\in A : \lambda(a)=0\}$, while keeping a vanishing Dirac operator,
 appears originally in \cite[denoted by $\cC_0$ there]{MvS}, in terms of which the object-set of
 our category reads $\ppS=\{(A,\lambda,H) \in \mathcal{S}_0 : \ker
 \lambda =0\}$.  We also comment that it is not usual to call $\lambda
 $ `action' but `representation'. Our terminology tries to prevent
 confusion thereafter, when we will treat representations of quivers.
\end{remark}

\subsection{Characterisation of morphisms}
We characterise morphisms in two steps.

\subsubsection{Involutive algebra morphisms}

We examine unital $*$-algebra morphisms, the first layer of $\hom_\ppS$.
We recall a  well-known fact in the next example in order to introduce some notation.

 \begin{example}\label{ex:expandingMaps}
Notice that on $*$-algebra maps $\phi: M_{m} (\C) \to M_{n}(\C)$ is
heavily constrained: due to $\phi$ being unital ($1_m\mapsto 1_n$) and
linear ($0_m\mapsto 0_n)$ it cannot be a constant, and since $\phi$
algebra morphism, it cannot reduce dimension, so $m\leq n$.  If $n=m$
the only map $M_m(\C) \to M_m(\C)$ is the identity, up to conjugation
$\Adj u$ by a unitarity $u\in \uni(n)$, \begin{align*} \phi =\Adj u
  \circ \includegraphicsc{.48}{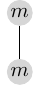} \nonumber
 \end{align*}
% \itemb
When $m<n$, we cannot use projections
of $a$ to `fill' the image, as again this would imply that $\phi$ is not an
algebra map.  Then maps $\phi$ exist only when $n$ is a
multiple $k$ of $m$. For example, $\phi(a)= a \oplus \ldots \oplus a$
($k$ times $a$ in block-diagonal structure), which we identify with
$\phi(a)=1_k \otimes a$ and represent by
\[
\bigg\downarrow\,\,\,
\includegraphicsc{.50}{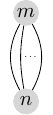} \text{with a $k$-fold line.  In full
generality, }  \phi_u  =1_k \otimes \Adj u ( \bullet ) = \Adj u \circ
\includegraphicsc{.50}{k-fold2zolty}  \text{ with }u\in \uni(n) . \]
%  \end{itemize}
 \end{example}
This motivates a diagrammatic representation of $*$-algebras due to
Ola Bratteli \cite{BratteliDiags}. Before describing it, we comment
that his diagrams were an important tool in the classification of
{approximately finite-dimensional algebras}\footnote{Approximately
finite-algebras are C$^*$-algebras that are direct limits of
finite-dimensional ones.} based on Elliott's $K_0$-based construction.
For spectral triples with finite-dimensional algebras (corresponding
to manifolds of $0$-dimensions) the classification is known and due to
Krajewski \cite{Krajewski:classif} (whence `Krajewski diagrams') and
Paschke-Sitarz \cite{PaschkeSitarz}.  We mention parenthetically that
AF-algebras are used in \cite{Masson:2022ppv} to lift those Krajewski
diagrams. Our prespectral triples are from the onset
finite-dimensional not as an approximation to infinite dimensional
ones (also not to avoid technical clutter) but by the
finite-dimensionality of the (physical) gauge group.

We allow ourselves
a certain abuse of notation and label
objects with $s$ and $t$ for the rest of this section
 (although later on, in a quiver context, they are no longer labels but maps).
Let $l_s,l_t \in \Z_{>0}$ throughout.
\begin{definition}\label{def:bratteli}
Given $\mathbf m \in \Z^{l_s}_{>0}$ and $\mathbf n \in \Z^{l_t}_{>0}$,
a \textit{Bratteli diagram compatible with} $\bm$ and $\bn$, which we
denote by $\Brat:\mathbf m\to \mathbf n$, is a finite oriented graph
$\Brat=(\Brat_0,\Brat_1)$ that is vertex-bipartite
$\Brat_0=\Brat_0^s\dot\cup \Brat_0^t$ (edges start at $\Brat_0^s$ can
connect end only at vertices of $\Brat_0^t$) and such that
\begin{enumerate}
 \item the vertex-set satisfies
$\#\Brat_0^s =  l_s $ and $\#\Brat_0^t = l_t$. This allows
us to label the vertices  $i\in\Brat_0^s$ by $i\mapsto m_i$ and
$j\in\Brat_0^t$ by $j\mapsto n_j$, and
\item denoting by $C_{k,k'}$ the number of edges
between vertices $k\in \Brat_0^s, k'\in \Brat_0^t$, the second condition reads
\begin{align}
\label{linkingC}
n_j= \sum_{i\in \Brat_0^s }  C_{i,j}  m_i\,.
\end{align}
\end{enumerate}
\end{definition}

\begin{remark}\label{rem:Ccolumns_nonzero}
 Due to bipartiteneness, $C_{i,i'}=0$ for $i,i'\in \Brat_0^s$ and
 $C_{j,j'}=0$ and for $j, j'\in \Brat^t_0$, so the adjacency matrix
 of a Bratteli diagram has the form
\begin{align} \bigg( \begin{matrix} 0_{l_s} &  C\\[.3ex]
 C\trans & 0_{l_t}
 \end{matrix} \bigg)  \label{biadjacencyfirst}\end{align}
 where $C$ is referred to as biadjacency matrix. Due to these zeroes,
 actually the sum in condition \eqref{linkingC} can run over all
 vertices.  Observe that for fixed $j\in \Brat^t_0$, courtesy of
 \eqref{linkingC}, $\sum_i C_{i,j}>0$ holds, since otherwise $n_j $
 vanishes. Hence any vertex in $\Brat^t_0$ has non-zero valence.
\end{remark}

Conventionally, we place $\Brat_0^s=\{1,\ldots, l_s\} $ on the top row
and   $\Brat_0^t= \{1,\ldots, l_t\}$ on the bottom one, so that a
generic Bratteli diagram $\Brat: \mathbf m \to \mathbf n$ looks like
\begin{align*}\nonumber %\label{BratteliTyppical}
\Brat=
\raisebox{-.46\height}{\includegraphics[width=6.17cm]{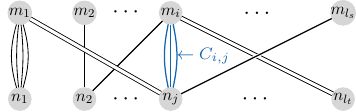}} \qquad \Bigg \downarrow
\end{align*}
\begin{lemma}\label{thm:Bratteli=AlgHoms}
If $*$-$\mathsf{alg}$ denotes the category of unital, involutive
algebras, for $\mathbf{m}\in \Z_{> 0 }^{l_s}$ and $\mathbf n\in \Z_{>
  0 }^{l_t}$,
\begin{align}\hom_{\text{$*$-$\mathsf{alg}$}} \big( \oplus_{i=1}^{l_s}  \M {m_i},\,  \oplus_{j=1}^{l_t} \M{n_i} \big) \simeq
\coprod_{\substack{\text{Bratteli diagrams} \\ \Brat\,:\, \mathbf m
    \to \mathbf n }} \mathcal U \big (\oplus_{i=1}^{l_t} \M {n_i} \big
)\,. \raisetag{1.2cm}\label{homStarAlg}\end{align}
\end{lemma}\begin{proof} (Inspired by \cite{MvS}.)
  Given a $*$-algebra morphism $\phi: A_s\to A_t$ let us associate to
  $\phi$ a Bratteli diagram.  Following Ex.~\ref{ex:expandingMaps},
  the restriction of $\phi$ to the $i$-th summand of $A_s$, $\phi|_{\M
    {m_i}}: {\M {m_i}} \to A_t$ can be seen to be given block
  embeddings. Write a simple edge from the upper $i$-th vertex to the
  bottom $j$-th vertex for each block in the block-embedding of $\M
  {m_i}$ into $\M {n_j}$. Being $\phi$ unital, for fixed $j$, the
  number $C_{i,j}$ of edges incident to the $j$-th bottom node should
  satisfy $\sum_i C_{i,j} m_i=n_j$, so Condition \eqref{linkingC} is
  satisfied and we have compatibility. This constructs a map $(\phi:
  A_s \to A_t) \mapsto (\Brat(\phi):\bm \to \bn)$ that is invertible
  up to unitarities in $A_s$ and $A_t$ (the inverse map $\Brat \mapsto
  \phi_\Brat$ is constructed from a Bratteli diagram $\Brat$ by block
  embeddings as dictated by its edges).  Such unitarities $u_s\in
  \mathcal U(A_s)$ and $u_t\in \mathcal U(A_t)$ act both by
  conjugation, $\phi_\Brat \mapsto \Adj {u_t} \circ \phi_\Brat \circ
  \Adj {u_s} $.  But $\phi_\Brat$ is a $*$-algebra map so, first,
  $\phi_\Brat(u_s a u_s^*)= \phi_\Brat(u_s)
  \phi_\Brat(a)\phi_\Brat(u_s)^*$ and, secondly, it preserves
  unitariness, i.e. $\phi_\Brat(u_s)\in \mathcal U(A_t)$.  This means
  that $\phi_\Brat(u_s)$ only shifts the action of $ \mathcal U(A_t)
  \ni u_t$ by $\Adj ( u_t) \mapsto \Adj[ u_t \phi_\Brat(u_s)]$ and no
  new maps are gained from $ \mathcal U(A_s)$. The result
  \eqref{homStarAlg} follows by spelling $A_t$ inside $\mathcal
  U(A_t)$ out as a direct sum of matrix algebras.
 % That the bottom
%  $j$-th vertex receives always
% with $1\leq j \leq l_t$
% \item a map $L$ that should satisfy $\Ad L \lambda_s (a) = \lambda_t [\phi(a)] $ for
% each $a\in A_s$. \qe
% \end{itemize}
% This must satisfy $\mathbf q= C \cdot \mathbf r \emph{ and } \mathbf n= \mathbf m  \cdot  C $
\end{proof}
\begin{notation}\label{notation_BrattsUnitarities}
The previous lemma allows us to represent morphisms $\phi:A_s\to A_t$
of involutive algebras by a pair $(\Brat, u)$ consisting of a Bratteli
diagram $\Brat=\Brat(\phi)$ and a unitarity $u\in \mathcal U
(A_t)$. The original morphism is reconstructed by $\phi(a)= u
\phi_\Brat (a) u^* $, $a\in A_s$. With some abuse of notation we
abbreviate $\phi_\Brat$ as $\Brat$, so $\phi=\Adj u \circ \Brat$. An
example of what one means by the map $\Brat $ is
\begin{align*} \nonumber \Brat=
 \raisebox{-.61cm}{\includegraphics{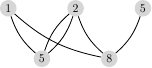}
 } \qquad \Brat (z,a,b ) = \bigg(\begin{matrix}
 \diag(z,a,a) & 0 \\ 0 &                  \diag(z,a,b )\,
                             \end{matrix}
 \bigg)\,,\end{align*} for $z\in \C, a\in \M 2, b\in \M 5$. Strictly,
we should also specify at bottom vertices a linear ordering, as to
distinguish, say $\diag(z,a,a)$ from $\diag(a,z,a)$, for the left
bottom vertex. One can take the convention that the arguments that
takes a bottom vertex appear in non decreasing matrix size (as above,
where $z$ is at the left of $a$ for being a smaller matrix). However,
any other convention would differ from this by a permutation matrix
$P$ (which can be seen as `shifting' $ \mathcal U (A_t )$ by $P$,
which is also a unitary).
\end{notation}

\subsubsection{Hilbert space compatible morphisms}
We now see the consequences of adding the inner product spaces.
\begin{lemma}[Consequence of Artin-Wedderburn Theorem]\label{thm:WedderburnArtinConsequence}
For $X \in \ppS$ there is an integer  $l >0$ and tuples of
non-negative integers%
% \begin{subequations}%
  \begin{align}\label{integer_tuples}
%   \mathbf m&=(m_1,\ldots,m_{l_s}),  & && \mathbf q&=(q_1,\ldots,q_{l_s}) \,, \\
 \,\,\,\mathbf n&=(n_1,\ldots,n_{l}), &\mathbf
 r&=(r_1,\ldots,r_{l})\,,\end{align}%
% \end{subequations}%
unique up to reordering, such that
% \begin{subequations}
\begin{align}
% A_s &= \bigoplus_{v=1,\ldots, l_s}  M_{m_v }  (\C) , & H_s &= \bigoplus_{v=1,\ldots, l_s} q_v \C^{ m_v } , \\
\qquad
 A&= \bigoplus_{w  =1,\ldots, l}  M_{n_w} (\C)  &  H & = \bigoplus_{w=1,\ldots, l }   r_w  \C^{n_w}  \,.\label{A_summands_H_summands}
\end{align}%
% \end{subequations}%
\end{lemma}

\begin{proof}
Artin-Wedderburn theorem guarantees for the algebras of the
prespectral triple $X $ the decomposition into matrix algebras, say
with $l$ direct summands, as in the equation of the left
\eqref{A_summands_H_summands}. Since the action $\lambda$ of these on
inner product space $H $ has to be faithful, the tuple $\mathbf r$ in
\eqref{integer_tuples} consist of non-negative integers.
\end{proof}

\begin{example}\label{ex:Aut(H)} Consider $\M n$
 with an $r$-fold action \begin{align*}\lambda(a)=1_r \otimes a
   =\diag(a,\ldots, a) \qquad\text{($a$ appearing $r$-times in the
     diagonal)}. \notag \end{align*} Let $X=(\M n, \lambda, r \C^n)$
 where $r \C^n:=(\C^n\oplus \cdots \oplus \C^n)$ has $r$-direct
 summands.  This notation allows one to drop $\lambda$ and simplify as
 $X=(\M n, r \C^n)$.  We determine now $\End_\ppS(X)=\hom_\ppS (X,X)
 $. By Ex.~\ref{ex:expandingMaps}, $\phi(a)=\Adj u (a)= u a u^*$ for
 some $u\in \mathcal U[ \M n ] =\uni (n) $. The compatibility
 condition for $\phi$ and $L \in \mathcal U ( r \C^n) =\uni(r\cdot n)$
 reads $ \lambda [\phi(a)]= L \lambda (a ) L^* $ $\text{ or,
   equivalently, }$ $( 1_r \otimes u a u^* ) = L ( 1_r \otimes a ) L^*
 $.  This determines $L=u'\otimes u$ for a $u'\in \uni(r)$,
 yielding \begin{align*} \End_\ppS( \M n , r \C^n) = \{ (\phi,L) :
   \phi=\Adj u , L= u' \otimes u \} \simeq \uni(r) \times \uni(n)
   \,.\end{align*}
% \rojo{up to permutation matrix}
 \end{example}

% % naechste Def. ist nicht brauchbar / no es 'util / def.    nie okazała się pożyteczną...

 \begin{proposition}[Characterisation of $\hom_\ppS$]\label{thm:characHomFerm}
For prespectral triples $X_s=X_{s}(\mathbf m,\mathbf q ) $ and
$X_t=X_{t}(\mathbf n,\mathbf r)$ parametrised as in Lemma
\ref{thm:WedderburnArtinConsequence} by $\bm,\bq \in \Zpos^{l_s}$ and
$\bn ,\br \in \Zpos^{l_t}$ one has
\begin{align}\label{homXsXt}
\hom_\ppS (X_{s},X_{t}) & \simeq \coprod_{ \Brat:  (\mathbf m,\mathbf q)\to (\mathbf n,\mathbf r)}
 \Big [ \textstyle\prod_{j=1}^{l_t}  \uni(r_j ) \times \uni(n_j )
%  \times \uni( r_j)
%  \big ]
 \Big] \,,
\end{align}
and the set of compatible Bratteli diagrams (further
constrained after one adds the layer of inner spaces) that appear below $\amalg$ is explicitly
\begin{align*} \notag
\{\Brat: (\mathbf m,\mathbf q)\to (\mathbf n,\mathbf r)\} =
 \Bigg\{ C \in M_{l_s\times l_t} (\Z_{\geq 0 })  : \bigg( \begin{matrix}
\mathbf n \\[.3ex] \mathbf q
 \end{matrix} \bigg) =  \bigg( \begin{matrix}
  C\trans & 0 \\[.3ex]
 0 & C
 \end{matrix} \bigg)
  \bigg( \begin{matrix}
    \mathbf m       \\[.3ex]
    \mathbf r
 \end{matrix} \bigg)
 \Bigg\}
\end{align*}
In particular, the biadjacency matrix $C$ has neither zero-columns nor
zero-rows.

\end{proposition}
See Lemma \ref{thm:WedderburnArtinConsequence} for the dependence on
$\bm,\bn,\bq,\br$ and notice that in the condition \eqref{homXsXt} the
tuples characterising objects are `crossed' (i.e. $\bq$ is a tuple
concerning $H_s$ and $\bn$ a tuple for $A_t$, and these appear in the
\textsc{lhs}).

\begin{proof}
By Lemma \ref{thm:Bratteli=AlgHoms}, a morphism $(\phi,L):
X_{s}(\mathbf m,\mathbf q ) \to X_{t}(\mathbf n,\mathbf r)$
determines a Bratteli diagram $\Brat$.

Since $\phi$ is a
$*$-algebra morphism it embeds blocks of $m_1\times m_1$-matrices,
\ldots, $m_{l_s}\times m_{l_s}$-matrices into $n_{1}\times
n_{1}$-matrices, $\ldots$, $n_{l_t}\times n_{l_t}$
(cf. Ex.~\ref{ex:expandingMaps}) by following the lines of the Bratteli
diagram. Matching dimensions yields
\begin{align} \label{costamniewaznego}
n_j= \sum_{i} C_{i,j } m_i\qquad \text{ that is } \qquad
  \mathbf n=   C\trans \cdot \mathbf m\, \text{ with $ \sum_v C_{v,w}>0$}.
\end{align}
\noindent
 \begin{minipage}{.48005\textwidth}
Also $\phi=\Adj u\circ \Brat $ for some $u\in \mathcal U(A_t)$,
according to Lemma \ref{thm:Bratteli=AlgHoms}.  To obtain the
remaining condition satisfied by $C$, we add the restriction coming
from the layer of the Hilbert spaces.  Each node labelled by $n$ in
the diagram \eqref{BratteliwHilberts} represents the algebra $\M n$
that acts on the vector space it has above or below it.  From the
compatibility condition $\lambda_t [ \phi(\mathbf a) ] = L \lambda_s
(\mathbf a) L^*$ for any $\mathbf a\in A_s=\oplus_v \M{m_v}$, which
clearly requires the dimension of $H_s$ and $H_t$ to coincide, one
gets $\sum_ iq _i m_i = \sum_j r_j n_j = \sum_{i,j} r_j C_{i,j}
m_i$---but more is true.
\end{minipage}\quad
\begin{minipage}{.4971\textwidth} \centering
% \begin{figure}
 \begin{align} \label{BratteliwHilberts}
\Brat= \!\!\!\raisebox{-.46\height}{\includegraphics[width=5.606cm]{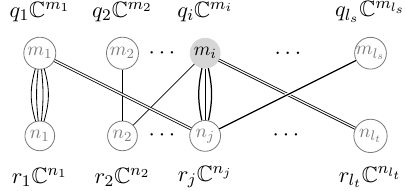}}
\end{align}

\vspace{2ex}
\begin{align} \label{comm_diag_ith}
 {\begin{tikzcd}[ampersand replacement=\&, column sep= 1.98cm]
 H_s   \arrow[]{r}{\lambda_s(\iota_i(a))}  \&  H_s   \arrow[]{d}{L}\\
 H_t   \arrow[]{u}{L^*}   \arrow[swap]{r}{  \text{\raisebox{-2ex}{$\lambda_t [\phi(\iota_i(a)) ] $}}  }  \& H_t
\end{tikzcd}}\end{align}
\end{minipage} \\

\noindent The diagram \eqref{comm_diag_ith} is obtained from the
compatibility condition for $\iota_i(a)=(0,\ldots, 0, a,0,\ldots, 0)$,
with $a\in \M{m_i}$ appearing in the $i$-th entry.  Then the traces of
the two possible maps $H_t\to H_t$ in diagram \eqref{comm_diag_ith}
coincide, and so do the traces of the horizontal maps, thanks to the
unitarity of $L$. Further, since $u$ is unitary, and $\lambda_t$ is a
$*$-action, $\Adj [\lambda_t(u) ]$ does not affect the trace
either. Picking $a\in \M{m_i} $ with non-vanishing trace and following
the edges in Eq.~\eqref{BratteliwHilberts} one obtains the second
condition $q_i=\sum_{j=1,\ldots,l_t} C_{i,j} r_j$, or $\mathbf q= C
\cdot \mathbf r$. \par It remains to see how much freedom does $L$
still contain.  Spelling out the compatibility condition and using
$\Brat=\Adj u^* \circ \phi$ for $u\in \mathcal U (A_t)$, one obtains
\begin{align}
\label{explicit_compa_inproof}\lambda_t [ \Brat (\mathbf a) ]  = \lambda_t (u )^*  L  \lambda_s( \mathbf a)  L^*   \lambda_t (u )
\text { for each } \mathbf a \in \oplus_v \M{m_v} \,.\end{align} [If
  the way to place the star seems odd, recall the relation right after
  \eqref{costamniewaznego}].  Since each matrix $a_i$ in
$\ba=(a_1,\ldots,a_{l_s})\in A_s$ should appear in blocks the same
number of times in the \textsc{rhs} than in the \textsc{lhs} of
\eqref{explicit_compa_inproof}, $ \lambda_t (u )^* L$ is a matrix
permutation $P_\pi$ for some $\pi\in \Sym( \dim H_t)$, up to an
abelian phase $\ee^{i\theta}$.  (One can determine $\pi$ in terms of
the integer parameters and of $\Brat$ as later in
Ex.~\ref{ex:permmatrix}, but the essence of the argument is that it
does not depend on anything else.)  This means that $L=\lambda_t(u)
\ee^{\ii \theta} P_\pi$, where, however, $P_\pi$ is far from unique.
For once a certain $L$ satisfies \eqref{explicit_compa_inproof}, so
does $L$ acted on by the unitarities via
\begin{align} \label{actionOnL}
L \mapsto   u_t ^* \cdot L\cdot u_s \qquad u_s \in \textstyle
\prod_{i=1}^{l_s} \uni(q_i)  ,
\,u_t \in\textstyle \prod_{j=1}^{l_t} \uni(r_i)\,,
\end{align}
where the star on $u_t$ is purely conventional.  The action of
$u_s=(u_{s,i})_{i=1\ldots,l_s}$ (with $u_{s,i}$ a unitarity matrix of
size $q_i$) is on the blocks of the form $ 1_{q_i} \otimes a_i $. But
since $L$ appears acting on $\lambda_s(\ba)$ by the adjoint action
letting the unitarities $u_s=(u_{s,i})_{i=1\ldots,l_s}$ act on $L$ as
in \eqref{actionOnL} for each $i$, yields a trivial action, namely $
u_i u_i^* \otimes a_i = 1_{q_i} \otimes a_i $. Hence the only
information $L$ retains comes from $u$ and $u_t$ (any of which can
absorb the abelian phase). \end{proof}%

\noindent
A similar statement to the previous one, with the possibility of
$\lambda$ being a non-faithful action, is \cite[Prop. 9]{MvS}. In that
sense, \textit{op. cit.} is more general and inspired our proposition.
However, thanks to the explicit action \eqref{actionOnL} in our proof,
we observe already a possible further reduction of the group
$\prod_j [ \uni
  (r_j)\times \uni (n_j) ] $ reported there;
  see Proposition \ref{thm:reductionU}.

\begin{definition}\label{def:reductionPU}
With the notation of Proposition \ref{thm:characHomFerm} let
\begin{align*}\hom_\pS  (X_{s},X_{t}) :=  \hom_{\ppS} (X_s,X_t) / \sim\,.  \end{align*}
For given $X_s=(A_s,\lambda_s, H_s), X_t=(A_t, \lambda_t, H_t)$, we
define two morphisms $(\phi_i, L_i):X_s\to X_t$ to be equivalent,
$(\phi_1, L_1) \sim (\phi_2, L_2) $, if the algebra-maps agree, $
\phi_1=\phi_2:=\phi$, and if so, further, if also $L_1$ and $L_2$
satisfy
\begin{align}\label{equivRel}
 L_1 \lambda_s (a) L_1^* =
  \lambda_t  [\phi( a) ] = L_2 \lambda_s (a) L_2^*,\quad
 \text { for all $a \in A_s $}.
\end{align}
With some abuse of notation we still write $(\phi,L)$ instead of the
correct but heavier $(\phi, [L])$ to denote the morphisms of
$\hom_\pS$.
\end{definition}

\begin{proposition}\label{thm:reductionU}
With the notation of Proposition
\ref{thm:characHomFerm}, one has the following characterisation:
\begin{align}\label{nonotriv homXsXt}
\hom_\pS (X_{s},X_{t}) & \simeq \coprod_{ \Brat:  (\mathbf m,\mathbf q)\to (\mathbf n,\mathbf r)}
  \textstyle\uni (\bn) \,, \text{ where } \uni (\bn):=\prod_{j=1}^{l_t} \uni(n_j )\,.
\end{align}
\end{proposition}
\begin{proof} It has been shown that  $\ppS$-morphisms
are indexed by compatible Bratteli diagrams. These are unaltered when
we reduce by $\sim$ in Eq.~\eqref{equivRel}. It remains to see how
such relation reduces $\textstyle\prod_{j=1}^{l_t} \uni(n_j ) \times
\uni(r_j )$ to the unitary groups as claimed. \par To prove the
triviality of the action of $\uni(\br)= \prod_{j=1}^{l_t} \uni(r_j )$
assume that $(\phi,L)$ is a $\ppS$-morphism $X_s\to X_t$, so
$\lambda_t[\phi(\ba)]=L \lambda_s(\mathbf a) L^*$ holds for all $\ba
\in A_s$. By definition of $\lambda_t$, $\lambda_t[\phi(\ba)]$
consists of matrix blocks of the form $1_{r_i} \otimes b_i $ where
each $b_j \in \M{n_j}$ is function of $\ba \in A_s$.  Now we let $
u_t=(u_{t,j})_{j=1\ldots,l_t} \in \uni(\br)$ act as in
\eqref{actionOnL} on $L$, and call $L' = u_t^* \cdot L $ the
result. The compatibility condition transforms as
\[
L' \vast[
 \begin{smallmatrix}
    1_{r_1} \otimes b_1 & 0   & \cdots  & 0 \\[.5ex]
     0  &1_{r_2} \otimes b_2  & 0       & 0 \\
     \vdots   & \ddots  & \ddots       & \vdots \\[1ex]
     0 &  \cdots    & 0 &\,\, 1_{r_{l_t}} \otimes b_{l_t}
   \end{smallmatrix}\vast]
 (L')^* =
 L \vast[
 \begin{smallmatrix}
    u_{t,1}^*u_{t,1}^{\phantom *} \otimes b_1 & 0   & \cdots   & 0 \\[.5ex]
     0  & u_{t,2}^*u_{t,2}^{\phantom *} \otimes b_2  &  \cdots & 0 \\
     \vdots   & \ddots  & \ddots      & \vdots \\[1ex]
     0 & \cdots    & 0 & \quad u_{t,l_t}^*u_{t,l_t} ^{\phantom *}\otimes b_{l_t}
   \end{smallmatrix}\vast]
 L^*
\]
which means that Eq.~\eqref{equivRel} is verified for $L'$. It is now
obvious that the action by $L$ is fully determined by $\Brat$ and by
the unitary group $\prod_{i=1\ldots,l_t} \uni(n_i)$.
\end{proof}
Therefore we can instead of parametrising each $\pS$-morphism
$\Phi=X_s\to X_t$ by the variables of the definition $\Phi=(\phi,L)$
it is possible to give for equivalently describe them by Bratteli
diagrams (or their biadjacency matrices) and unitarities of the target
algebra. As a slogan (whose concrete meaning is delivered by
Proposition \ref{thm:reductionU}),
 \begin{align}\text{``}\Phi=(\Brat,U) \in \{ \text{Bratteli diagrams } (  \mathbf m,\mathbf q)\to (\mathbf n,\mathbf r)\} \times
\uni(\bn )
\text{''}\,. \label{ParametrisationbyBratteliUnit}\end{align} (Observe
 this last group's argument coincides with a datum of $\Brat$.)  We
 remark that the unitary group could be further be reduced
 to\footnote{The projective group $\mathrm{PU}(m) $ is the unitary
 group modded out by its center: $\mathrm{PU}(m)=\uni(m)/Z(\uni(m))
 =\uni(m)/ \uni(1)$.} $\prod_{i=1\ldots,l_t} \mathrm{PU}(n_i)$, since
 $L$ appears only through its adjoint action.

\begin{example}[Illustrating notation in the proofs above]
Consider as input the diagram $\Brat$ of \eqref {example_illustrative},
with labels for $H_s$ in the above row still to be determined.
  \begin{align} \label{example_illustrative}
    \Brat&=
 \raisebox{-1.091cm}{ \includegraphics[height=2.3cm]{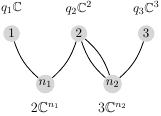}}
  \quad  C=\Big(\begin{smallmatrix}
  1& 0 \\[-.00041ex]
  1 & 2 \\[-.00041ex]
  0 & 1
   \end{smallmatrix}\Big)
 \quad
 \bm=  (1,2,3)\trans %\\[0ex]
 \quad
\br =  (2,3)\trans \end{align}
By Prop.~\ref{thm:characHomFerm} morphisms $(A_s,H_s) \to
(A_t,H_t)$ exists only if $\bq=C \cdot \br$ and $\bn=C\trans \cdot \bm$,
whose only solution is $\bq= (2,8,3)\trans $ and $\bn=(3,7)\trans$.
Any $\phi:A_s \to A_t$ is of the form $\phi(\ba) =u \Brat(\ba)u^*$,
$u\in \mathcal U(A_t)$ and $\ba=(z,a,b)\in A_s=\C\oplus \M 2 \oplus \M
3$. Written down in full,
\begin{align*}
\lambda_s (z,a,b)=\bigg(\begin{smallmatrix}
                   z1_2 & 0 & 0 \\
                   0 & 1_8 \otimes a& 0\\
                   0 &0 &  1_3\otimes b
                  \end{smallmatrix}\bigg)
\qquad
\lambda_t (a'_1,a'_2)=\Big(\begin{smallmatrix}
                    1_2 \otimes  a_1' & 0  \\
                   0 & 1_3 \otimes a_2' \\
                  \end{smallmatrix}\Big)
                  \qquad
\Brat(\ba)=\Bigg[ \begin{matrix}
                   \begin{smallmatrix}
                    z & 0 \\
                    0 & a
                   \end{smallmatrix} \!\!\!\! & \boldsymbol{0}  \\
                   \boldsymbol{0} & \!\!\! \begin{smallmatrix}
                    a & 0 & 0 \\
                    0 & a & 0 \\
                    0& 0& b
                   \end{smallmatrix} \\
                  \end{matrix}\Bigg]
\end{align*}
Now we focus on the unitarity $L:\C^2\oplus2\C^8\oplus 3\C^3\to
2\C^3\oplus 3\C^7$.  Suppose that  $\pi$ is any permutation such that the matrix
$P_\pi$ satisfies (for explicit computation of such $\pi$ see
Ex.~\ref{ex:permmatrix})
\begin{align}
P_\pi \diag ( 1_2  \otimes z ,1_8 \otimes a, 1_3 \otimes b) P_\pi^* = \diag\Big [  1_2 \otimes \big (\begin{smallmatrix}
                                                                z& 0\\
                                                                0 & a
                                                               \end{smallmatrix}\big) ,
                                                               1_3 \otimes \Big (\begin{smallmatrix}
                                                                a& 0 & 0 \\
                                                                0 & a & 0 \\
                                                                0 & 0 & b
                                                               \end{smallmatrix}\Big)
                                                               \Big]\,. \label{conjugationbyPpi}
\end{align}
Then any $L$ satisfying the compatibility condition, $\Adj L \lambda_s
(\ba)= \lambda _t ( \Adj u \Brat(\ba))$ is of the form $L=\exp(\ii
\theta) \lambda_t(u) ( u_t \cdot P_\pi \cdot u_s)$ where $u_s= \diag (
u_{s,1} ,1_2\otimes u_{s,2},1_3\otimes u_{s,3} )$ with $u_{s,i}$ in
the $i$-th summand of $\mathcal U(A_s)$, and $u_t= \diag (1_2\otimes
u_{t,1} ,1_3\otimes u_{t,2})$, $u_{t,i}$ in the $i$-th summand of
$\mathcal U(A_t)$ and $\exp(\ii\theta)\in \uni(1)$.  However, by
substituting $P_\pi$ by $u_t \cdot P_\pi \cdot u_s$ in
\eqref{conjugationbyPpi}, these actions by $\mathcal U(A_s)$ and
$\mathcal U(A_t)$, as well as the abelian phase $\exp(\ii \theta)$,
are seen to be trivial, and since $L$ appears acting by $\Adj L$, do
not modify $L$ in $\pS$ (although they do in $\ppS$). Alternatively,
if rather $\{\bn,\bm,\bq,\br\}$ is the input, many $\Brat: \bm\to \bn$
exist, like
\begin{align*}
 \runter{\includegraphics{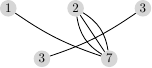} \quad\qquad  \includegraphics{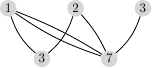}\quad \qquad  \includegraphics{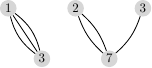}}\quad\ldots
\end{align*}
but only the $\Brat$ from \eqref{example_illustrative} is
compatible also with the Hilbert spaces, so
\begin{align*}
\hom_{\pS} [  X_s(\bm,\bq) ,X_t(\bn,\br)]  = \uni(3)\times \uni(7)\,.
\end{align*}
\end{example}
% \end{minipage}\quad
% \begin{minipage}{.35\textwidth}\centering

% \end{minipage}

\begin{example}(How $\Brat$ defines the permutation $P_\pi$.) \label{ex:permmatrix}
Consider the Bratteli diagram
\begin{align*} \Brat=
 \raisebox{-.551cm}{\includegraphics{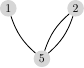}}\end{align*}
 \begin{minipage}{.60\textwidth}
with actions $\lambda_s$ and $\lambda_t$ on $H_s=3\,\C \oplus 6\,\C^2$
and $H_t=3\, \C^5$ and each $\C^i$ being acted on by $\M i$ in the
natural way.  The compatibility map $L:H_s \to H_t$ is then given by a
permutation matrix (denoted by $P_\pi$) that implements the
identification $H_s\cong H_t$ and is explicitly given by the
(certainly unitary) permutation-matrix $P_\pi$ associated to
permutation $\pi=(2,6,4)(3,11,10,9,8,7,5)$, shown on the right.
 \end{minipage}\quad
 \begin{minipage}{.373\textwidth}
 $ P_\pi=\left[
\begin{smallmatrix}
 1 & 0 & 0 & 0 & 0 & 0 & 0 & 0 & 0 & 0 & 0 & 0 & 0 & 0 & 0 \\
0 & 0 & 0 & 1 & 0 & 0 & 0 & 0 & 0 & 0 & 0 & 0 & 0 & 0 & 0 \\
0 & 0 & 0 & 0 & 1 & 0 & 0 & 0 & 0 & 0 & 0 & 0 & 0 & 0 & 0 \\
0 & 0 & 0 & 0 & 0 & 1 & 0 & 0 & 0 & 0 & 0 & 0 & 0 & 0 & 0 \\
0 & 0 & 0 & 0 & 0 & 0 & 1 & 0 & 0 & 0 & 0 & 0 & 0 & 0 & 0 \\
0 & 1 & 0 & 0 & 0 & 0 & 0 & 0 & 0 & 0 & 0 & 0 & 0 & 0 & 0 \\
0 & 0 & 0 & 0 & 0 & 0 & 0 & 1 & 0 & 0 & 0 & 0 & 0 & 0 & 0 \\
0 & 0 & 0 & 0 & 0 & 0 & 0 & 0 & 1 & 0 & 0 & 0 & 0 & 0 & 0 \\
0 & 0 & 0 & 0 & 0 & 0 & 0 & 0 & 0 & 1 & 0 & 0 & 0 & 0 & 0 \\
0 & 0 & 0 & 0 & 0 & 0 & 0 & 0 & 0 & 0 & 1 & 0 & 0 & 0 & 0 \\
0 & 0 & 1 & 0 & 0 & 0 & 0 & 0 & 0 & 0 & 0 & 0 & 0 & 0 & 0 \\
0 & 0 & 0 & 0 & 0 & 0 & 0 & 0 & 0 & 0 & 0 & 1 & 0 & 0 & 0 \\
0 & 0 & 0 & 0 & 0 & 0 & 0 & 0 & 0 & 0 & 0 & 0 & 1 & 0 & 0 \\
0 & 0 & 0 & 0 & 0 & 0 & 0 & 0 & 0 & 0 & 0 & 0 & 0 & 1 & 0 \\
0 & 0 & 0 & 0 & 0 & 0 & 0 & 0 & 0 & 0 & 0 & 0 & 0 & 0 & 1
\end{smallmatrix}\right]$
 \end{minipage}
\vspace{1ex}
 \par
\end{example}

\begin{corollary}[Automorphisms] \label{coro:Automorphisms}
For $X\in \pS$ parametrised as $X=X(\bn,\br)$ by $\bn,\br\in \Zpos^l$,
\begin{align*}
\Aut_\pS  (X)= \coprod_{\sigma \in \Sym(\bn,\br)}  \Big\{\textstyle\prod_{j=1}^l \uni(n_j) \Big\}\,,
\end{align*}
meant as an
equality of sets (cf. group structure later),
 where $\Aut_\pS  (X)$ are the invertible elements of
 $\hom_\pS(X,X) $ and \begin{align*}\Sym(\bn,\br)=\{ \sigma \in
   \Sym(l) : n_{\sigma(i)} =n_i \text{ and } r_i=r_{\sigma(i)} ,
   i=1,\ldots,l\} \,.\end{align*} That is, $(ij) \in \Sym(l)$ is in $\Sym(\bn,\br)$ only if both $n_i=n_j$ and
 $r_i=r_j$.
\end{corollary}

\begin{proof} By Proposition \ref{thm:characHomFerm}, a compatible Bratteli diagram
$ \Brat:(\bn,\br)\to (\bn,\br)$ has a biadjacency matrix $C$ such that
  $\bn=C\trans \cdot \bn$ and $\br=C\cdot \br$. Since all entries of
  $\br$ and $\bn$ are positive, $C$ cannot have nonzero rows or
  nonzero columns.  Having any entry in a column is larger than $1$
  yields a sum of entries $\sum_{i,j} C_{i,j} n_i > \sum_j n_j$,
  contradicting $\bn=C\trans \cdot \bn$. Similarly the nonzero entries
  in rows of $C$ must be $1$. Thus $C$ is orthogonal, which, being
  nonnegative, is equivalent to be a permutation matrix $C=P_\sigma$
  for $\sigma\in\Sym(l)$ that respects the matrix size $n_{\sigma(i)}
  =n_i$ and degeneracy $r_{\sigma(i)} =r_i$.
\end{proof}

\begin{example}[Automorphism]
If ${\bn \choose \br}= {2\,2\,4\,4\,5\,5\,5\,5 \choose
  1\,2\,2\,2\,1\,1\,1\,3 }$ then $\Sym(\bn,\br )= \Sym(2)\times
\Sym(3)$ and $X=X(\bn,\br)$ has a (set underlying to the) group
$\Aut_\pS (X)$ consisting of 12 copies of $\uni(2)^2\times
\uni(4)^2\times \uni(5)^4$, since 12 Bratteli diagrams exist for the
given data.
\end{example}

\section{Quiver representations on prespectral triples}\label{sec:rep_theo}

%  \subsection{Quivers and paths on quivers}
Before addressing the main point of this section, representation
theory, in Sec.~ \ref{sec:Q-reps} we introduce quivers in the next.
Throughout this section, $B$ will be a finite-dimensional, unital
$*$-algebra.
%  \\ \\
 \subsection{Quivers weighted by operators}\label{sec:weighted_quivers}
% \begin{definition}
A \textit{quiver} is a directed graph $Q= ( Q_0,Q_1) $. Its vertex set
is denoted by $Q_0$ and its set of edges by $Q_1$, both of which are
assumed to be finite, at least so in this paper. The edge orientation
defines maps $s, t:Q_1 \rightrightarrows Q_0$ determined by $s(e) \in
Q_0$ being the \textit{source} and $t(e)\in Q_0$ the
\textit{target}\footnote{ Since the notation $t(e)$ could have,
depending on the source, the opposite meaning, we stress $t$ stands
here for target.  Elsewhere `target' is called `head' and `source' is
referred to as `tail', so $t(e)$ and $h(e)$ are used, respectively,
for our $s(e)$ and $t(e)$.}  of an edge $e\in Q_1$.  Multiple edges,
that is $e_1,\ldots,e_n \in Q_1$ with $s(e_1)=s(e_2) =\ldots =s(e_n) $
and $t(e_1)=t(e_2) =\ldots =t(e_n)$, are allowed, as well as
\textit{self-loops}, to wit those $e\in Q_1$ with $s(e)=t(e)$.
% Unless  it is typographically convenient, we adhere to
% the notation $Q= ( Q_0,Q_1) $ for quivers,
% reserving $G=(V (G) ,E(G) )$ for graphs.
\par

A quiver $Q$ is $B$-edge-weighted by a $*$-algebra $B$, or just $B$\textit{-weighted}, if
 there is a map $b: Q_1\to B $. The matrix of weights,
 $\mathscr{A}_Q(b)=(\mathscr{A}_Q(b)_{i,j}) \in M_{\# Q_0} (B)$, has entries
 \begin{align}\label{symm_weight_matrix_def}
 [\mathscr{A}_Q(b)]_{i,j} = \sum_{ \substack{ e \in Q_1 \\ s(e) =i
     \\ t(e)=j } } b_e \qquad i,j \in Q_0\,.
  \end{align}
The symmetrised weight matrix $\mathscr{A}_Q\symmetr(b)
\in M_{\#Q_0 } (B)$ is defined by its entries being
\begin{align*}
[ \mathscr{A}_Q\symmetr(b)]_{i,j}= \sum_{ \substack{ e \in Q_1 \\ s(e) =i
     \\ t(e)=j } } b_e+ \sum_{ \substack{ e \in Q_1 \\ s(e) =j
     \\ t(e)=i } } b_e^* \qquad i,j \in Q_0\,.
\end{align*}
%  \end{definition}
\noindent
 Clearly $ \mathscr{A}_Q\symmetr(b) \in M_{\#Q_0 }
 (B) $ is a self-adjoint matrix.
\begin{definition}\label{def:Augmentation}
   Let $Q$ be a quiver and denote by $Q^\star$ the following
   \textit{augmentation of} $Q$
  \begin{align} \label{augmentation}
  Q^\star=( Q_0, Q_1 \cup \bar Q_1)\qquad \bar Q_1 = \{ \bar e:  e\in Q_1 , t(e)\neq s(e) \}\,,
  \end{align}
  where $\bar e$ is the edge $e$ with the opposite orientation,
  $s(\bar e) = t( e)$ and $t(\bar e)=s (e)$.
\end{definition}
\noindent
Notice that self-loops
  cause no additional edges in this augmentation
  (which is explicit in Ex.~\ref{ex:Jordan1}).

 \begin{example}
For the triangle quiver $C_3$ of Figure \ref{fig:triangle_weights} any
weight matrix is of the form
\begin{align*} \nonumber
\mathscr{A}_{C_3}(b)= \begin{pmatrix}
       0 & b_{12} & 0 \\
       0 & 0 & b_{23} \\
       b_{31} & 0 & 0
      \end{pmatrix}\,\,
      %\end{align*}
      \text{while its symmetrisation reads }\,\,
      %\begin{align*}
     \mathscr{A}\symmetr _{C_3}(b) = \begin{pmatrix}
       0 & b_{12} & b_{31}^* \\
       b_{12}^*  & 0 & b_{23} \\
       b_{31} & b_{23}^* & 0
      \end{pmatrix}\end{align*}
      with $b_{ij}\in B$. 
The symmetric quiver $C_3$ of Figure \ref{fig:triangle_weights}
has the following general weight matrix:
\begin{align*}\mathscr{A}_{C_3^\star}(b)=  \begin{pmatrix}
       0 & b_{12} & b_{13} \\
       b_{21} & 0 & b_{23} \\
       b_{31} & b_{32} & 0
      \end{pmatrix}\end{align*}
 which is not forced to be self-adjoint (as $b_{12}$ could be
 chosen independent of $b_{21}$), but $     \mathscr{A}\symmetr _{C_3}$ is:
      \begin{align*}
     \mathscr{A}\symmetr _{C_3^\star} (b)=\begin{pmatrix}
       0 & b_{12} +b_{21}^* & b_{13} + b_{31}^* \\
       b_{21} + b_{12}^* & 0 & b_{23} + b_{32}^*\\
       b_{31}+b_{13}^* & b_{32} + b_{23}^*  & 0
      \end{pmatrix}\end{align*}
 \end{example}
\begin{figure}[htb!]
\includegraphics[width=.8\textwidth]{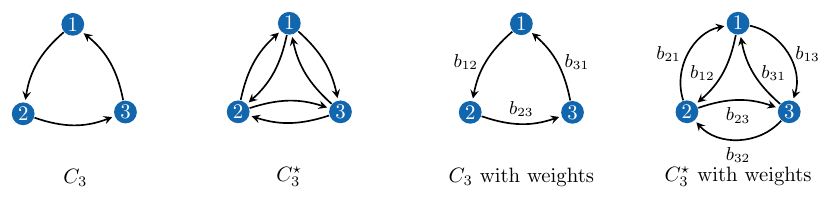} \vspace{-2ex}
 \caption{\label{fig:triangle_weights}
 The $3$-cycle quiver $C_3$, its augmented quiver $C_3^\star$ and their weights 
 }
\end{figure}

% A trace $\Tr_B$ on an involutive algebra $B$ will be for us a cyclic
% linear functional $\Tr_B:B\to \C$ with $\Tr_B(b^* ) =
% \overline{\Tr_B(b)}$ \gris{and $\Tr_B(b^*b) \geq 0$ for each $b\in
%   B$}.

%  \end{align*}

% \begin{definition}
\subsection{Path algebras} \label{sec:path_alg}
Recall that a \textit{path} $p=[e_1,\ldots, e_ k ]$ in a quiver $Q$ is
an ordered sequence $e_1, e_2, \ldots, ,e_{ k -1},e_ k $ of edges
$e_i\in Q_1$ with $t (e_a) =s (e_{a+1})$ for each $a=1,\ldots, k -1$,
for some $ k \in\Zpos$ which we refer to as the length of $p $.
Such integer $k$ will be denoted by $\ell(p)$.  We order the
edges from right to left, so any path looks for $e_1,\ldots, e_ k \in
Q_1$ and $v_0=s(e_1), v_j=t(e_{j}) \in Q_0 $ like
\begin{align} p= [e_1,\ldots,e_k]= \!\! \vspace{-1ex}
\begin{tikzcd}[ampersand replacement=\&]
 v_ k \& v_{ k -1 } \arrow{l}{e_{ k }} \& \arrow{l}{e_{ k -1}} \cdots
 \& \arrow{l}{e_{2}} v_1 \& \arrow{l}{e_{1}} v_0\!
 \end{tikzcd}\! \,. \label{pathcomposition}\end{align}
The source $s(p)$ (resp. target $t(p$)) of a path $p$ is the source
(resp. target) of its first (resp. last) edge, $s(p)=v_0$ and
$t(p)=v_k$ in the case above. If from $v$ to $w$ there is a single
edge $e$, we write $e=(v,w)$, and generally for paths $p=[e_1,\ldots,
  e_ k ] $ made of single edges, an alternative notation for $p$ in
terms of a sequence of vertices is $p=(s(e_1),t(e_2),\ldots, t(e_k)
)$.

% \end{definition}
% \begin{definition}
The set $\mathcal P Q$ consists of all paths in $Q$. These generate
the \textit{path algebra}
% \begin{align*}
 $ \C Q= \langle \mathcal P Q\rangle_\C=\big \{  \sum_{ p \in \mathcal P Q } c_p p : c_p\in \C  \big\}
%  \end{align*}
 $. Given two paths $p_1=[e_1,\ldots,e_k]$ and $
p_2=[f_1,\ldots,f_l]$, their product $p_2 \cdot p_1
=[e_1,\ldots,e_k,f_1,\ldots,f_l]$ is defined to be the concatenation
of $p_2$ after $p_1$ if $t(p_1) =s(p_2) $ and $p_2 \cdot p_1=0$
otherwise.
% \text{ with product }
%  \begin{align*}
% p_2 \cdot p_1 =
% \begin{cases}
%  \text{concatenation of $p_2$ after $p_1$ } & \text{ if } t(p_1) =s(p_2)  \\
% 0 & \text{ otherwise.}
% \end{cases}\end{align*}
The identity is $\sum_{v \in Q_0 } E_v$, where $E_v$ is the
zero-length constant path at $v$.  A \textit{loop} or \textit{closed
  path} at $v\in Q_0$ is a path $p$ of positive length with ends
attached to $v$, $t(p)=v=s(p)$. The \textit{set of loops at} $v$ is
denoted by $\Omega_vQ$. The set of loops on $Q$ based at any vertex is
$\Omega Q=\cup_{v\in Q_0} \Omega_v Q\,. $ \\
% The set of loops of length $ k  >0 $ is denoted by $\Omega^ k  Q$.
% \end{definition}
% \\[.1ex]

\noindent\begin{minipage}{.67\textwidth}
\begin{example}[Path algebra of a quiver]
We count the paths spanning the path algebra for the quiver $Q$ on the
right. Starting at $v=1$ only the constant path $E_1$ ends at $1$;
else one has $e$ and $e'e$, ending at $2$ and $3$, respectively.
Starting at $2$ there are two paths only $E_2$ and $e'$. At $3$ only
the constant path $E_3$ exists, yielding for the most general path the
expression for $P$ in the right for some $\alpha,\beta,c_a \in \C$.
Denoting by $p' \in \C Q$ a path on the same basis with those complex
parameters primed, one has $p' \cdot p = \alpha \beta' e' e +\alpha
c_2 e + \beta c_3 e' +\alpha ' c_1 e + \sum_{a=1}^3 c_a c_a' E_a $.
\end{example}
\end{minipage}\hfil
\begin{minipage}{.383\textwidth}
\begin{align*}\nonumber
% & \\
Q &=\nonumber \runter{\includegraphics[width=2.04cm]{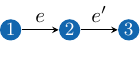}} \\[-2.3ex] \nonumber
p& = \alpha e + \beta e'+ \sum_{a=1}^3 c_a E_a \,.
\end{align*}
\end{minipage}\\[1ex]

If we add an edge $(3\to 1)$ to $Q$ of the previous example (or any
anti-parallel to the existent ones), the path algebra becomes
infinite-dimensional, as it is the general case for quivers with
cycles (the converse statement that infinite-dimensionality implies
cycles is also true \cite[Ex. 1.5.2]{bookQuivRep}).\\

\noindent

\begin{definition}\label{def:Wilson_and_Holonomy}
 Given a quiver $Q$ weighted by $b:Q_1\to B$, the \emph{holonomy}
 $\hol_b(p)$ of a loop $p=[e_1,\ldots,e_ k ]$ and its \emph{Wilson
 loop} $\Wils(p) = \Tr_B \circ \hol_b(p) $ are defined by
 \begin{align}
  \hol_b(p) :=  b_{e_1}  b_{e_2} \cdots   b_{e_ k }  \in B,\, \qquad
  \Wils(p) : = \Tr_B ( b_{e_1}  b_{e_2} \cdots   b_{e_ k } ) \in \C \,. \label{Wilson_and_Holonomy}
     \end{align}
The previous definition is for positive $k$. But it will be also
practical to have $\Wils$ defined\footnote{At risk of being redundant,
we stress that the path $E_v$ should never be confused with a
self-loop $o_v$ at a vertex $v$, which has length $1$ and generally a
non-trivial $\Wils(o_v)$. In this article, $\Wils$ is still classical
and only in \cite{MakeenkoMigdalNCG} we cared about expectation
values.} at $E_v$, the length-$0$ path at $v\in Q_0$. We define a
vanishing holonomy for $E_v$, so thus $\Wils(E_v)=0$.
\end{definition}
Above,  the definition of holonomy $  \hol_b(p)$ could have been
$ b_{e_k} \cdots b_{e_2}    b_{e_ 1 }$ instead, but since the
interesting action functional arises from unitary weights ($b_e^* b_e=1, b_eb_e^*=1$ for all $e\in Q_1$),
and since that functional will depend on the real part of $ \Wils(p) $,
thus of $\hol_b(p) + \hol_b(p)^*=\hol_b(p) + \hol_b(\bar p)$, (being $\bar p$ the reverse of $p$)
this order is not relevant. Let us now denote by $\Tr_{M_n(B)}$ the trace on the algebra of $n\times n$-matrices $M_{n}
(B)$ with entries in $B$, $\Tr_{M_n( B) }( W ) = \sum_{i=1}^n
\Tr_B(W_{i,i})$ for $ W\in M_{n} (B)$.

 \begin{proposition}\label{thm:Wk}
  Let $ k \in \Zpos $ and suppose that $Q$ is $B$-weighted by $\{b_e
  \in B\}_{e\in Q_1}$. Then
  \begin{align} \label{Wilson loop_wilsonloop}
\Tr_{M_n B}  \big( [\mathscr{A}_Q (b) ] ^{ k }  \big)  =  \sum_{ \substack{p \in \Omega Q  \\ \ell(p)=k } }  \Wils(p) \qquad \text{where } n=\# Q_0\,.
%     \vspace{-2ex}
   \end{align}
  \end{proposition}

 \begin{proof} Let us write $ \mathscr A_Q (b)$ as $\mathscr A_Q$
and assume  % \begin{align*}
$ \Tr_B[ (\mathscr A_Q)_{v,i_1} (\mathscr A_Q)_{i_1,i_2} \cdots (\mathscr A_Q)_{i_{k-2},i_{k-1} } (\mathscr A_Q)_{i_{k-1},v }  ] \neq 0$
% \end{align*}
for some fixed set $v, i_1,i_2,$ $\ldots, i_{k-1} \in Q_0$ of
vertices. This implies, by linearity of $\Tr_{B}$, that its argument
is non-zero, so none of the entries $(\mathscr A_Q)_{v,i_1}, (\mathscr
A_Q)_{i_1,i_2}, \ldots, (\mathscr A_Q)_{i_{k-2},i_{k-1} }, (\mathscr
A_Q)_{i_{k-1},v } $ can vanish. Hence there exist edges $e_1$ from $v$
to $i_1, e_2$ from $i_1$ to $i_2$ $\ldots$ and $e_k$ from $i_{k-1}$ to
$v$ on the quiver $Q$. Thus $p=(e_1,\ldots, e_k)$ is a path of length
$k$, or more specifically a loop based at $v$, $p\in \Omega_vQ$, which
one can summarise as:
% \begin{align*} \nonumber
$ \Tr_B [(\mathscr A_Q^k)_{v,v}   ]
% \sum_{i_1,i_2,\ldots, i_{k-1 } \in Q_0 } \Tr_B( b_{v,i_1} b_{i_1,i_2} \cdots b_{i_{k-2},i_{k-1} } b_{i_{k-1},v } ) %\\ \nonumber
%  &
 =\sum_{ \substack{ p \in \Omega_v ( Q),  \ell(p)=k  } } \Wils(p)$, and summing over $Q_0$,
%  \end{align*}
%   \vspace{-1ex}hence \vspace{-1ex}
 \begin{align*}
%  \text{ so}
\Tr_{M_nB}( \mathscr A_Q^k   ) & = \sum_{v \in Q_0 }  \Tr_B( [\mathscr A_Q^k]_{v,v} )  =
\sum_{v \in Q_0 } \sum_{ \substack{ p \in \Omega_v ( Q) \\ \ell(p)=k  } }  \Wils(p) \,. \end{align*}
% = \sum_{ \substack{ p \in \cup_{v\in Q_0} \Omega_v ( Q) \\ \ell(p)=k  } } \Wils(p) \, \qedhere\end{align*}
% \end{align*}which is the claim.
The sum over $v$ cancels the restriction the loops being based at $v$
in $\Omega_v ( Q)$, so one sums over all paths $p$ in $ \cup_{v\in
  Q_0} \Omega_v ( Q) = \Omega Q$.
\end{proof}

\begin{figure}
 \includegraphics[width=.76\textwidth]{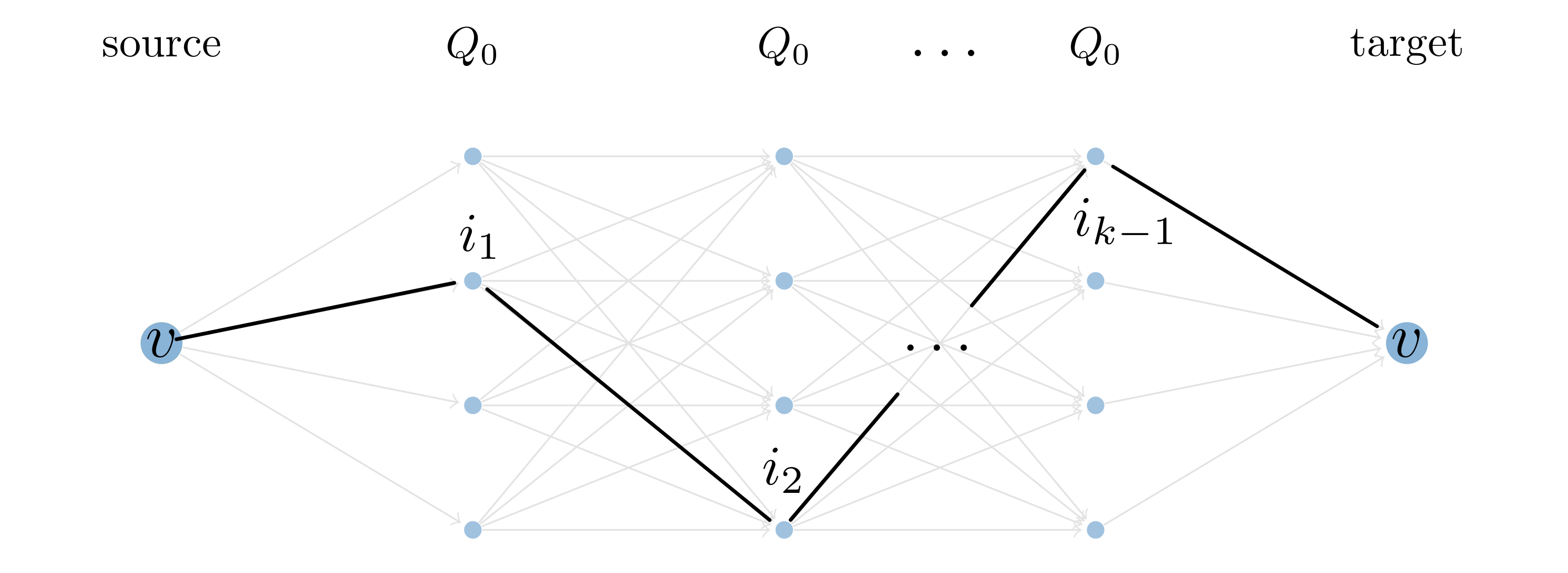}
 \caption{On the proof of Proposition \ref{thm:Wk} and Corollary
   \ref{thm:Wksym}.}
\end{figure}

 \begin{corollary}\label{thm:Wksym}
 With $p=[e_1,\ldots,e_ k ]$ for $e_i\in Q_1$, $
\Wils\symmetr (p)= \Tr  \big(b\symmetr_{e_1} \,b\symmetr_{e_2} \,\cdots   b\symmetr_{e_ k } \big)$ and the  same hypothesis of Proposition \ref{thm:Wk}, one has
 \begin{align} \label{pathcountingWilsonloops}
\Tr_{M_n B} \big(\, [\mathscr A_Q\symmetr (b) ] ^{ k } \,\big) =
\sum_{ \substack{p \in \Omega {Q^\star} \\ \ell(p)=k } }
\Wils\symmetr(p) %$\quad\,.\quad
%   p=[e_1,\ldots,e_ k ]
  \,.
  \end{align}\vspace{-0ex}%
%   with
% \begin{align*}
%   \Wils\symmetr (p)= \Tr  \big(b\symmetr_{e_1} \,b\symmetr_{e_2} \,\cdots   b\symmetr_{e_ k } \big)\,,\qquad
%  p=[e_1,\ldots,e_ k ]\,.
%    \end{align*}
Notice that if $\bar p$ denotes the loop $p\in \Omega _v(Q^\star)$ run
backwards,
\begin{align} \Wils\symmetr(\bar p) = \overline{ \Wils\symmetr(p)}  \,, \label{Wbar=barW}
\end{align}
where the bar on the \textsc{rhs} denotes complex conjugate, so $\Tr_{M_n B} \big(\, [\mathscr A_Q\symmetr (b) ] ^{ k }  \,\big) $
is real-valued.
  \end{corollary}
\begin{proof}
Finding contributions to $\Tr_{B} [\, (\mathscr A_Q) ^{ k }_{v,v} \,]
$ boils down to finding all possible indices $i_1,\ldots, i_{k-1}$
such that none of $(\mathscr A\symmetr_Q) _{v,i_1}, (\mathscr
A\symmetr_Q)_{i_1,i_2}, \cdots ,(\mathscr A\symmetr_Q)_{i_{ k -1},v} $
vanishes. But $(\mathscr A_Q\symmetr)_{a,c}$ does not vanish either if
there is an edge $e$ from $a$ to $c$ or if it exists in the opposite
orientation. So $(\mathscr A\symmetr_Q)_{v,i_1},$ $(\mathscr
A\symmetr_Q)_{i_1,i_2}, \cdots ,(\mathscr A\symmetr_Q)_{i_{ k
    -1},v}\neq0$ implies the existence of loops $p$ in ${Q^\star}$, or
more precisely $p\in \Omega_v {Q^\star}$. The rest of the proof
follows as that of Prop.~\ref{thm:Wk}.
\end{proof}

\begin{example}
Take the following cyclic $B$-weighted quiver% (tacitly $b_{v,v'} \in B$ for each edge from $v$ to $v'$)
\begin{align*}
C_4=
\raisebox{-.45\height}{\includegraphics[width=1.32cm]{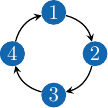}}
\qquad \qquad \mathscr A_{C_4}=\begin{pmatrix} 0 & b_{12 } & 0 & 0
\\ 0 & 0 & b_{ 23} & 0\\ 0& 0& 0 & b_{34 }\\ b_{41 } &0 &0 & 0
        \end{pmatrix}
\end{align*}
Considering the cycle $\sigma=(1234)$, for any vertex $v=1,\ldots, 4$
there is only one loop at $v$, contributing $
\Tr_B( b_{v,\sigma(v)}
b_{\sigma(v),\sigma^2 (v)} b_{\sigma^2(v),\sigma^3 (v)}
b_{\sigma^3(v),\sigma^4 (v)}) = \Tr_{B}(b_{12} b_{23} b_{34}
b_{41}) $ (the 4th power of $\sigma$ is of course the identity, hence
it is a legal contribution; and the equality is just a restatement of
$\Tr_{B}$ being cyclic).
% More generally, $\Tr_{M_4 B} [\, (\mathscr A_{C_4})
%   ^{ k } \, ] $ vanishes unless $k\in 4\Z_{\geq0}$, in fact:
% \begin{align*}
% \frac14\Tr_{M_4 B} [\, \mathscr A_{C_4}^{ k }  \, ]  =
% \begin{cases}
% \Tr_{B}(1_B)  & k=0 \\
% %   & k=4 \\
%  \Tr_{B}[ (  b_{12} b_{23} b_{34} b_{41} )^q ]  & k=4 q\, \qquad (q\in \N)\,.
% \end{cases}
% \end{align*}
In fact, assuming that the cyclic quiver $C_n$ in $n$ vertices is
weighted, one obtains similarly
\begin{align*}
\frac 1n
\Tr_{M_n B} [\, \mathscr A_{C_n}^{ k }  \, ]  =
\begin{cases}
\Tr_{B}(1_B)   & k=0 \\
%   & k=n \\
 \Tr_{B}[ (  b_{12} b_{23} \cdots  b_{n-1,n} b_{n1} )^q ]  & k=n q\, \qquad (q\in \N) \\
0 & 0<k, \text{ $n$ does not divide $k$} \,.
\end{cases}
\end{align*}
\end{example}

\begin{example}\label{ex:traces_symm}
For the quiver above we want traces of the symmetrised weight matrix,
\begin{align*} \mathscr A_{C_4}\symmetr(b)=\begin{pmatrix}
0 & b_{12 } & 0 & b_{41 }^* \\ b_{12}^* & 0 & b_{ 23} & 0\\ 0& b_{23
}^* & 0 & b_{34 }\\ b_{41 } &0 &b_{34}^* & 0
\end{pmatrix}
\text{. Corollary \ref{thm:Wksym} states that we need } C_4^\star =
\raisebox{-.45\height}{\includegraphics[width=1.72cm]{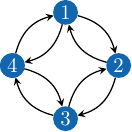}}\,.
\end{align*}
% according to .
 % \begin{align*}
%  \qquad
% \end{align*}
There are eight classes of paths based at, say, the vertex $1$. Since
given two vertices and an orientation, there is one single edge, we
write the paths in terms of the ordered vertices they visit. They read
\begin{itemize}
 \item $p_1=(1,2,1,4,1)$ and its (left-right) specular $p_2=(1,4,1,2,1)$

 \item $p_3=(1,2,3,2,1)$ and its specular $p_6=(1,4,3,4,1)$

 \item $p_5=(1,2,3,4,1)$ and its specular $p_7=(1,4,3,2,1)$

 \item $p_7=(1,2,1,2,1)$ and its specular $p_8=(1,4,1,4,1)$
\end{itemize}
If the cycle $\sigma=(1234)$ acts on these paths \begin{align*}p=(v_1
  ,v_2, v_3, v_4 , v_1 ) \mapsto \sigma p= (\sigma( v_1),
  \sigma(v_2),\sigma(v_3), \sigma(v_4), \sigma(v_1))\,\end{align*} we
  get all the paths of length 4, and $ \Tr_{M_4 B} \big[\, (\mathscr
    A_{C_4} \symmetr)^{ 4 } \, \big] = \sum_{q=0}^3\{ \sum_{a=1}^8
  \Wils [{\sigma^q (p_a)}]
% +\sum_{b=1,2} \Wils _{\sigma^q (r_b)} +\sum_{c=1,3} \Wils _{\sigma^q (p_c^*)}
\} $.  The Wilson loops are implicit, but immediate to compute,
e.g. $\Wils({p_2}) = \Tr_B( b_{41}^* b_{41} b_{12}b_{12}^*)
=\Wils({p_1}) $ for the paths $p_1$ and $p_2$ listed above.
\end{example}
\vspace{-4ex}
\begin{example}\label{ex:Jordan1}
 Using the Jordan quiver $J=
 {\includegraphics[height=28.5pt]{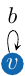}}$ ($b\in B$ is the weight),
 we illustrate now how self-loops and double edges sharing endpoints
 are treated. Any formal series $f(x)=\sum_{l=1}^\infty f_l x^l$ can
 be computed
% \begin{align*}
$\Tr_B \big ( f(\mathscr A_{J}) \big) =\sum_{l=1}^\infty f_l \Tr_B(b^l).$
% \end{align*}
The Jordan quiver does not suffer from augmentation, $J^\star=J$, but
later on adding self-loops will be important. Let
$\mathring J$ denote the quiver $J$  with an extra self-loop, $\mathring J= \!\!
\raisebox{-5pt}{\includegraphics{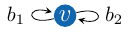}}$, with weights $b_1,b_2 \in
B$. Then the path algebra $\C \mathring J = \C\langle b_1,b_2 \rangle
$ is the free algebra in two generators, so the formal series
evaluated in the weights matrix reads
% $F(x)=\sum_{l=1}^\infty  f_l x^l$,
\begin{align}   \notag \Tr_B   \big (  f(\mathscr A_{\mathring J}) \big)
 \notag =\sum_{l=1}^\infty f_l \Tr_B \bigg( \sum_{\substack{ m \text{
      \tiny monic, degree $l$} \\ \text {\tiny monomials in $b_1, b_2$
}}} m \bigg)  & =    f_1 \Tr_{B}( b_1 +b_2) + f_2 \Tr_B( b_1^2 +
2b_1b_2 + b_2^2 ) \\[-3.3ex] &+ f_3 \Tr_B( 3 b_1^2 b_2 + 3 b_2 \notag b_1^2+
b_1^3 + b_2^3)  + O(4). %\\
% &+f_4 \Tr_B( b_1^4 + b_2^4 + 2 b_1 b_2 b_1 b_2 + 4
% b_1^2 b_2^2 +4 b_1^3 b_2 \notag +4 b_2^3 b_1 ) + O(5)\,.
\end{align}
\end{example}

 \subsection{Quiver $\pS$-representations and path algebra modules}\label{sec:Q-reps}

  \begin{definition}\label{def:QuivRepsCat}
 Given a  given (small) category $\cC$, a \textit{representation of a quiver} $Q$
 in $\mathcal C$ is a functor $\mathcal P Q\to \cC$, i.e. a pair of set maps
 \begin{align}
 \label{setmapsPQ-C}
 X: Q_0 \to\mathcal C \qquad  \Phi: \mathcal PQ \to \hom_{\mathcal C}
 \end{align}
 where the map $\Phi$ is shorthand for a family of maps $\{\Phi(p) \in
 \hom_{\mathcal C} (X_{s(p)}, X_{t(p)} )\}_{p\in \mathcal PQ}$. We
 sometimes write the arguments as subindices $\Phi_p:=\Phi(p)$ and
 $X_v:=X(v)$ in order to minimise brackets (and thus avoid $X(s(p))$
 when $v=s(p)$ for instance).
 \end{definition}

Therefore a representation of a quiver $Q=(Q_0,Q_1)$ in $\pS$ is an
association of prespectral triples $X_v=(A_v,\lambda_v,H_v)$ to
vertices $v\in Q_0$ and of $*$-algebra maps $\phi_p:A_{s(p)} \to
A_{t(p)} $ and unitary maps $L_p: H_{s(p)} \to H_{t(p)} $ to paths $p
\in \mathcal P Q$. All satisfy $L_e \lambda_{v} ( a ) L_e^* = \lambda_w[
  \phi_e(a) ] $ if  $v=s(e)$ and $w=t(e)$ for each $e\in Q_1\subset \mathcal PQ$.
  Functoriality implies that $L_p= L_{e_k}L_{e_{k-1}}\cdots L_{e_1}$
  for  $p=[e_1, e_2,\ldots, e_{n-1}, e_{k} ]\in \mathcal P Q $.
%   Later,
%   we will be interested in tracing these matrices, and in the relevant
%   expressions, we always have sums of the form $\Tr (L_p+L_{\bar p}) $.

\begin{definition}\label{def:CQmodpS}
 Let $Q$ be a quiver and denote by $\C Q\modpS$ the \textit{category
   of $\C Q$-modules over $\pS$}.  To wit, objects of $\C Q\modpS$ are
 prespectral triples that further carry an action of the path algebra
 $\C Q$ by $\pS$-morphisms.  Matching our path composition
 \eqref{pathcomposition}, this action is by the left and will be
 denoted\footnote{For completeness, we stress that since this $M$ is
 also in $\pS$, it is therefore of the form $M=(A,H)$ and therefore
 the action of a path $p$ looks like $m=(a,\psi) \mapsto p\cdot
 (a,\psi)=p\cdot m$ for $a\in A,\psi\in H$.}  by $m\mapsto p\cdot m$
 ($p\in \mathcal P Q, m\in M$). For $M,M' \in \C Q\modpS $, morphisms
 $\chi \in \hom_{\C Q\modpS}(M,M')$ are $\pS$-morphisms that respect
 the action of $\C Q$, $\chi (p\cdot m) = p\cdot \chi(m)$ for all
 $m\in M$ and all $p\in \mathcal P Q$.
\end{definition}
\noindent
The previous category will be proven to be equivalent to one more interesting in this paper.
\begin{definition}\label{def:RepQ}
All functors $\mathcal PQ\to\pS$ from the free category\footnote{I
thank John Barrett and Masoud Khalkhali for correcting my previous
notation during a talk; I had written `functors $Q\to \pS$'.}
associated to a quiver $Q$ to $\pS$ form the \textit{space of
  representations},
\begin{align*} \Rep_\pS(Q) := \{ \text{functors } \mathcal PQ \to \pS  \}\,.\end{align*}
\end{definition} \par
\noindent
% \begin{minipage}{.57\textwidth}
Actually $\Rep_\pS(Q) $ is a category too, the functor category
(other notations are $\mathcal P Q^{\pS} $ or $ [\mathcal P Q,\pS]$, which will not be used here),
but with some usual abuse of notation we shall also refer to
$\Rep_\pS(Q)$ as the objects of that category, quiver representations. Once we know that
$R=(X_v ,\Phi_p) _{v\in
  Q_0, p\in  \mathcal PQ}
  \in \ReppS(Q)$
  we are sure that we can gain any $\Phi_p$
  from the values of $\Phi$ at the edges that compose $p$, so
  we write  $ R=(X_v ,\Phi_e) _{v\in
  Q_0, e\in Q_1}$.
  Given another $ R'=(X_v' ,\Phi_e') _{v\in
  Q_0, e\in Q_1}\in \pS$, a morphism $\Upsilon\in \hom_{\Rep_\pS(Q)}(R,R')$ is by definition
a natural transformation $\Upsilon:R \to R' $,
i.e. a  family $\{ \Upsilon_y:(A_y,H_y) \to (A_y',H_y') \}_{y\in Q_0}$
that makes the diagram \eqref{comm_diag_NatTr_expl}
commutative for each $e$, wherein  $v=s(e) \text { and } w=t(e)$,
% \end{minipage}
% \quad
% \begin{minipage}{.4\textwidth}
\begin{align} \label{comm_diag_NatTr_expl}
% \qquad \text{ If } ) \qquad\,
\begin{tikzcd}[ampersand replacement=\&, column sep=large]
  (A_v,H_v)  \ar [d,swap,"\Upsilon_v"] \ar [r,"{\Phi_e=(\phi_e,L_e)}"] \&  (A_w,H_w)  \ar [d, "\Upsilon_w"] \\
  (A_v',H_v') \ar [r,swap, "{\Phi_e'=(\phi_e',L_e')} "] \& (A_w',H_w')
\end{tikzcd}
% \begin{tikzcd}[ampersand replacement=\&, column sep=large]
%   (A_v,H_v)  \ar [d,swap,"G_v"] \ar [r,"{\Phi_e=(\phi_e,L_e)}"] \&  (A_w',H_w')  \ar [d, "G_w"] \\
%   (A_v',H_v') \ar [r,swap, "{\Phi_e'=(\phi_e',L_e')} "] \& (A_w',H_w')
% \end{tikzcd}
\end{align}
% \end{minipage}
% \\

% \begin{align*}
% \begin{tikzcd}[ampersand replacement=\&]\label{diagrm}
%   X_v  \ar [d,swap, "G(\alpha)_v"] \ar [r, "\Phi_e"] \&  X_w  \ar [d, "G(\alpha)_w"] \\
%   X_v'  \ar [r,swap, "\Phi_e' "] \& X_w'
% \end{tikzcd}\end{align*}

% \begin{align*}
% \begin{tikzcd}[ampersand replacement=\&]
%   (A_v,H_v)  \ar [d,swap, "G_v"] \ar [r, "\Phi_e=(\phi_e,L_e)"] \&  (A_w',H_w')  \ar [d, "G_w"] \\
%   (A_v',H_v')  \ar [r,swap, "\Phi_e'=(\phi_e',L_e') "] \& (A_w',H_w')
% \end{tikzcd}
% \end{align*}

The next classical fact for ordinary quiver representations (see \cite{bookQuivRep} for the proof
having vector spaces as target category) can be extended to prespectral triples.
% The strategy of the proof is to let one functor be `taking direct sums over all vertices';
% the other direction is, in essence, by
% restricting to a vertex. This will now be spelled out carefully.

\begin{proposition}\label{thm:Eq_Cats}
The following equivalence of categories holds:
% \begin{align*}
$\Rep_\pS Q \simeq \C Q\modpS\,.$%\end{align*}
\end{proposition}

\begin{proof}  We exhibit
two functors that are mutual inverses% \begin{align*}
$\begin{tikzcd}[ampersand replacement=\&]
%     \ReppS  \C Q & \C Q \modpS
\ReppS  Q  \arrow[r,shift left,"F"] \&  \C Q\modpS.  \arrow[l,shift left,"G"] \end{tikzcd}$ \\
%    \noindent
% \end{align*}
\textit{From representations to modules}. Let $R=(X_v ,\Phi_e) _{v\in
  Q_0, e\in Q_1}$ be a $\pS$-representation of $Q$, and \begin{align}
                           \label{FromReptoMods}F(R) :=\oplus_{v\in Q_0} X_v\,.
                          \end{align}
  To give $F(R)$ the structure of module we
take a generator $p=[e_1,
  e_2,\ldots, e_{n-1}, e_{n} ]\in \mathcal P Q $, where each $e_j\in
Q_1$ and extend thereafter by linearity to $\C Q$.  For $x = (x _v
)_{v\in Q_0}\in F(R) $ define $p\cdot x \in F(R)$ to have the only
non-zero component
 \begin{align} \label{rightorder}
(p\cdot x)_{t(p)}   & =  \Phi_{p}  x_{s(p)} =
\Phi_{e_n}\circ \Phi_{ e_{n-1}} \circ \cdots  \circ\Phi_{ e_1} ( x_{s(p)}) \,.
 \end{align}
Equivalently, $p\cdot x_v=0$ unless $p$ starts at $v$, in which case
the only surviving component of $x$ after being acted on by $p$ is
$(p\cdot x)_{w} = %\Phi(e_n) \Phi( e_{n-1}) \cdots \Phi( e_1)
\Phi_{e_n}\circ \Phi_{ e_{n-1}} \circ \cdots \circ \Phi_{
  e_1}(x_{v})$; here $v=s(e_1) $ and $ w=t(e_n)$. This defines the
functor $F$ on objects---now we define $F$ on a morphism $\Upsilon \in
\hom_{\ReppS(Q)}(R,R')$ by letting $ F(\Upsilon) : F(R)\to F(R')$ be
(adopting notation similar to \eqref{FromReptoMods} for primed objects
too)
 \begin{align} F(\Upsilon) := \bigoplus_{v\in Q_0} \Upsilon_v : \bigg(\bigoplus_{v \in Q_0}  X_v \bigg) \to \bigg( \bigoplus_{v \in Q_0} X_v' \bigg), \quad  (x_v)_{v\in Q_0} \mapsto \big ( \Upsilon_v(x_v) \big)_{v\in Q_0} \label{defFonMorphisms} \, . \end{align}
It remains to verify that $F(\Upsilon)$ is a $\C Q$-module map over
$\pS$. Given $p=[e_1, e_2,\ldots, e_{n-1}, e_{n} ]\in \mathcal P Q $,
the horizontal concatenation of the diagram
\eqref{comm_diag_NatTr_expl} for the edges $e_i$ of $p$, so
$t(e_i)=s(e_{i+1})$ for $i=1,\ldots, n-1$, implies the commutativity
of the next diagram
\begin{align} \label{Fismodulemap}
\begin{tikzcd}[ampersand replacement=\&, column sep=large]
  X_{s(p)} \ar [d,"\Upsilon_{s(p)}"] \ar [r,"\Phi_{e_1}"] \&  X_{t(e_1)}    \ar [d, "\Upsilon_{t(e_1)}"]
  \ar [r,"\Phi_{e_2}"]
  \&   X_{t(e_2)}    \ar [d, "\Upsilon_{t(e_2)}"]  \ar [r,"\Phi_{e_3}"] \& \cdots
   \ar [r,"\Phi_{e_{n-1}}"]  \&  X_{t(e_{n-1})}    \ar [d, "\Upsilon_{t(e_{n-1})}"]   \ar [r,"\Phi_{e_n}"] \&  X_{t(p)} \ar [d, "\Upsilon_{t(p)}"]
  \\
 X_{s(p)}'    \ar [swap,r,"\Phi_{e_1}' " ]\&  X_{t(e_1)}' \ar [swap,r,"\Phi_{e_2}' " ]
  \&   X_{t(e_2)}' \ar [swap,r,"\Phi_{e_3}' " ]   \& \cdots \ar [swap,r,"\Phi_{e_{n-1}}' " ]  \&  X_{t(e_{n-1})}' \ar [swap,r,"\Phi_{e_n}' " ] \&  X_{t(p)}'
\end{tikzcd}
\end{align}
which, together with Eq.
\eqref{rightorder} and the definition \eqref{defFonMorphisms} of $F$
on morphisms,
yields component-wise (that is, vertex-wise) the property
$F(\Upsilon) (p\cdot x) = p \cdot F(\Upsilon) (x)$ for all $ x \in \oplus_{v\in Q_0} X_v = F(R)$.
So $F(\Upsilon)$ is indeed a  module morphism in $F(\Upsilon ) \in \hom_{\C Q \modpS} ( F(R),F(R') )$.
\\

\noindent
\textit{From modules to representations}.  Now take a module $M \in \C
Q\modpS $.  By definition, $M=(A,H)\in \pS$. Since $(A,H)$ bears an
action of $\C Q$, one can let the constant paths $E_v$ act on it to
build a prespectral triple $(A_v,H_v) := E_v \cdot (A,H) $ for each
$v\in Q_0$ as follows.  If $e\in Q$, we prove that the action of $e$
on $(A,H)$ allows a restriction $E_{s(e)} (A,H) \to E_{t(e)} (A,H) $
of multiplication by $e$. For this it is enough to observe that since
$(A,H)$ is a module, the multiplication of the paths $e \cdot
E_{s(e)}=e = E_{t(e)} \cdot e $ holds also `in front of $M=(A,H)$',
namely
\begin{align} \label{contentionXv} %\raisetag{1cm}
e \cdot (A_{s(e)},H_{s(e)}) =e  ( E_{s(e)} M  ) =( e  \cdot E_{s(e)} ) M =e M
= (E_{t(e)} \cdot  e ) M = E_{t(e)} \cdot  (e M)\,.
\end{align}
But then  $ e \cdot (A_{s(e)},H_{s(e)})    \subset  E_{t(e)} (A,H)= (A_{t(e)},H_{t(e)})      $,
so we can define $\Phi_e : (A_{s(e)} ,H_{s(e)}) \to (A_{t(e)} ,H_{t(e)})$ as the restriction
of $ m\mapsto  ( e\cdot  m)$ to $(A_{s(e)} ,H_{s(e)})$. We let thus
\begin{align*}
G(A,H)= (X_v,\Phi_e)_{e\in Q_1,v\in Q_0} \text{ or explicitly }  (A_v,H_v)  = E_v \cdot (A,H) , \Phi_e =
( m\mapsto e \cdot m  )\Big|_{ (A_{s(e)},H_{s(e)} ) }\,,
\end{align*}
where each $\Phi_e\in \hom_\pS\big ( (A_{s(e)},H_{s(e)}),(
A_{t(e)},H_{t(e)}) \big)$ is well-defined since $\C Q$ acts on $(A,H)$
by $\pS$-morphisms, by definition of $\C Q\modpS$.  It remains to
verify that a $\C Q\modpS$-morphism $\alpha: (A,H) \to (A',H') $
yields a $\ReppS(Q)$-morphism $G(\alpha): G(A,H) \to G(A',H') $. We
let $G(\alpha)_v: (A_v,H_v) \to (A_v',H_v') $ be the restriction of
$\alpha $ to $ (A_v,H_v)$.  Indeed, its well-definedness follows from
% \begin{align*}
\[
G(\alpha)_v (A_v,H_v)= \alpha (E_v \cdot  (A,H) ) =
E_v \cdot \alpha ( A,H)   \subset E_v (A',H') =(A_v',H_v').\]
% \end{align*}
Finally, we verify that the family $\{G(\alpha)_v\}_{v\in Q_0} $
is indeed a map of representations
\begin{align}\label{diagrm}
\begin{tikzcd}[ampersand replacement=\&]
  (A_v,H_v)  \ar [d,swap, "G(\alpha)_v"] \ar [r, "\Phi_e"] \&  (A_w,H_w)  \ar [d, "G(\alpha)_w"] \\
  (A_v',H_v')  \ar [r,swap, "\Phi_e' "] \& (A_w',H_w')
\end{tikzcd}\end{align}
That this diagram commutes follows from (starting from the right-down composition)
\begin{align*}
G(\alpha)_w \circ \Phi_e  (a_v,\psi_v) &= G(\alpha)_w \big(e\cdot (a_v,\psi_v)\big) \qquad (a_v,\psi_v)\in (A_v,H_v) \\
\notag & = \alpha\big(e\cdot (a_v,\psi_v)\big) \\
\notag &= e\cdot   \alpha  (a_v,\psi_v)\\
\notag &= e\cdot  G( \alpha)_v  (a_v,\psi_v) \\
\notag &= \Phi'_e  \circ  \alpha_v  (a_v,\psi_v)\,.
\end{align*}
The first line is by definition of $\Phi_e$. For the second observe
that the relation \eqref{contentionXv} implies $e\cdot (a_v,\psi_v)\in
X_w$, where $G(\alpha)_w$ is by definition $\alpha$. The third
equality holds since $\alpha$ is a module morphism, and fourth and
fifth equalities follow by the same token as the second one and the
first one did, respectively.  As every step in the construction of $F$
and $G$ is explicit above, it is routine to check that the
compositions of $F$ (essentially, taking direct sum over vertices) and
$G$ (restriction to a vertex $v$ via the constant path $E_v$) are
naturally equivalent to the identities, $G\circ F \simeq \mathbf
1_{\ReppS(Q)}$ and $F\circ G\simeq \mathbf 1_{\C Q\modpS}$.\end{proof}

 \subsection{Combinatorial description of $\Rep (Q)$} \label{sec:combinatorialRepQ}
In order to classify representations, we introduce a combinatorial object.
\begin{definition}\label{def:Bratteli_sieć}
A \textit{Bratteli network} over $Q$ is a collection of:
 \begin{itemize}
 \itemb some positive integer $l_v$ per vertex $v$,\vspace{1ex} \itemb
 $l_v$-tuples $\bn_v,\br_v \in \Z_{>0} ^{l_v}$ for each
 vertex, \vspace{1ex} \itemb for all $e\in Q_1$, $C_e \in M_{l_{s(e)}
   \times l_{t(e)}} ( \Z_{\geq 0}) $ is compatible with the previous
 maps in the sense that: \begin{align*}\bn_{t(e)}= C\trans_e
   \bn_{s(e)} \quad\text{and}\quad\br_{s(e)}= C_e
   \br_{t(e)}\,. \end{align*}
 \end{itemize}
 Although due to the last condition the integer tuples are not
 arbitrary, we denote Bratteli networks by ($\bn_Q,\br_Q$), leaving
 $C$ implicit, and sometimes the quiver too, thus writing only
 $(\bn,\br)$.
\end{definition}

\noindent
The compatibility conditions on $C_e$ are an element of the
construction of a module for the algebra $\C Q$, in the sense of the
\textsc{rhs} of the equivalence of categories of Proposition
\ref{thm:Eq_Cats}. Let
\begin{align*}C (p):=
(C_{e_{1} }  C_{e_2} \cdots  C_{e_{k-1 } }C_{e_k})
  \text{ for } p=[e_1,e_2,\ldots,  e_{k}]\,.
\end{align*}
Notice the `wrong order' ($C\trans$ then satisfies a similar condition
in the right order).  Thus for any path $p$ in $Q$, the labels
assigned to the vertices satisfy
\begin{align} \label{both-way}
\bigg( \begin{matrix}
\bn_{t(p)}  \\[.3ex] \mathbf r_{s(p)}
\end{matrix} \bigg) =  \bigg( \begin{matrix}
C_p\trans &  0_{l_s}  \\[.3ex]
0_{l_t} & C_p
\end{matrix} \bigg)
\bigg( \begin{matrix}
\mathbf n_{s(p)}       \\[.3ex]
\mathbf r_{t(p)}
\end{matrix} \bigg) \,.\end{align}

\noindent
For a connected quiver $Q$ and for fixed $N\in \Z_{>0}$ we define the \textit{restricted representation space}
\begin{align} \label{rep_N}
\Rep^N_{\pS}(Q) :=\{ R\in  \ReppS(Q)  : \dim H_v = N  \text{ for some }v\in Q_0\}\,,
\end{align}
for whose size we estimate an upper-bound
in Appendix \ref{app:boundRepQN}. Observe that
$\dim H_v=\dim H_w$ for any two vertices $v,w$
of a connected quiver. Although exact counting should be
possible, we only need later the fact that this new space is finite-dimensional
(in fact, the finiteness of Bratteli networks of
the restricted representation spaces is what is essential). Recall first, that the underlying graph $\Gamma Q$ of a quiver $Q$
has the same vertices as $Q$ and keeps all edges after forgetting orientations, e.g.
\begin{align*}
\Gamma \big ( \runterhalb{\includegraphics{Pi1}} \big) =  \runterhalb{\includegraphics{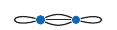}}  \end{align*}

\begin{figure}
% \subfloat[Bratteli network from $\Rep_\pS(Q)$ but which
%  lies outside $\Rep^N_\pS(Q)$ for any $N$. Here all hidden multiplicities satisfy $\br_x=(1,1,1), \br_v=(1,1)=\br_z,\br_y=1$.\label{fig:Bratteli_sieć}]{
%  \includegraphics[width=4.985cm]{Brattelinetz}
% }
 \quad \subfloat[Bratteli network from $\Rep^{N=5}_\pS(Q)$. Here all
   hidden multiplicities satisfy $\br_x=(1,1,1),
   \br_v=(1,1)=\br_z,\br_y=1$. \label{fig:Bratteli_siećN}]
       {\includegraphics[width=5.1985cm]{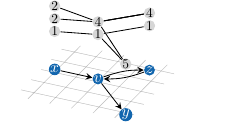}} \quad
       \subfloat[A non-example of Bratteli network on $Q$, since at
         the edge from $z$ to $v$ we do not have a $*$-algebra
         map \label{fig:Bratteli_sieć}]{
         \includegraphics[width=5.1985cm]{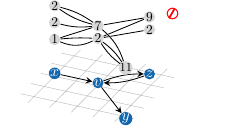} }
 \caption{Bratteli networks from $\ReppS(Q)$ with light gray
   vertices over a quiver $Q$ in blue vertices.
%    The entries of  all $\br_w$ for each $w\in Q_0$ are full of ones, which, to simplify the picture,
%  are omitted.
 Any path  in $Q$ should lift to a sequence of Bratteli diagrams. This constrains e.g. the
   edges between $v$ and $z$ to have a symmetric Bratteli diagram (thus the identity).}
 \end{figure}

% We stress that in \eqref{ReppSasdisjointsums} the quantity in curly brackets is
% non-empty

With the
notation $|\mathbf t|=\#\{j : t_j >0 \}$ for a given $\mathbf t\in
\Z^\infty_{\geq 0},$ one has:

\begin{lemma} \label{thm:RepQis...}
 For $\mathbf n\in \Z^\infty_{\geq 0}$ with finite $|\mathbf n|$, we abbreviate  $\uni(\mathbf n):=\textstyle\prod_{j=1}^{|\mathbf n|} \uni(n_j)$. One has
\begin{align}\Rep_\pS(Q) & =
\coprodsub {\text{Bratteli}\\ \text{networks} \\ (\bn_Q, \br_Q ) } \bigg\{
\prod_{e\in Q_1}\uni\big(\mathbf n_{t(e)} \big)  \label{ReppSasdisjointsums}
\bigg\}
\end{align}
where the disjoint union is over all integers $(\bn_v,\br_v)_{v\in Q_0} $  that yield  Bratteli networks over $Q$.\end{lemma}

\begin{proof}

 By definition,  the space of all quiver representations will
% count all labels at vertices $v$ of $Q$ by some prespectral triple $X_v$,
count all morphisms $X\sour \to X\targ$ along all edges $e$, over all
labelings $Q_0 \ni v\mapsto X_v \in \pS$ that are compatible along all
paths. The latter means to remove from the next space
\begin{align}\label{toobigRepQ}
% \Rep_\pS(Q)=
\coprodsub{Q_0\to \pS  \\ \quad v\mapsto X_v\,\,} \bigg\{   \prod_{e\in Q_1}
\hom_\pS ( X_{s(e)}, X_{t(e)}) \bigg\}\,
\end{align}all those vertex-labelings for which
there exist an edge $e_0\in Q_1$ such that $\hom_\pS ( X_{s(e_0)},
X_{t(e_0)}) =\emptyset $ (cf. Example \ref{ex:addedV2}, tailored at
showing why we can exclude said maps $X: Q_0\to \pS$).  This happens
for instance if there exist no Bratteli diagrams $A_{s(e)} \to
A_{t(e)}$, or if the Hilbert spaces they act on are not isomorphic; in
general, such situations are avoided by imposing that the labels of
the vertices satisfy Eq.~\eqref{both-way} along any path $p$ in $Q$.
Let us use a combinatorial description of the prespectral triples and
denote by `\text{\scriptsize $\lozenge$}' the next condition on labels
$Q_0 \to \Z_{\geq 0}^\infty\times\Z_{\geq 0}^\infty,\, v \mapsto
(\bn_v,\br_v)$:
  \begin{align*}
  \lozenge = \Big\{ & |\bn_v| =|\br_v|=:l_v<\infty \text{ and for each
    path $p\in \mathcal P Q$, there exist a matrix } \\[.5ex] &C_p\in
  M_{l_{s(p)} \times l_{t(p)} } (\Z_{ \geq 0}) \text{ such that }
  \bn_{t(p)} = C_p\trans \mathbf n_{s(p)} \text{ and } \br_{s(p)} =
  C_p \mathbf r_{t(p)} \Big\}\,. \end{align*} Aided by Proposition
  \ref{thm:reductionU}, we obtain
\begin{align*}\nonumber
\Rep_\pS(Q) & =\coprodsub {Q_0 \to \Z_{\geq 0}^\infty\times\Z_{\geq
     0}^\infty  \\[.3ex] v \mapsto (\bn_v,\br_v) \,\, }^ \lozenge
%      \\[.2ex] |\bn_v|
%   =|\br_v|<\infty \\[.5ex] \text{ \eqref{both-way} holds  }\forall p\in \mathcal PQ\\
%  \bn_{t(p)}  = C_p\trans  \mathbf n_{s(p)}      \\
% \br_{s(p)} = C_p \mathbf r_{t(p)}
%   }
\bigg\{ \prod_{e\in Q_1} \hom_\pS ( X_{s(e) }(
 \bn_{s(e)}, \br_{s(e)} ) , X_{t(e) }( \bn_{t(e)}, \br_{t(e)} ) )
\bigg\}
\\ & = \coprodsub {Q_0 \to \Z_{\geq 0}^\infty\times\Z_{\geq
    0}^\infty \\[.3ex] v \mapsto (\bn_v,\br_v) \,\,}^\lozenge \bigg\{ \prod_{e\in Q_1} \,\,
\coprodsub{(\bn_{s(e)},\br_{s(e)}) \stackrel{\Brat}{\to}
  (\bn_{t(e)},\br_{t(e)}) } \Big[\textstyle\prod_{i=1}^{|\mathbf
    n_{t(e)}|} \uni\big( n_{t(e),i} \big) \Big] \bigg\}\,.
\end{align*}
But this can be rephrased in terms of Bratteli networks as in the
claim.
\end{proof}
\noindent
\begin{minipage}{.734\textwidth}
\begin{example}[Bratteli networks exclude spurious labels $Q_0\to \pS$]\label{ex:addedV2}
To see which maps $X$ in \eqref{toobigRepQ} do not contribute to
$\ReppS (Q)$, which we call them \textit{spurious} and tag with a
`$\diamond$', consider the quiver $Q$ drawn on the plane on the right.
The hidden Hilbert space the matrix algebras (lightgray) act on are:
$3 \C^3$ acted upon by $\M 3$ and else $\C^n$ acted upon by $\M n$.
Although the black arrows upstairs are all legal, this choice of
algebras does not lead to a representation of $Q$, since the edge
$e_3=(z,v)$ can never be lifted, regardless of how we redefine the
Hilbert spaces.
\end{example}
\end{minipage}
~
\begin{minipage}{.2\textwidth}
\[\includegraphics[width=5.5cm]{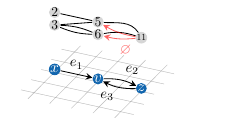}\]
\end{minipage}\\[.3ex]%
\noindent This picture defines a spurious map $X^\diamond:Q_0\to \pS$,
i.e. whose contribution to $\ReppS (Q)$ is empty, as
\[
\prod_{e\in Q_1} \hom_\pS ( X_{s(e) }^\diamond, X_{t(e) }^\diamond) =
     [\uni(5)\times\uni(6)]_{v=t(e_1)}\times \uni(11)_{z=t(e_2)}
     \times \emptyset_{v=t(e_3)} =\emptyset\,. \] Bratteli networks
     are the combinatorial data behind all those labels $X:Q_0\to \pS$
     that, unlike this $X^\diamond$, do contribute to $\ReppS (Q)$.
     Bratteli networks thus already excluded all spurious labels.
 \subsection{The gauge group}

\begin{definition}[Equivalence of representations]\label{def:equiv_of_reps}
Two $\pS$-representations $R=(X_v ,\Phi_e) _{v\in Q_0, e\in Q_1}$ and
$R'=(X_v' ,\Phi_e') _{v\in Q_0, e\in Q_1}$ of $Q$ are
\textit{equivalent} if there exist a family $\{\Upsilon_v: X_v\to
X'_v\}_{v\in Q_0}$ of invertible $\pS$-morphisms, such that for any
path $p\in \mathcal P Q$ the leftmost diagram of Figure
\ref{fig:NatTrans} commutes. This boils down to the existence of an
invertible natural transformation $\Upsilon: R \to R'$ (we then write
$R\simeq R'$).
\end{definition}
% \end{minipage}
% \begin{minipage}{.4\textwidth}
%  \end{minipage}
% \reqnomode
We aim at determining $\ReppS(Q)$ modulo equivalence, which we wish to
represent by quotienting by a group action, $\ReppS(Q) / \mathcal
G_\pS(Q) $.  This group $ \mathcal G_\pS(Q) $, or just $ \mathcal G
(Q) $, is defined next:

\begin{definition} The \textit{gauge group} $ \mathcal G(Q)$  of a quiver
reads $\mathcal G(Q):=\coprod_{R/\simeq } \Aut_{\ReppS(Q)}(R)$.
 \label{def:GaugeGroup}
\end{definition}
\vspace{-2ex}
\begin{figure}[htb!]
\begin{align*}\notag
{
\begin{tikzcd}[ampersand replacement=\&]
  X_{s(p)}  \ar [d,swap, "\Upsilon_{s(p)}"] \ar [r, "\Phi_p"] \&  X_{t(p)}  \ar [d, "\Upsilon_{t(p)}"] \\
  X_{s(p)}'  \ar [r,swap, "\Phi'_p "] \& X_{t(p)}'
\end{tikzcd}}
\runter{\includegraphics[width=.39\textwidth]{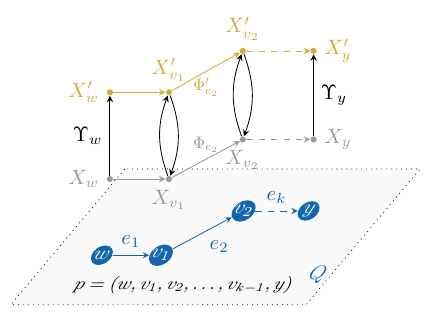}} \hspace{-3ex}
\raisebox{-.506\height}{\includegraphics[width=.39\textwidth]{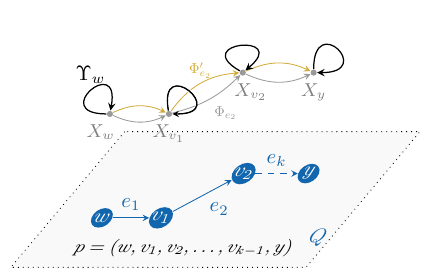}}
\end{align*}\vspace{-3ex}
 \caption{\label{fig:NatTrans}\textit{Left:} The commutative diagram
   defining invertible natural transformation $R=(X,\Phi)\to
   R'=(X',\Phi')$. \textit{Middle:} A path on a quiver $Q$ is shown,
   and two lifts $\Phi'_p = \Phi_{e_k}' \circ \cdots \circ \Phi_{e_2}'
   \circ \Phi_{e_1}' $ (uppermost) and $\Phi_p=\Phi_{e_k} \circ \cdots
   \circ \Phi_{e_2} \circ \Phi_{e_1}$ (lower). As vertical morphisms
   are invertible, the diagram on the left commutes for any path $p$
   if and only if it does so for any edge $e$, for intermediate curved
   arrows cancel out after concatenation of single-edge
   diagrams. \textit{Right:} The (as explained below w.l.o.g.) reduced
   case of the two representations $R, R'$ coinciding in objects,
   $X_v=X_v'$ for all $v\in Q_0$. Black arrows represent the
   parameters $\Upsilon_v\in \Aut_{\pS}(X_v)$ of the gauge group
   $\mathcal{G}(Q)$ (given a Bratteli network). }
\end{figure}
 \begin{lemma} \label{thm:GaugeGroupThm}
 The gauge group is well-defined. Leaving the target category
 implicit, it in fact reads
\begin{align}\label{gauge_groupofQ}
\mathcal{G}(Q) =
\coprodsub {\text{Bratteli}\\ \text{networks} \\ (\bn_Q, \br_Q ) } \prod_{v\in
  Q_0 } \Aut_{\pS} (X_v (\bn_Q,\br_Q) )\,. \end{align}
 \end{lemma}
 \begin{proof}
 If $R'\simeq R$ then there is an invertible natural transformation
 $\alpha: R\to R'$ and $\Aut_{\ReppS(Q)}(R') $ is isomorphic to
 $\Aut_{\ReppS(Q)}(R) $ via conjugation by $\alpha$, so the choice of
 representing element $R$ occurring in $ \coprod_{R/\simeq }
 \Aut_{\ReppS(Q)}(R)$ does not alter $\mathcal G(Q)$.  \par Observe
 first that if two representations satisfy $R=(X,\Phi)\simeq
 R'=(X',\Phi') $, the invertibility of the natural transformation
 implies that the objects $X_v$ and $X'_v$ for each vertex $v$ underly
 the same Bratteli network data up to a permutation of the Bratteli
 network parameters.  W.l.o.g. we can choose then $X=X'$ (should this
 not be the case, if $X \neq X'$ but still $R\simeq R'$, that group
 results in translating by permutations of the Bratteli parameters
 appearing in Corollary \ref{coro:Automorphisms}; pictorially our
 natural transformations look like in the right panel of Figure
 \ref{fig:NatTrans} instead of like the middle panel there). This
 means, these natural transformations are determined by a Bratteli
 network and by the automorphism groups at the vertices as in
 \eqref{gauge_groupofQ}.
 \end{proof}

From the description given in Corollary \ref{coro:Automorphisms}
(therein only the underling set), we can parametrise gauge
transformations by
\begin{align} \label{sigma_g}
(\sigma, g) = \{(\sigma_v,g_v): \sigma_v \in \Sym( \bn_v,\br_v)  \text{ and }  g_v \in   \uni(\bn_v ) \}_{v\in Q_0}
\end{align}
This notation will be useful in the proof of the next result, which is
partially based on that of \cite[Prop. 13]{MvS}, but which has
a fundamentally different output (cf.
Rem.~\ref{rem:GaugeGroups_differences}).
 \begin{lemma}[The gauge group action $\mathcal{G}(Q) \curvearrowright \Rep_\pS(Q) $]
 \label{lem:GaugegroupAction}
Let $R=(X,\Phi) \in \ReppS(Q)$ and let $\mathcal{G}_v$ denote the automorphisms
of  the prespectral triple  at  $v\in Q_0$, $\mathcal G_v:=\Aut_\pS( X_v)$.
If $X_v$ has parameters $ (\bn_v,\br_v)$ as described by Lemma
\ref{thm:WedderburnArtinConsequence},  then
 \begin{align} \label{gauge_semidirectprod}
\mathcal{G}_v =\Sym( \bn_v,\br_v)\ltimes \uni(\bn_v) \quad v\in Q_0\,,
\end{align}
in terms of which the gauge group of the quiver reads
\begin{align}\label{gauge_groupofQ}
\mathcal{G}(Q):=
\coprodsub {\text{Bratteli}\\ \text{networks} \\ (\bn_Q, \br_Q ) } \prod_{v\in
  Q_0 } \mathcal{G}_v \,. \end{align}
It acts on the space of representations as
follows. We parametrise the morphisms
$\Phi=(\Brat, U)$  as in \eqref{ParametrisationbyBratteliUnit}
and call $\Phi'=(\Brat',U')$ the morphism acted on  by $(\sigma, g)$ as in \eqref{sigma_g}. If we let
\begin{align*}
(\Brat'_e,U'_e) &:=\big(  (\Brat,U )^{(\sigma, g)}  \big)_e, \qquad  (\sigma,g ) \in
\mathcal{G}\,, \text{ for each $e\in Q_1$,}
\end{align*}
 then the transformed Bratteli diagrams and unitarities are given by
\begin{align} \Brat'_e = \sigma\targ \circ \Brat_e \circ \sigma\inv\sour   ,\qquad \text{ and }\qquad% \\
U_e' =g\targ \cdot \sigma\targ ( U_e) \cdot  \sigma\targ \big( \Brat_e (g\inv\sour) \big)\,. \label{gauge_Action_formeln}
\end{align}
The dot is the product in $\mathcal U (A_{t(e)}) $, and $\circ$ composition,  of course.
 \end{lemma}

 \begin{proof} Given any $e\in Q_1$,
let  $w=s(e)$ and $y=t(e)$.
Since we are looking for
automorphisms, we set $X_w'=X_w$ and $X_y'=X_y$  in the leftmost diagram in Figure \ref{fig:NatTrans} for the path $p=e$.
 Assuming $(\Upsilon_{v})_{v\in Q_0}\in\mathcal{G}(Q)$
one gets the gauge transformation rule
\begin{align}\Phi_e'= \Upsilon_y \circ \Phi_e\circ \Upsilon_w\inv\,  \qquad
\text{ so } \qquad \phi_e'= \phi_y\circ \phi_e \circ \phi_w\inv
\,, \label{gauge_Phi_transform}\end{align} where $ \Phi_e=(\phi_e,L_e)
$ and $\Upsilon_v=(\phi_v,L_v)$.  If we instead use the
parametrisation \eqref{sigma_g} for the gauge transformations,
$(\Upsilon_{ v})_{v\in Q_0}=(\sigma_{v}, g_v)_{v\in Q_0}$, where
$\sigma_v\in \Sym(\bn_v,\br_v)$ is the Bratteli (permutation) diagram
and $g_v\in \uni(\bn_v)$ for each vertex $v$, we get for $a_w\in A_w$,
\begin{align} \nonumber
 \phi_e'(a_w) & = (\Adj {g_y} \circ\, \sigma_y) \circ ( \Adj U_e \circ
 \Brat_e) \circ (\Adj g_w\inv \circ \sigma_w\inv) (a_w) \\\nonumber &=
 (\Adj {g_y} \circ\, \sigma_y) \circ  \Adj [U_e \cdot \Brat_e
   (g_w\inv ) ] (\Brat_e\circ \sigma_w\inv) (a_w) \\ &= \Adj \big\{
 g_y \cdot \sigma_y [U_e \cdot \Brat_e (g_w\inv ) ]\big\} ( \sigma_y
 \circ \Brat_e\circ \sigma_w\inv) (a_w)
 \label{gauge_formula_trans}
\end{align}
where one uses that $\sigma_w, \Brat_e$ and $ \sigma_y$
are $*$-algebra morphism; in particular,
$\Brat_e \Adj g_w\inv = \Adj [\Brat_e ( g\inv_w)]$
and similar relations satisfied by the permutations $\sigma_w$ and $\sigma_y$.
Since the \textsc{rhs} of Eq.~\eqref{gauge_formula_trans}
must be of the form $\Adj U'_e \circ \Brat_e' (a_w)$, one can uniquely read off
the transformation rule \eqref{gauge_Action_formeln}.

To derive \eqref{gauge_semidirectprod}, we gauge transform twice as
$\Phi_e\stackrel{\Upsilon_y }\mapsto \Phi_e\stackrel{\Xi_y }{\mapsto}
\Phi''_e$ at the target $y$ by $\Upsilon_y=(\sigma_y,g_y),
\Xi_y=(\tau_y,h_y) \in\Sym( \bn_y,\br_y)\times \uni(\bn_y)$.  From
\eqref{gauge_formula_trans} it follows that $\Brat''_e=\tau_y \circ
\sigma_y\circ \Brat_e$ as well as $U_e''=h_y\cdot \tau_y(g_y) \cdot
      [(\tau_y\circ \sigma_y) (U_e) ]$, which is described by the
      semidirect product.  The action of $\Upsilon$ and $\Xi$ at the
      source is obtained in an analogous way. \end{proof}

\begin{remark}\label{rem:GaugeGroups_differences}
Let $\mathcal{C}$ denote either $\pS$ or $\mathcal{S}_0$
(cf. Rem.~\ref{rem:kernels}), grasped as target categories.
It is important not to confuse the symmetries $\Aut_\cC X$ of the objects
$X$ of $\cC$ with those of the representation category,
$\Aut _{\Rep Q} (R)$, which is what the gauge group computes (for
running $R$).  In the previous version of this article, as well as in
\cite[Prop. 13]{MvS}, a gauge group is reported that has the form $
\coprod_{X: Q_0 \to \mathcal{C}} \prod_{v\in Q_0}
\Aut_\mathcal{C}(X_v)$. When $\mathcal{C}$ is a category, as those chosen above,
for which $\hom _{\mathcal{C}} (X_{s(e)} , X_{t(e)})$ can be empty for
some edge $e$, the ``$\coprod_{X: Q_0 \to \mathcal{C}}$'' above still
contains maps $X$ that are spurious, in the sense that they do not
yield a $\mathcal{C}$-representation of $Q$ (thereof, in particular,
they do not guarantee that $\hom _{\mathcal{C}} (X_{s(e)} , X_{t(e)})$
is non-empty for all $e\in Q_1$). A map $X^\diamond: Q_0\to \mathcal
K$ such that for some edge, $e^\diamond \in Q_1$, $\hom _{\mathcal{C}}
(X^\diamond_{s(e^\diamond)} , X^\diamond_{t(e^\diamond)}) = \emptyset$
holds, yields (recall Ex.~\ref{ex:addedV2})
\begin{align}
 \label{emptyX}
\prod_{e\in Q_1}
\hom _{\mathcal{C}} (X^\diamond_{s(e)} , X^\diamond_{t(e)}) = \emptyset.
\end{align}
This means that no $R \in \Rep_{\mathcal{C}} Q $ has the
form $R=(X^\diamond, \Phi)$ for any such label, but due to
\eqref{emptyX}, $X^\diamond$ can be listed or not in $\Rep_{\mathcal
  K} Q =\coprod_{X: Q_0 \to \mathcal{C}} \prod_{e\in Q_1} \hom
_{\mathcal{C}} (X_{s(e)} , X_{t(e)}) $.  This way, \textit{the first
  purpose of the Bratteli networks is to} list all those $X: Q_0 \to
\pS $ for which this does not happen, thereby
selecting \textit{actual contributions to $\Rep_{\pS} Q$}. But the true substance of
Bratteli networks is the following: Since the gauge group is defined
as equivalence of representations, spurious labels $X^\diamond$ with
\eqref{emptyX} do not
lead to a representation and must be excluded. We observe that the next
contention is proper:
\begin{align}\label{proper_contention}
\mathcal{G}_\pS(Q)=
\coprodsub {\text{Bratteli}\\ \text{networks} \\ (\bn_Q, \br_Q ) }
\prod_{v\in
  Q_0 } \Aut_{\pS} (X_v (\bn_Q,\br_Q) )
\subsetneq  \coprod_{X: Q_0 \to \pS} \prod_{v\in Q_0} \Aut_\pS(X_v).
\end{align}%
While for $\Rep_\pS(Q)$ it is a matter of choice whether one lists all
maps $X:Q_0\to \pS$, or not, when it comes to the gauge group, such
trivial maps must be excluded.  This is because $\hom$-sets can be
empty but $\Aut$-groups cannot be. Hence the expression in the right
of \eqref{proper_contention} has no mechanism to exclude spurious
labels as $X^\diamond$ above.  \textit{This shows how Bratteli
  networks are not only useful; for sake of determining the gauge
  group, they are quintessential}.  An analogous (not yet reported)
structure to Bratteli networks would be required to determine
$\mathcal G_{\mathcal{S}_0}(Q)$ or $\mathcal G_{*\text{-\textsf{alg}}}(Q)$.
\end{remark}

\noindent
\begin{minipage}{.734\textwidth}
\begin{example}[Why the gauge group requires Bratteli networks]\label{ex:GGdifferences}
Recall Example \ref{ex:addedV2}, where it was seen that the map on the
right $X^\diamond : Q_0 \to \pS$, given by $X_x^\diamond = (A_x,H_x) =
[\M 2 \oplus \M 3 , \C^2 \oplus (\C^3 \otimes \C^3) ]$, by
$X_v^\diamond= [\M 5 \oplus \M 6 , \C^5 \oplus \C^6 ] $ and
$X_z^\diamond = [\M {11},\C^{11}]$, is not a datum for a quiver
$\pS$-representation since $e_3=(z,v)$ does not lift to a
Bratteli \textit{diagram} (thus to no $\pS$-morphism), so it does not
contribute to $\ReppS Q$.
\end{example}
\end{minipage}
~
\begin{minipage}{.2\textwidth}
\[\includegraphics[width=5.5cm]{BratteliGegenbeispiel2}\]
\end{minipage}\\[.3ex]%
This spurious $X^\diamond$ is not the object-label of any $R\in \ReppS
Q$ [i.e. no $R$ there is of the form $R=(X^\diamond, \Phi)$], and
since the gauge group is computed in terms of automorphisms of
representations $R$, the label $X^\diamond$ should not appear in the
gauge group.  Nevertheless it does contribute \[ [ \uni(5)\times
  \uni(6) ]_{v=t(e_1)} \times \uni(11)_{z=t(e_2)} \times [
  \uni(5)\times \uni(6) ]_{v=t(e_3)} \] to the right expression of
\eqref{proper_contention}, showing the proper contention there. After excluding from $ \coprod_{X: Q_0\to
  \pS} $ all those spurious labels that do not yield representations,
one is left the actual gauge group; when this process finishes, the
$\coprod_{X: Q_0\to \pS}$ boils down to a disjoint union over Bratteli
networks.

\begin{theorem}\label{thm:RepQ_modG}
With the group action of Lemma \eqref{lem:GaugegroupAction},
\[
\frac {\ReppS Q}{\mathcal G(Q)} = \coprod _{\substack{\text{Bratteli} \\ \text{networks}\\ (\bn , \br  ) }} \Bigg\{
 \frac{
\,\prod_{e \in
  Q_1 }   \uni \big( \bn_{t(e)}\big)
}{
% \coprodsub { \text{Bratteli networks }(\bn , \br ) }
\prod_{v\in
  Q_0 }  \Sym(\bn_v,\br_v)  \ltimes \uni (\bn_v)   } \Bigg\}\,.
\]
\end{theorem}

\subsection{The spectral triple of a quiver} \label{sec:STofQuiver}

We introduce now a spectral triple and a Dirac operator that is
determined by a representation of $Q$ and by a (graph-)distance $\rho$
on $Q$.  We remark that the addition of the latter is not overly
restrictive since, in the most interesting cases, the action
functional that depends on such Dirac operator will be either fully
independent of the graph-distance $\rho$ and end up depending only on
the holonomies, or, in the worst case, it will be dependent only on
the distance evaluated on the self-loops $(\rho_{v,v})_{v\in Q_0} \in
\R^{\#Q_0}_{\geq 0}$ (see Prop.~\ref{thm:Higgsa_od_kołczanów}).  \par
We now make precise some ideas already commented on in Sec.~
\ref{sec:motivation}. A (finite-dimensional) spectral triple $(A,H,D)$
is by definition a prespectral triple $(A, H) \in \pS$ together with a
self-adjoint operator $D: H\to H$.

\begin{definition}[Spectral triple and $ D_Q(L,\rho)$ for a quiver representation]
Given a quiver $Q$ and a representation in $R\in\ReppS(Q)$,
  $ R=\{   (A_v,H_v)_v, (\phi_e,L_e)  \}_{v\in Q_0,e\in Q_1}$,
  define \begin{align*} A_Q=\bigoplus_{v\in Q_0} A_v \, \quad \text{ and } \quad H_Q=\bigoplus_{v\in Q_0} H_v\,. \end{align*}
  This definition is motivated by applying to $R$ the functor \eqref{FromReptoMods} that yields a path algebra module.
 We construct the third item to get a spectral triple $(A_Q,H_Q,D_Q)$. Given a graph distance $\rho:Q_1\to \R_{>0} \cup \{\infty \} $ on $Q$, the  \textit{Dirac operator $D_Q(L,\rho) $ of a quiver representation} is defined by
\begin{align}
D_Q(L,\rho ): H\to H\,,\qquad
D_Q(L,\rho) = \mathscr{A} \symmetr_Q(b)\,, \label{defD}
\end{align}
(see Eq.~\eqref{symm_weight_matrix_def} for definition of $\mathscr{A}
\symmetr_Q$), with weights $b$ given by scaling the unitarities $L$ by
the graph distance inverse $\rho\inv$, that is
\begin{align*}
b:Q_1\to \hom_{\C\text{\tiny -Vect}}(H_{s(\,\, \cdot \,\, )},H_{t(\,\,
  \cdot \,\,)}) \quad b_e: H\sour \to H\targ, \quad
b_e:=\frac{1}{\rho(e)} L_e, \text{ for each } e\in Q_1\,.
\end{align*}
(Thus $D_Q(L,\rho)$ is the Hadamard product $ \mathscr{A}
\symmetr_Q(L)\odot \rho\inv$ when $Q$ does not have multiple edges.)
On a lattice, $\rho$ is the lattice spacing.
If no graph distance $\rho$ is mentioned, we will tacitly understand that the Dirac operator
is the one in \eqref{defD} with $\rho(e)=1$ for all edges $e$, and denote it by $D_Q(L)$.
For a loop $p$ based at $v$, we define the \textit{holonomy} in this
spectral triple as $\hol_b$, that is\begin{align*} \hol_{\rho\inv L}(p) =
\prod_{e\in p}^{\to} \frac{1}{\rho(e)} L_e : H_{v} \to H_{v} \,,%
\end{align*}
where the arrow on the product sign emphasises the coincidence of the
order of the product of the unitarities with the order the edges
appear in $p$ using the same criterion as in Eq.
\eqref{Wilson_and_Holonomy}. In this context, given a closed path $p$
based at $v$, the Wilson loop is obtained by tracing this holonomy,
$\Wils= \Tr_{v} \circ \hol_{\rho\inv L}$, where we started
to abbreviate $\Tr_v$ for traces of operators $H_v\to H_v$.
\end{definition}
% \newpage
\begin{figure}[h!]
\noindent
\begin{minipage}{.624\textwidth}
\begin{example}[Explicit Dirac operator for a representation]
For $m\in \N$ let $n=m^2$ and consider the quiver $\mathcal T_n$ with vertex
set $(\mathcal T_n)_0=\{1,\ldots, m^2\}$ and edges set
\begin{align*} \notag
(\mathcal T_n )_1& =\{ (v, v+1 ) : v=1,\ldots, m^2-1 \} \\
 & \cup \{ (v, v+m ) : v=1,\ldots, m^2-m \}\notag \\
 & \cup \{ (m(m-1)+v,v ) : v=1,2,\ldots, m   \}   \cup  \{ (1,m^2) \} \notag
\end{align*} For instance, the quiver on the right represents
$\mathcal T_{16}$.  If we distribute the vertices on a square lattice
then $v$-th vertex is source of the arrow pointing to $v+1$ and also
of an arrow with target $v+m$. When we are near to the right or upper
boundary of the square, compactification of it to a torus yields the
last line of edges. \end{example}
\end{minipage}
~
\begin{minipage}{.37\textwidth}\captionsetup{font=small}
% \begin{figure}
%  \begin{align*} \nonumber\qquad
\hspace{0.452cm}\runter{\includegraphics[width=5.17cm]{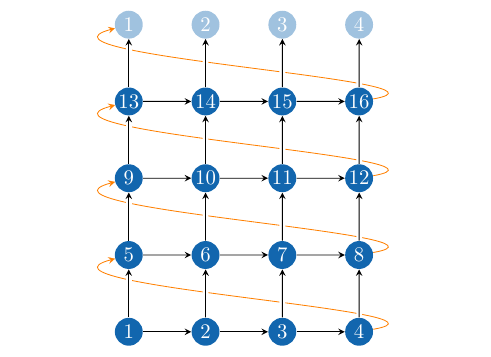}}
% \end{align*}
\caption{$\mathcal T_{16}$ (opaque
vertices reappear from below;
arrow colour plays no role).}
% \end{figure}
\end{minipage}\vspace{-2ex}
\end{figure}
%   \newpage
% \noindent
The connectivity at the right boundary being `shifted by one' in the
vertical direction (instead of naturally connecting $n$ with $1$, $2n$
with $n+1$, etc.) is for sake of providing in explicit way the Dirac
matrix.  For a representation $R=(X,\Phi) =\big ( (A,H), (\phi,L)
\big) :Q \to \pS$, the Dirac operator $D_{T_{n}} (L)$ reads (writing
$L_e=L_{v,w}$ for $e=(v,w)$ with $v<w$)
 \begin{align*} \nonumber D_{\mathcal T_n} (L)=\!\!\!\runter{
\includegraphics[width=.72\paperwidth]{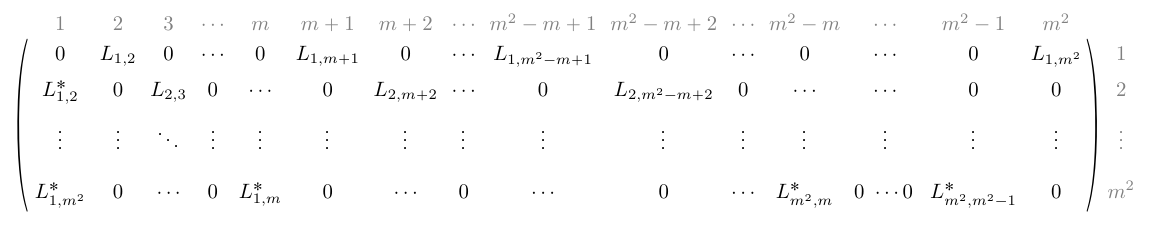} }
\end{align*}

\begin{remark}\label{rem:infdims}
When the triple $(A_Q,H_Q,D_Q)$ becomes infinite-dimensional (for
instance, when $\#Q_0$ is infinite) two axioms that trivially hold
when $\dim A_Q $ and $\dim H_Q $ are finite need to be verified in
order to still call this triple spectral.  From
\cite[Def. 2.6]{Chamseddine:2006ep} a spectral triple\footnote{The
author thanks John Barrett for the question that motivated this
remark. } $(A,H,D)$ should also satisfy the boundedness of the
operators $[D,\lambda (a)] $ on $H$ (for all $a \in A$, being
$\lambda: A \curvearrowright H$ the $*$-action) and the compactness of
the resolvent of $D$---and even more operators on $H$ and further
axioms are needed to prove the reconstruction theorem
\cite{Reconstr}. It remains as perspective to prove both properties
for the triple $(A_Q,H_Q,D_Q)$ above, which for a finite quiver hold,
for infinite quivers.
\end{remark}

\section{Representations on lattice quivers}\label{sec:lattice_paths}
In order to exploit the path formula \eqref{pathcomposition} we count
paths without weights first ($b_e=1,e\in Q_1)$. For $v\in Q_0, k\in
\N$, we use the following notation
\begin{align*}
\mathcal N_k (Q,v) = \{w \in Q_0 : \text {shortest path from $v$ to
  $w$ has length exactly $k$}\}\,.
\end{align*}
The cardinality of this set is independent of the vertex $v$ if $Q $
is a lattice. If this lattice is $d$-dimensional and periodic with
$Q_0= (\Z/m\Z)^d$, we write $h_d(k):= \# \mathcal N_k (Q,v)$, that is,
the volume of the $L^1$-sphere of radius $k$ (see
Fig. \ref{fig:L1norm}).  These integers form the coordination sequence
$\{h_d(1),h_d(2),h_d(3),\ldots\}$, which for an orthogonal lattice
($d=2$ square, $d=3$ cubic,...) can be obtained from the
Harer-Zagier\footnote{In the Harer-Zagier formula
\cite[Prop. 3.2.10]{GraphsonSurfaces},
\begin{align*}\notag
 \bigg(\frac{1+z}{1-z}\bigg)^d = 1+ 2z d + 2 z \sum_{k\geq 1 }
 \frac{   T_k(d) }{(2k-1)!!} z^{k}
\end{align*}\thispagestyle{empty}the polynomial $T_k(d)=\sum_{2g\leq k} c_g(k)d^{k+1-2g}\,$
that generates the number $c_g(k)$
of pairs of sides of an $2k$-agon
that yield a  genus $g$ surface,
have the following integral representation
 \begin{align*}\notag
     T_{k}(d) :=d^{k-1} \int_{\M d _{\mathrm{s.a.}}}
%  \sum_{\substack{P \in \mathscr A_{m,n}}}
  \mathrm{Tr} (X^{2k}) \,\dif \nu(X)  \,,
  \end{align*}
 where $\dif \nu (X)$ is the normalised Gaussian measure
%\begin{align*} \label{GaussianMeasure}
$\dif \nu(X) = C_d \ee^{-d \Tr \frac {X^2}{2}} \dif{X}\,,$
% \end{align*}
being $\dif X$ the Lebesgue measure on the space of
hermitian $d$ by $d$ matrices.  } generating function
% \begin{figure}[h!]
%  \begin{minipage}{.53\textwidth}
 \begin{align}\label{HZ_lattice}
 \textsc{HZ}_d(z) &:=\bigg[\frac{1+z}{1-z} \bigg]^d =\sum_{k\geq 0 } h_d(k)z^k   \\
% \end{align*}
% \end{example}
% \begin{align*}
\nonumber
&\textstyle=1+ 2 \, d z+ 2 \, d^{2} z^{2}+ \frac{2}{3} \, {\left(2 \, d^{3} + d\right)}  z^{3}
+ \frac{2}{3} \, {\left(d^{4} + 2 \, d^{2}\right)} z^{4}\\ &\textstyle
\nonumber \textstyle+\frac{2}{15} \, {\left(2 \, d^{5} + 10 \, d^{3} + 3 \, d\right)} z^{5}
%  \\ &
 \textstyle+\frac{2}{45} \, {\left(2 \, d^{6} + 20 \, d^{4} + 23 \, d^{2}\right)} z^{6} +\ldots\nonumber
\end{align}
Figure \ref{fig:L1norm}, where lattice points are labelled by the
$L^1$-distance to the vertex $v$, shows to sixth order the first three
series for $d=1,2,3$, explicitly
\begin{subequations}%
 \begin{align}
\textsc{HZ}_1(z)& =1 + 2 z + 2 z^{2} + 2 z^{3} + 2 z^{4} + 2 z^{5} + 2 z^{6} + 2 z^{7} + \ldots  \label{HZ1}\\
\textsc{HZ}_2(z)& =1 + 4 z + 8 z^{2} + 12 z^{3} + 16 z^{4} + 20 z^{5} + 24 z^{6} + 28 z^{7}+\ldots \label{HZ2}\\
\textsc{HZ}_3(z)& =1 + 6 z + 18 z^{2} + 38 z^{3} + 66 z^{4} + 102 z^{5} + 146 z^{6} + 198 z^{7}+\ldots\label{HZ3} \end{align}
\end{subequations}
The very last series is \textsc{A005899} of \cite{oeis_walks_Kn}.
% \end{minipage}~
% \begin{minipage}{.45\textwidth}\centering
\begin{figure}[b]\captionsetup{font=small}
{\centering \includegraphics[width=.75\textwidth]{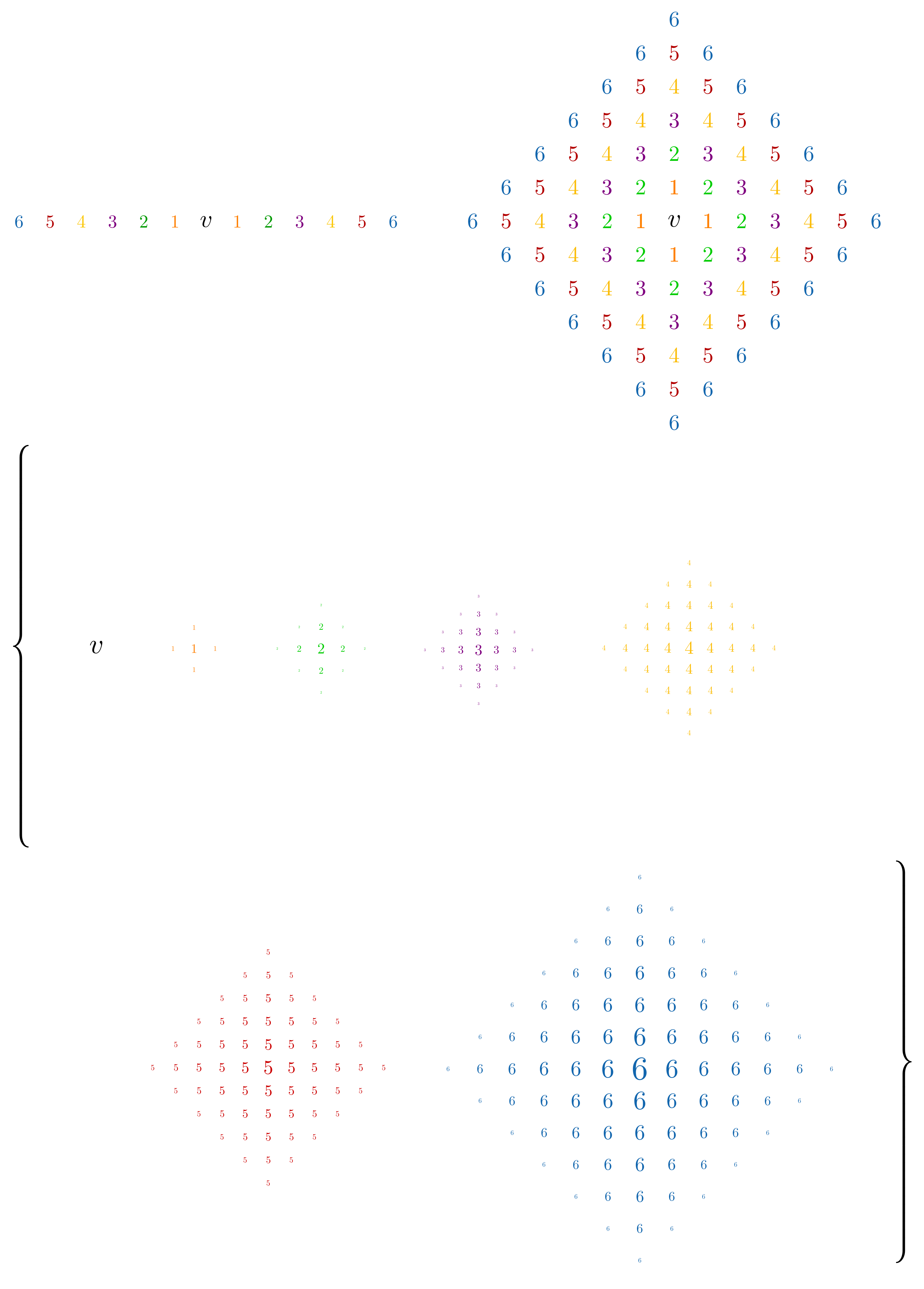}}
\caption{$L^1$-spheres in a $\Z_m^d$-lattice ($m>6$); each integer
  depicts the radius $k=1,\ldots, 6$ of the sphere around $v$ it lies
  on.  \textit{Upper left:} for $d=1$, there are only two possible
  ends of paths of any fixed length, cf . \eqref{HZ1}, for any given
  radius.  \textit{Upper right:} with $d=2$, around $v$ we see four
  points of radius 1, eight of radius 2, $\ldots$, twenty four of
  radius $6$, corresponding to the first six terms of Series
  \eqref{HZ2}.  \textit{In curly braces:} for $d=3$\label{fig:L1norm},
  we see $L^1$-spheres `from above' distributed on an octahedron.
%   ;
%   the
%   larger the displayed number, the more protruding (for a fixed radius, or
%   colour).
  For instance, for radius $6$, the 85 points one
  \textit{sees} have to be doubled, to add those seen from below and
  not shown---except those in the equator (24, shown in tiny) are
  double counted and should be subtracted to get the $2\times 85 -24=
  146$ in agreement with $146=[z^6]\textsc{HZ}_3(z)$ from Series
  \eqref{HZ3}.  }
%  \end{minipage}
 \end{figure}
% \newpage
For $m,d\in \Z_{\geq 2}$ let us define the quiver $Q=T^d_m$ by
$Q_0=(\Z/m\Z)^d$ and by following incidence relations. The outgoing
$t\inv (v)$ edges and incoming edges $ s\inv (v)$ at any $v\in Q_0$
are, by definition of $Q_1$, given by
\noindent
% \begin{minipage}[l]{.65\textwidth}
\begin{align*}
s\inv (v) &= \{ (v,w) : w= v_{+i} := v + \mathbf e_i ,\quad i=1,\ldots, d \} \\
 t\inv (v) & =  \{ (w,v) : w=v_{-i} := v - \mathbf e_i, \quad  i=1,\ldots, d \}\end{align*}
where $\mathbf e_i$ is the $i$-th standard basic vector, $\mathbf e_1=(1,0,\ldots, 0),
\mathbf e_2=(0,1,0\ldots, 0),
\ldots, \mathbf e_d=(0,\ldots, 0,1)$ and the sum is component-wise on $\Z_m$.
For $i,j\in\{ 1,\ldots,d\}$ with $i\neq j$ %and $\epsilon_1,\epsilon_2\in \{1,2\}$,0
the \textit{plaquettes}  on $Q^\star$
are an important type of length-$4$ loops based at $v$ defined by
% \begin{align*}f^{\epsilon_1,\epsilon_2}_{i,j} (v)=[v,v+\epsilon_1 \mathbf e_i, v+\epsilon_2\mathbf e_j, v\mp \mathbf e_i, v- \mathbf e_i]\,,\end{align*}
\begin{align*}P_{\pm i,+ j} (v)& =(v,v\pm \mathbf e_i, v+\mathbf e_j, v\mp \mathbf e_i, v- \mathbf e_j )\,, \\
 P_{\pm i, -j} (v)& = (v,v\pm \mathbf e_i, v-\mathbf e_j, v\mp \mathbf e_i, v+ \mathbf e_j )\,.
\end{align*}
% \end{minipage}\qquad
% \begin{minipage}{.3\textwidth}
\begin{figure}[htb]
 \begin{align*}
\runter{
\includegraphics[width=4.2cm]{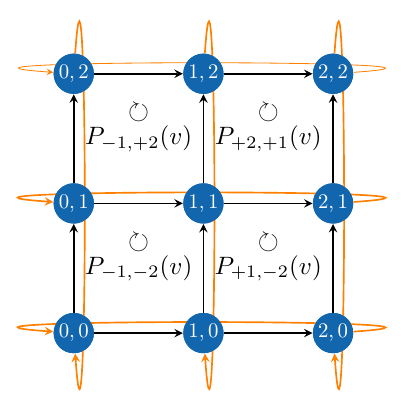}}
%   \qquad\qquad\quad
 \end{align*}
\caption{$T_3^2$ and some of its plaquettes around $v=(1,1)$ are shown.\label{T23}}%
\end{figure}%
% \end{minipage}\\[1ex]
\noindent%
For $d=2$ and $m=3$ the quiver\footnote{Since plaquettes are paths in
$Q^\star$ we should show this augmented quiver, but to simplify
visualisation we show $Q$ and expect the reader to add, for each edge,
another in the opposite direction.} $T_m^d$ and some plaquettes are
shown in Figure \ref{T23}.%\newpage
\begin{proposition}\label{thm:TracesTorusD}
Let $N\in\Z_{>0}$. Given a representation $R=\{ (A,H)_v, (\phi,L)_e \}
_ {v\in Q_0, e \in Q_1}\in \Rep^N_{\pS}(Q)$ of $Q={T^d_m}$,
abbreviating $D=D_{T^d_m} (L) $, and setting the natural edge distance
$\rho:Q_1\to \R$ to be the constant lattice spacing $\rho(e)=1$ for
each edge, if $m\geq 5$ and $d\geq 2$, one obtains
\begin{subequations}
 \begin{align}
 \Tr (D^0) &= m^d N\,,
 \\
 \label{Tr2DTorus}
\Tr (D ^2)  & = m^d \times 2d \times N \,, \\
 \label{Tr4DTorus}
\Tr (D^4)& = (8d^2-2d) m^d N  + \sum_{v \in \Z_m^d} \sumsub{P \in
  \Omega_v(T^d_m) \\ {\mathrm{\tiny plaquettes}}} \Tr_v \big[ \hol_L (P)
  \big]\,.
\end{align}%
\end{subequations}%
\end{proposition}%
Since $m>1$, there is no length-1 loop, so $\Tr (D_Q(L))=0$. In fact,
since the lattice is rectangular, $\Tr (D_Q(L)^k)=0$ for odd $k$,
since $k< m$ forbids loops of odd length (for $k\geq m$ a straight
path through the vertices $v, v+\mathbf e_i, v+2\mathbf e_i, \ldots,
v+m \mathbf e_i=v$ could be a loop, e.g. if $m$ is odd).
\begin{proof}
We use $Q={T^d_m}$ to simplify notation.  First, $\Tr D^0$ is the
trace of the identity on $H=\oplus_{v\in Q_0} H_v$, which amounts to
$\Tr(1) = \sum_{v\in Q_0} \dim H_v= \#Q_0 N $.  For higher powers of
the Dirac operator, we can use the path formula of Corollary
\ref{thm:Wksym}. If $k>0$ is even, no loop goes outside $\mathcal
N_{k/2} ^d(v)$ and we can ignore paths that exceed this radius.  \par
Now we observe that the set of length-$2$ loops at $v$ is in bijection
with the set $\mathcal N_1^d(v)$ of the nearest neighbours of
$v$. Then $\Tr [D_Q(L)^2] =\sum_{v\in Q_0 }\sum_{w \in \mathcal
  N_1^d(v)} \Tr (L_{v,w}L_{w,v} ) = \sum_{v\in Q_0 }\sum_{w \in
  \mathcal N_1^d(v)} \Tr (1_{H_v})$ by \S ~\ref{sec:geometry} and
$2d=\# \mathcal N_1^d(v)$ by \eqref{HZ_lattice}, for any
$v$. Eq.~\ref{Tr2DTorus} follows. \par

We turn to the case $k=4$. First, split the length-$4$ loops $p \in
\Omega_v Q$ into two cases, according to `how far' a path goes from
$v$, namely the largest $\varrho=\varrho(p)$ for which $p$ intersects
$\mathcal{N}_\varrho(v)$.
%  Write  $p=(e_1,e_2,e_3,e_4) = [v,w_1,w_2,w_3,v]$
\begin{itemize}
 \item[]\textit{Case I.} If $\varrho(p)=1$. Write
   $p=[e_1,e_2,e_3,e_4]$ for $e_j\in Q_1$. Notice that $e_2$ must be
   $\bar e_1$ since otherwise $\varrho>1$, so $e_3$ starts at $v$ and
   it can end anywhere in $\mathcal{N}_1(v)$; but $e_4=\bar e_3$ again
   since $\varrho=1$. Clearly this defines a bijection
   $\mathcal{N}_1(v)\times \mathcal{N}_1(v)\to
   \Omega_v(Q)|_{\ell=2,\max \varrho =1}$. In any of these cases, the
   $h_d(1)^2=(2d)^2$ paths $p$ with $\varrho(p)=1$ contribute
   \begin{align*}\Wils(p)= \Tr ( L_{e_1} L_{e_2} L_{e_3} L_{e_4}) =
   \Tr ( L_{e_1} L_{e_1}^* L_{e_3} L_{e_3^*}) = N\, . \end{align*}
   \item[]\textit{Case II.} If $\varrho(p)=2$ then there is a unique
     $w\in \mathcal{N}^d_2(v)$ reached by $p$ (else $\ell(p)>4$),
     cf. Figure \ref{fig:pathstypeslength4}. We now count the loops at
     $v$ that can contain $w$:
 \begin{itemize}
  \item[(a)] If $e_1$ is parallel to $e_2$ then there is a unique loop at
    $v$ containing $w$, and since $\varrho=2$ one has $e_3=\bar e_2$ and
    $e_4=\bar e_1$. Thus $e_1$ fully determines
    $\#\mathcal{N}^d_1(v)=h_1(d)=2d$ such paths, all of which contribute
    $\Wils(p)=N$.
\item[(b)] Else, $p=[e_1,e_2,e_3,e_4]$ and $e_2$ is not parallel to
    $e_1$. Suppose that $w$ is visited by the loop $p=[e_1,e_2,\bar e_2,
    \bar e_1]$ (if not, $p$ is a plaquette, see below). Since the number of such points $w$
     equals those on the $L^1$-ball of radius-$\varrho=2$ minus
     those reached by $e_1$ parallel to $e_2$,
         there are $2\times [h_2(d)-h_1(d)]=2\times 2d(d-1)$ loops (the
         factor of $2$ due to two ways of reaching the same point $w$)
         all of them  contributing $\Wils(p)=N$.
\item[(c)] Or else, $p=[e_1,e_2,\bar e_1, \bar e_2]$ is a plaquette, and
  then $\Wils(p) = \Tr
  (L_{e_1}L_{e_2}L_{e_1}^*L_{e_2}^*)=\Tr \hol_L(p)$. Clearly, if $w$ lies
  in the path $p=P_{i,j} $, then so does in $\bar p=P_{j,i}$, both
  of which are different paths (swapping clockwise with
  counter-clockwise).  This yields an extra $2$ factor and $2
  [h_2(d)-h_1(d ) ] =4d(d-1)$ plaquettes. \qedhere
 \end{itemize}
\begin{figure}[h!]\captionsetup{font=small}
\raisebox{0ex}{\includegraphics[width=.70317\textwidth]{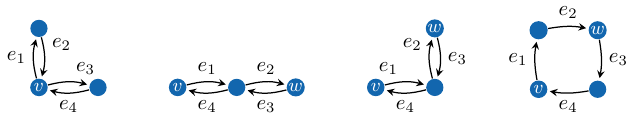}}
\caption{\label{fig:pathstypeslength4} \textit{From L to R}: Case
  $\varrho=1$, and Cases $\varrho=2$, (a), (b) and (c) in the proof of
  Prop.~\ref{thm:TracesTorusD}.  }
\end{figure}
\end{itemize}
% The constant terms in Formula \ref{Tr4DTorus} \gris{fill this ...}
\end{proof}

\subsection{Adding self-loops}

Given a quiver $Q$ let us denote by $\mathring Q $ or, ad libitum, by
$Q^\circ$ the quiver obtained from $Q$ by adding to the edge set a
self-loop (denoted $o_v$) for each vertex of $Q$. So $
Q^\circ=(Q_0,Q_1\dot\cup \{o_v: v\in Q_0\} )$.  From
Figure \ref{fig:LoopedQ} one can see that the notation's origin is
nothing else than mnemonics.
Due to \eqref{augmentation}, augmentation does not modify self-loops, so it commutes with
adding them and we can define $Q\starcirc:= (Q^\star) ^\circ =  (Q^\circ) ^\star  $.
% \par
% \noindent
% \begin{example}
For instance, from the Jordan quiver $J=\!\raisebox{-5pt}{\includegraphics{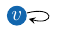}}$ one obtains
$
 J^\circ=\!\! \!\raisebox{-3pt}{\includegraphics{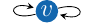}}\!\!$ and (recall Ex.~\ref{ex:Jordan1}) $J^\star=J$,
so $J\starcirc=J^\circ$. \par
% \end{example}
%
Given a path $p=[e_1,\ldots,e_k]$ in $Q$, provided $v=s{(e_j)}$, one
can insert $o_v$ just before $e_j$ as follows \begin{align} (v^\vee_j
  p) :=[e_1,\ldots, e_{j-1},o_v,{e_j},\ldots, e_k ], \qquad 1\leq j
  \leq k,
\label{deletion_vertex}\end{align}  cf. Figure \ref{fig:LoopedPath}. If $t(e_k)=v$, one can extend $p$ by $o_v$,
$v_{k+1}^\vee {p}= o_v \cdot p $ as in Figure \ref{fig:LoopedEnd}, but
for closed paths the case $j=k+1$ coincides with the case $j=1$.
Since only closed paths contribute to the traces of the Dirac
operator, we ignore paths $p$ with $t(p)\neq s(p)$ from now on.  If
the condition $v=s{(e_j)}$ is not met or if $j>\ell(p)$, $(v^\vee_j p)
:=E_{s(p)}$ is the trivial, $0$-length loop at $s(p)$. We thus get
maps $v^\vee_j: \Omega Q \to \Omega Q^\circ$ for $j\in \Zpos$.  \par
On their `support', those maps increase the path length; other maps
exist that decrease it, also by $1$. Going in the opposite direction,
we define $v^\wedge_j: \Omega Q^\circ \to \Omega Q^\circ$. If the
$j$-th edge of a path $p$ is a self-loop $o_v$, let $v^ \wedge_j (p)$
be the path obtained from $p$ by omitting $o_v$, and otherwise let $v^
\wedge_j (p)$ be a trivial path. Equivalently,
\begin{align}  (v^\wedge _j  p)   := \begin{cases} [e_1,\ldots,
e_{j-1},\widehat{e}_j ,{e_{j+1}},\ldots, e_k] \, & \text { if $e_j$ is
      a self-loop and if $1\leq j \leq k$,}\\ E_{s(p)} & \text {
      otherwise.}
\end{cases}
\label{insertion_vertex}
\end{align}

\begin{figure}[htb]
\subfloat[\label{fig:LoopedQ} By definition, $\mathring Q_0=Q_0$. With
  gray $Q_1$, and with black the new edges added to $Q_1$ to form
  $\mathring Q_1$.]{\includegraphics[height=2.2cm]{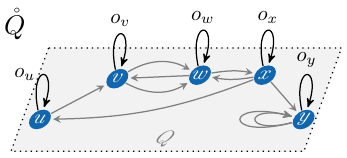}}
\qquad\,\, \subfloat[A path $p$ in $Q$ and the $j$-th insertion
  $v_j^\vee p$ of a self-loop edge $o_v$, $1\leq j\leq \ell(p)$.
\label{fig:LoopedPath}]{\quad\includegraphics[height=2.35cm]{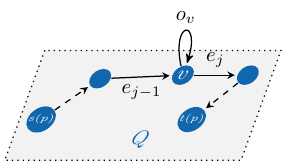}}
\qquad\,\,
\subfloat[A path $p$ in $Q$ and the corresponding $(k+1)$-th insertion
$v_{k+1}^\vee (p)$, where $v=t(e)$.
\label{fig:LoopedEnd}]{\quad\includegraphics[height=1.70cm]{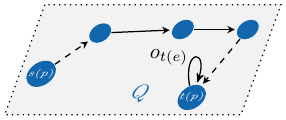}}
\caption{Illustrating the meaning of $Q^\circ$ and $v_j^\vee (p)$. \label{fig:LoopedBoth}}
\end{figure}

It will be convenient to use the following multi-index notation. Given
two ordered $q$-tuples, one $I=(i_1,\ldots,i_q) $ of indices
$i_1,\ldots, i_q\in \Zpos$, assumed to be increasingly ordered ($a<b
\Rightarrow i_a <i_b$), and another of vertices, $\mathbf
v=(v_1,\ldots, v_q) \in Q_0^q$, we define
\begin{align}
\bv_I^\vee:=( v_{q})^\vee_{i_q} \circ  ( v_{q-1}) ^\vee _{i_{q-1}} \circ \cdots
\circ ( v_{2})^\vee_{i_2} \circ
( v_{1})^\vee_{i_1}\,. \label{insertion_multiindex}
\end{align}
Whenever this does not yield the trivial path $E_{s(p)}$,
we get from a  loop $p\in \Omega Q$
of length $k$,
a new loop $\bv _I^\vee(p) \in \Omega Q^\circ$ of length $k+q$. If
the insertion \eqref{insertion_multiindex} yields the trivial
path, it can be ignored, since the next formulae select those of
positive length (recall Def. \ref{def:Wilson_and_Holonomy}).

\begin{example}[Notation for path insertions] \label{ex:notation_insertions}
Consider $Q$ with $Q_0=\{v,w\} $  and a single edge $e=(v,w)$.
All possible insertions to get a length-$4$ loop
in $Q^\circ$ out of the only  length-$2$ path $p\in \Omega Q$ based at $v$, $p=[e,\bar e]$,
are
\begin{align}\label{pathinsertions}
\runter{
\includegraphics[width=.8\textwidth]{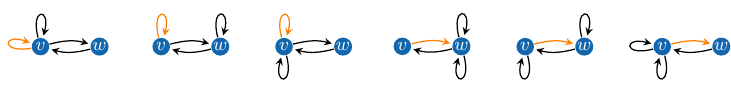}}% \vspace{-3ex}
\end{align}
where the orange arrow is the first edge in the path (one later takes
the order induced by clockwise orientation, in case that more arrows
start from the same vertex). They correspond to
\begin{align*} \notag
(v,v)^{\,\,\vee}_{1,2}(p)\qquad\!
(v,w)^{\,\,\vee}_{1,3}(p)\qquad\!
(v,v)^{\,\,\vee}_{1,4}(p)\qquad\!
(w,w)^{\,\,\vee}_{2,3}(p)\qquad\!
(w,v)^{\,\,\vee}_{2,4}(p)\qquad
(v,v)^{\,\,\vee}_{3,4}(p)\,.
\end{align*}
\end{example}

\begin{lemma}\label{thm:Higgsa_od_kołczanów}
Let $R $ be a representation of a quiver $Q$ in prespectral triples
and let $\mathring R=(\mathring X_v,\mathring \Phi_e) \in
\ReppS(\mathring Q)$ extend $R$ in the sense that it coincides on all
objects, and for morphisms it satisfies $(\Brat_e, L_e)=\Phi(e)
=\mathring\Phi(e)=(\mathring \Brat_e,\mathring L_e)$ for each $e\in
Q$. Let $\mathring D=D_{\kol Q}(\kol L,\rho)$, where $\rho$ is a
graph-distance on $Q^\circ$.  Contributions to $\Tr (\mathring D^k)$
in terms of insertions of loops into existing paths in $Q$ read
\begin{align}
\Tr(\mathring D^k)= \Tr(D^k)+\sum_{q=1}^{k- 1}\,
\Bigg\{\sumsub{p\in \Omega Q \\ \ell (p)=k-q }
\,\,\,
\sumsub{I \in \{1,\ldots, k\}^q \\  i_a < i_b \,\text{when } a < b  }\,\,\,
\sumsub{\bv \in  Q_0^q  }
 \Wils{  [ \bv^\vee _I (p)  ]} \Bigg\}
+\sum_{v\in Q_0} \Tr_v (\varphi_v^k )
\label{SpectAction_shift_selfloops}
\end{align}
with $I=(i_1,\ldots, i_q)$, where for each $v$, $\varphi_v: H_v\to
H_v$ is the (self-adjoint) operator
\begin{align}\varphi_v:= \sumsub{e \in Q^\circ_1 \setminus Q_1 \\
t(e)=v=s(e)
}\frac{ 1}{{\rho(e)}}(\kol  L _{e}+ \kol  L _{e}^*) \,. \label{Higgs_scalar_def}\end{align}
If $p=[e_1,\ldots, e_{k-l}] \in \Omega Q$ and $\alpha(i) := i - \# \{j\in I: j<i \}$, the
Wilson loop reads
\begin{align*}
\Wils[ \bv_I^\vee (p)  ]= \Tr_{H_{s(p)}} ( b_1 b_2 \cdots b_k )
 \qquad
 \text{ with }
 b_i= \begin{cases}\displaystyle
    \frac{  L _{e_{\alpha(i)} }}{{\rho(e_{\alpha(i)})}}  &  i \notin I ,\\[3ex]
 \displaystyle  \frac{ 1}{{\rho(o_{v_i})}}(\kol  L _{o_{v_i}}+ \kol  L _{o_{v_i}}^*)   & i \in I.
     \end{cases}
\end{align*}
\end{lemma}
\begin{remark}\label{rem:graphdistance}
Observe that if $Q$ itself did not have self-loops at $v$,
%\begin{align*}
$\varphi_v = \frac{ 1}{{\rho(o_v)}}(\kol L _{o_v}+ \kol L _{o_v}^*)$
holds.  Also we anticipate that $\varphi$ plays the role of a
discretised scalar (whose aim is to model a Higgs) field. The graph
distance $\rho$ could play a role in the construction of physical
models, as the spectrum of $\{\varphi_v\}_{v \in Q_0}$ with the tacit
constant graph distance $\rho \equiv 1$ is in $[-2,2]$, whilst a
general graph distance allows a spectrum for $\{\varphi_v\}_{v \in
  Q_0}$ in $[-2/ \rho_{\text{\tiny min}}, 2 / \rho_{\text{\tiny min}}
]$, with $\rho_{\text{\tiny min}} =\min_{v\in Q_0} \rho( o_v) $.
Since $Q$ is to be understood as a `microscopic' model for space (for
which one requires to endow $Q$ with additional information like
edge-coloring, or other decorations), this minimum distance being
small delocalizes the spectrum extending it effectively to the real
line.
\end{remark}

\begin{proof}
Given a loop $p\in \Omega Q^{\star \!\!\!\circ}$ of length $k$, let
$q$ be the number of self-loops $o_v\in Q^\circ_1\setminus Q_1$ in
$p$. If $q=0$, then $p$ is a path in $Q$. The sum of all such paths is
precisely $ \Tr(D^k)$. In the other extreme, $q=k$, the path consists
of self-loops, but since $p$ has then no edge of $Q$, all such
self-loops are based at the same vertex.  All paths with $q=k$ yield $
\sum_{v\in Q_0} \Tr_v (\kol D_{v,v} ^k)$, which can be re-expressed in
terms of $ \varphi_v$ as in Eq.~\eqref{SpectAction_shift_selfloops} if
$\varphi_v$ is given by Eq.~\eqref{Higgs_scalar_def}. \par For the
rest of the cases, $0<q<k$, observe that even though $ (p, I,\bv )
\mapsto \bv_I^\vee (p)$ is not a bijection from $\Omega Q|_{\ell=k-q}
\times \{ I \in \{1,\ldots, k\}^q : i_a < i_b \,\text{ if } a < b \}
\times Q_0^q \to \Omega Q^\circ |_{\ell=k}$, the support of both sets
(i.e. the subdomains not yielding the trivial paths) coincides, and
that is enough for \eqref{SpectAction_shift_selfloops} to hold.
Indeed, for $p\in \Omega Q|_{\ell=k-q}$ (thus a nontrivial path),
either $ \bv_I^\wedge \circ \bv_I^\vee (p) =p$ or $\bv_I^\wedge \circ
\bv_I^\vee (p)$ is the trivial path (and does not contribute to the
sum). Conversely any length-$k$ path $ p^\circ$ in $Q^\circ$ can be
gained from a unique set of parameters $I$ and $\bv$ that correspond
to an insertion in a unique path $p \in \Omega Q|_{\ell=k-q}$ through
$p^\circ=\bv_I^\vee(p)$ with $p$ given by $p=\bv_I^\wedge
p^\circ$. The uniqueness guarantees no double nor multiple counting
while splitting the sum over $p\in \Omega Q^\circ|_{\ell=k}$ in the
three sums in Eq.~\eqref{SpectAction_shift_selfloops}. To obtain the
reported holonomies, if $i \in I$ this means that we inserted a loop
$o_{v_i}$, and in this case we obtain for the holonomy the sum
$(L+L^*)/\rho$ evaluated at $o_{v_i}$. Else, if $i\notin I$, we have
to evaluate $L/\rho$ at an edge that has been shifted by as many
self-loop insertions from indices of $I$ have been performed before,
resulting in $L_{e_{\alpha(i)} } /\rho(e_{\alpha(i)})$ for $\alpha(i)
= i - \# \{j\in I: j<i \}$ indeed. \end{proof}

\begin{notation} \label{not:abbrev}
The following abbreviations, $\mathbf e_{-i} = -\mathbf e_{i}$ for
$i>0$ and $v_j=v+a \cdot \be_j$ for $|j| \in \{1,\ldots, d\}$, where
$a$ is the lattice space, will be practical; further, if a
representation is implicit, $L_j(v)=L_{(v,v_j)}$.
\end{notation}

\begin{proposition}\label{thm:YMH_nonexplicitform}
Consider $O_m^d:=\kol T_m^d$ and a representation $\kol R \in \Rep^N_\pS (O_m^d)$.  Let $\kol D$ be the Dirac operator with respect to $\kol R$. Then for $m\geq 4, d\geq 2$,
\begin{subequations}%
\begin{align*}%
\Tr(\kol D^0)&=m^d \times N \,, \\
\Tr(\kol D^2)&= m^d \times (2d)\times N  +\sum_{v \in \Z_m^d}  \Tr _v (\varphi_v^2)  \,,  \\
\Tr(\kol D^4)&=  (8d^2-2d) m^d N
+ \sum_{v \in \Z_m^d} \bigg\{   \sumsub{P \in
  \Omega_v(T^d_m) \\ {\mathrm{\tiny plaquettes}}} \Tr \big[ \hol_L (P)\big]
   + \Tr _v (\varphi_v^4) \\
 & \qquad\qquad   \qquad\qquad+  8 d \Tr_v (\varphi_v^2)
  + \sumsub{ |j|=1 \\ j\in\Z }^d 2\Tr_v [  \varphi_v L_{ j} (v) \varphi_{v_j} L_{ j}^* ( v_j )  ]  \bigg\}\notag  \,,
\end{align*}%
\end{subequations}%
cf. Notation \eqref{not:abbrev} (also observe that  $j$ can be negative in the last sum).
\end{proposition}
% Observe that in the
% Also, we shall later simplify the last two quadratic terms in $\varphi$ under a single sum over vertices (but keep them split by now).
\begin{proof}
Since $\Tr(\kol D^0)$ only sees the vertices, the result is the same
as for $\Tr(D^0)$, where $D$ the Dirac operator of the restriction of
$\kol R$ to $O_m^d$.  \par For positive powers $k$, we use the path
formula to find $\Tr (\kol D^k)$.  If $k=2$, then notice that the
integer $0<q<k$ in the Formula \eqref{SpectAction_shift_selfloops}
cannot be $q=1$, since removing one self-loop cannot yield a closed
path (which should consist of $k-q$ edges of the lattice $T^d_m$
without self-loops).  Thus tracing the square of the Dirac operator
splits only as the contributions from $D$ on $T^d_m$ and contributions
purely of self-loops, which is the new term $\sum_{v\in Q_0 } \Tr _v
(\varphi_v^2)$.  \par For $k=4$ the middle sum over path insertions,
Formula \eqref{SpectAction_shift_selfloops}, forces $q=2$.  To
evaluate this term, we introduce some notation.  Given any path $p\in
\Omega (O_m^d)^\star=\Omega (T^d_m) \starcirc $ and two vertices
$v_1,v_2\in \Z^d_m$, let $\delta_{v_1,v_2} p =E_{s(p)}$ (the trivial
path at $s(p)$) if $v_1\neq v_2$ and $\delta_{v_1,v_2} p=p$ if
$v_1=v_2$.  Analogously, for two integers, $i_1,i_2$, to wit
$\delta_{i_1,i_2} p $ is the trivial path at $s(p)$ if those integers
do not coincide and the path itself it they do. Contributions to the
mentioned $q=2$ term come from length-$2$ paths (since $2=k-q$ here)
on the lattice without self-loops. For a fixed vertex $v$, any such
path is of the form $p=(v,v_{\pm j})$ with $v_{\pm j} = v \pm \mathbf
e_1$ for some $j=1,\ldots, d$.  We fix the path $p=[e,\bar e]$ with
$e=(v,w_{\pm j})$, keeping in mind the dependence on the sign and on
$j$.  For $w$ a nearest neighbour of $v$, i.e. $w=v_{\pm j}$, we have
for any $(x,y)\in \Z^d_m\times \Z^d_m $, and $I=(i_1,i_2)$
\begin{align*} \notag
(x,y)^\vee_I  ( v,w ) & =
\delta_{x,v}\delta_{y,v }\delta_{1,i_1}\delta_{2,i_2} [ o_v ,o_v,  e,  \bar e ]
+
\delta_{x,v}\delta_{y,w}
\delta_{1,i_1}\delta_{3,i_2}
[ o_v , e, o_w ,\bar e]  \\
\notag
& +
\delta_{x,v}\delta_{y,v} \delta_{1,i_1}\delta_{4,i_2}
[ o_v, e, \bar e  ,o_v ]
 + \delta_{x,w}\delta_{y,w} \delta_{2,i_1}\delta_{3,i_2}
 [ e,o_w,o_w, \bar e ]
 \\\notag
&+\delta_{x,w}\delta_{y,v}\delta_{2,i_1}\delta_{4,i_2} [ e,o_w,\bar  e,o_v ]+
\delta_{x,v}\delta_{y,v} \delta_{3,i_1}\delta_{4,i_2}[ e, \bar e, o_v,o_v ]\,,
\end{align*}
% \begin{align*} \notag
% (x,y)^\vee_I  ( v,{w_{\pm j}} ) & =
% \delta_{x,v}\delta_{y,v }\delta_{1,i_1}\delta_{2,i_2} [ o_v ,o_v,  e,  \bar e ]
% +
% \delta_{x,v}\delta_{y,{w_{\pm j}}}
% \delta_{1,i_1}\delta_{3,i_2}
% [ o_v , e, o_{w_{\pm j}} ,\bar e]  \\
% \notag
% & +
% \delta_{x,v}\delta_{y,v} \delta_{1,i_1}\delta_{4,i_2}
% [ o_v, e, \bar e  ,o_v ]
%  + \delta_{x,{w_{\pm j}}}\delta_{y,{w_{\pm j}}} \delta_{2,i_1}\delta_{3,i_2}
%  [ e,o_{w_{\pm j}},o_{w_{\pm j}}, \bar e ]
%  \\\notag
% &+\delta_{x,{w_{\pm j}}}\delta_{y,v}\delta_{2,i_1}\delta_{4,i_2} [ e,o_{w_{\pm j}},\bar  e,o_v ]+
% \delta_{x,v}\delta_{y,v} \delta_{3,i_1}\delta_{4,i_2}[ e, \bar e, o_v,o_v ]\,,
% \end{align*}
which can be matched to the paths in \eqref{pathinsertions} in that order.
Thus the sum in question reads
\begin{align*} \nonumber
\sumsub{p\in \Omega Q \\ \ell (p)=2 }
\,\,\,
\sumsub{I \in \{1,\ldots, k\}^2 \\  i_1 < i_2 }\,\,\,
\sumsub{ (x,y) \in  \Z_m^d\times \Z_m^d}
 \Wils{ [ (x,y)^\vee_I(p)]}
 =\sumsub{j=1\\ \varepsilon = \pm }^d
 3 W_1(v,j, \varepsilon)
 +2 W_2(v,j,\varepsilon) +
 W_3(v,j,\varepsilon).
  \end{align*}
Here, if $e_j=(v,v+\mathbf e_j)$, the six paths yield three types of contribution:
 \begin{align*}
 W_1(v,j,\pm) & =\Tr (\varphi^2_v),\quad &&&
 W_2(v,j,+) & =\Tr (\varphi_v L_{e_j} \varphi_{v_{+j}} L_{e_j}^* ),\quad \\
 W_2(v,j,-) &=\Tr (\varphi_v L_{e_j}^* \varphi_{v_{-j}} L_{e_j} ),\quad & &&
 W_3(v,j,\pm)& =\Tr (\varphi_{v_{\pm j}} ^2 ).
\end{align*}
The $ W_2(v,j,\pm)$ terms contribute the sum for $ \pm j \in
\{1,2,\ldots, d\} $ (expressed above over $|j|$) as listed at the very
end of the last formula of this proposition.  Translation invariance
allows to group the $3W_1$ and $W_3$ terms, and since $j$ takes $2d$
values, one obtains the $8d \Tr \varphi_v^2$ term. At risk of being
redundant, we recall that the quartic term there, $\Tr \varphi_v^4$,
is monic since there is a unique length-$4$ path consisting of
self-loops (or recall cf. Lemma \ref{thm:Higgsa_od_kołczanów} with
$k=4$).
\end{proof}
 \section{Applications to gauge theory}\label{sec:physics}

 \subsection{From the lattice to the theory in the continuum}\label{sec:QgaugeTh}

%  \begin{definition}[From \cite{Chamseddine:1996zu}]
    Given a $\Lambda>0$, the (bosonic) \textit{spectral action}
    \cite{Chamseddine:1996zu} at \textit{scale} $\Lambda$ of a given
    finite spectral triple $(A,H,D)$ is $\Tr [ f(D/\Lambda) ]$.  Its
    evaluation is possible in the finite-dimensional case for a
    polynomial $f:\R\to \R$ (in contrast to other settings where $f$
    is required to be a bump function), which is that of our quivers.
%   However, the existence of
%   the partition function later on will require the decay of $\exp(-\Tr_H [ f(D/\Lambda) ] )$,
%   so we keep in mind that $f\to \infty$ as the absolute value of its argument becomes large.
\begin{lemma}\label{thm:YMH_almostexplicit}
With some abuse of notation let $Q=O^d_m$ denote also the quiver with
lattice space $a>0$, that is $\rho(e)=a $ on edges $e$ that are not
self-loops, instead of the unit lattice space (and otherwise under the
same assumptions) of Proposition \ref{thm:YMH_nonexplicitform}.  For
$f(x)=\sum_{k=0}^4 f_k x^k$ the spectral action of a quiver
representation of $Q$ at the scale $\Lambda=1/a$ is real-valued and
reads \vspace{-.4ex}
  \begin{align} \label{costam_phiquadrat}
\Tr f(D/\Lambda)&=  m^d N [  f_ 0 + 2d\cdot f_2 +(8 d^2 -2 d) f_ 4]  +  f_4  \sum_{v \in \Z_m^d}
\sumsub{ p \in \Omega_v(T^d_m) \\ {\mathrm{
      plaquettes}}}  \Tr_v\big[ \hol_L (p)\big]     \\[-1.5ex] & + a^2
\sum_{v \in \Z_m^d} \Tr _v \bigg\{  ( f_2 +8 d\cdot f_4 ) \varphi_v^2  + 2 f_4  \sum_{j=1}^d \big [   \varphi_v L_{j}(v) \varphi_{v_j} L_{j}(v)^*
+  \varphi_v L_{ {-j}} (v)\varphi_{v_{-j}} L_{  {-j}}(v_{-j}) ^*\big] \bigg\} \notag \\[1ex]
&
   +a^4  \sum_{v \in \Z_m^d}  f_4\Tr _v (\varphi_v^4)\,. \hspace{7.84cm}  [\text{cf. Notation \eqref{not:abbrev}}]
 \notag
\end{align}%
\end{lemma}

\begin{proof}This is a consequence of rewriting Proposition \ref{thm:YMH_nonexplicitform}
after replacement of the new lattice space, and reordering.  It is
noteworthy that the self-adjointness of the argument of the traces in
the spectral action is not explicit for the two terms that mix
$\varphi$ with $L$, namely $\Tr(\varphi_v L_{e_j} \varphi_{v_j}
L_{e_j}^* )$ and $ \Tr (\varphi_v L_{e_{-j}} \varphi_{v_{-j}} L_{
  e_{-j}}^*)$.  Although these are in general not mutual hermitian
conjugates, the action is real, since the respective hermitian
conjugate terms come from paths based at translated vertices $v_j= v+a
\mathbf e_j$ and $v_{-j}= v-a \mathbf e_j$, when the outer sum takes
those values; see Figure \ref{fig:hermicity_Hgaugesector}. \qedhere
\begin{figure}
 \begin{align*}
% [o_v, e_j,o_{v_j},\bar e_{j}] \,\,\,&\leftrightarrow \text{ black path starting at $v$ shown in}&
 %\\[-2ex]
% [o_{v_j},\bar e_{j},o_v, e_j] \,\,\,&\leftrightarrow \text{ black path that starts at $v+\mathbf e_j$ in}&
% [o_{v_{-j}},\bar e_{j},o_v, e_j] \,\,\,&\leftrightarrow \text{ orange path starting at $v-\mathbf e_j$ in}&
\runter{\includegraphics[width=3cm]{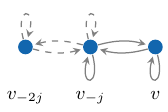}} \qquad\qquad
\runter{\includegraphics[width=3cm]{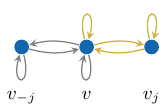}}\qquad\qquad
\runter{\includegraphics[width=3cm]{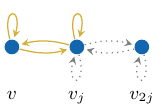}}%\\[-2ex]
\notag
\end{align*}
\caption{ Each of the three panels shows length-$4$ paths based at the
  vertex of the middle.  For a fixed $j=1,\ldots,d$, pairing the two
  paths on the $j$-th axis around $v$ (in the middle panel) with the
  rightmost and leftmost solid-colour paths of the same colour of the
  other panels, implies that the spectral action is explicitly real
  valued. The left and right figures appear when the sum over $Q_0$
  takes the values $v_{\pm n j}=v\pm n a \mathbf e_j$ (self-loops
  directed upwards or downwards only for sake of visualisation).
\label{fig:hermicity_Hgaugesector}}
\end{figure}

\end{proof}
\noindent
We analyse the spectral action in the limit of small $a$ and large $m$
of the vertices $(a \Z/m a\Z)^d$ of the quiver $O_m^d$ is now
obtained. This torus $\mathbb T^d$ conventionally will have volume
$(am)^d$, without $2\pi$-factors.

\begin{theorem}[The smooth limit] \label{thm:YMH_explicitandsmooth}
For $Q= O_m^d$ assume that the representation $R \in \Rep^N_\pS(Q) $
that yields the spectral action in Lemma \ref{thm:YMH_almostexplicit}
has $(A_{v_0},H_{v_0})=(\M N,\C^N)$ for some $v_0\in Q_0$.  Then in
the limits of the lattice space $a\to 0^+ $ and the vertex number
$m\to \infty$, that action reads
\begin{align}
\Tr f(D/\Lambda) &= \Lambda^{d}N [ f_0 + 2d\cdot f_2 +(12d^2-6d) f_4 ]
\vol(\mathbb T^d ) % \notag \\[1ex] &
- 2 \Lambda^{d-4} f_4
\int_{\mathbb T^d} \sum_{i,j=1}^d\Tr_N (F_{ij}^2 ) \dif ^d
x \label{smooth_SpAct} \\ &-\int_{\mathbb T^d} \Tr_N \Big\{
\Lambda^{d-4 } 2 f_4 \sum_{j=1}^d (D_j h)^2 -\Lambda^{d-2} ( f_2 +16
d\cdot f_4 ) h^2 - \Lambda^{d-4} f_4 h^4 \Big \}\dif ^d x+ O
(\Lambda^{d-5}) \,,\notag
\end{align}where $\Lambda:=1/a$,  and $h, A_j$ are hermitian $\M{N}$-valued fields on $\mathbb T^d$. Here
\begin{align}
\label{defDjFij_smooth}
D_j h := \partial_j h + [\ii A_j, h]\quad\text{ and }\quad F_{ij}:=
\partial_i A_j - \partial_j A_i + \ii [A_i, A_j] \qquad \text{ for }
i,j=1,\ldots, d\,.
\end{align}
\end{theorem}

\begin{proof}
Since the vertices of $Q$ are modeled by $\Z_m^d$, using the mod-$m$
arithmetic one has always both-way paths between any $v\in Q_0$ and
$v_0$ in $Q$ (there is no necessity to see these paths in the
augmented quiver, where of course, they exist). Hence the only
possible Bratteli network has $\bn_v =N,\br_v=1$ at each vertex, due
to Eq.~\eqref{both-way}, so $(A_v,H_v)=(\M N, \C^N)$ for all $v\in
Q_0$.  Define for each $v$ a hermitian matrix $A_j(v) \in \M{N}$ by
\begin{align} \exp [ \ii a A_j (v) ]:= L_{ (v, v+a  \be_j) } (v) =:L_{j} (v) ,  \text{  where }   j>0
% \qquad F_{ij}= \partial_i A_j - \partial_j A_i + \ii  [A_i, A_j]
\, . \label{defAj}\end{align} where the rightmost is an abbreviation.
In order to identify the gauge-Higgs sector define first $ \Delta_i
\varphi_v:= \tfrac1a [ \varphi(v+ a\be_i ) -\varphi(v) ]$. Comparing
the two paths of different colors, \vspace{-2.002ex}
\[  \runter{\includegraphics[width=2.2cm]{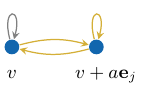}} \]
motivates us to define the covariant derivative
of matrix fields $\varphi$ in the lattice by
\begin{align*}
  (D_j\latt \varphi   ) (v) :\!& =  \frac1a \big \{ L_{j} (v) \varphi_{v + a\be_j } [ L_{j}  (v )]^* - \varphi_v \big\} \in \End( H_v ) ,
 \\   [ (D_j\latt \varphi)  (v) ]^*& =
 D_j\latt \varphi (v)   .
\end{align*}
Despite being different from known lattice-derivatives
(cf. \cite[Secs. 3 \& 6]{montvay_muenster_1994}) by Eq.~\eqref{defAj}
this definition is seen to be correct, as in the $a \to 0$ limit, one
obtains \[ (D_j\latt \varphi ) (v) = \Delta_j \varphi_v + [ \ii
  A_j(v),\varphi_v] + O (a^2)\,, \] which yields in the smooth limit
the covariant derivative in Eq.~\eqref{defDjFij_smooth}.  This allows
us to identify, recalling that $v_{\pm j} $ abbreviates $v \pm a
\be_j$,
 \begin{align*}
\mathfrak
r_{v,j}&:=2 \Tr(\varphi_v L_{ j} (v) \varphi_{v_j} L_{ j}(v)^* ) , \\ % \qquad &&&
\mathfrak l_{v,j} & := 2\Tr \big [\varphi_v L_{-j}(v) \varphi_{v_{-j}} L_{
  {-j}}^*(v_{-j}) \big]
  =  2 \overline{\Tr_{v_{-j}} \big  (
\varphi_{v_{-j}} L_{ j} (v_{-j})\varphi_v L_{ j}^*(v_{-j})
\big) }
= \overline{
  \mathfrak r_{v{-a \be_j}, j }} =   \mathfrak r_{v{-a \be_j}, j } ,
 \end{align*}
 with their smooth counterparts. The cyclicity of the trace and the
 facts that $H_{v}$ and $ H_{v_{-j}}$ are isomorphic, along with
 Property \eqref{Wbar=barW} have been used in the last line (which
 also can be deduced from the middle panel in
 Fig.~\ref{fig:hermicity_Hgaugesector} that depicts the traced terms
 in $\mathfrak r_{v,j}$ and $\mathfrak l_{v,j}$) along with the fact
 that these terms are of the form $\Tr ( \mathfrak a u \mathfrak b
 u^*) $, for $\mathfrak a =\mathfrak a^*, \mathfrak b=\mathfrak b^*
 \in\M N, u \in \uni(N)$, and thus real. One finds
\begin{align}
  \Tr \big [ (  D_j\latt \varphi   )^* (  D_j\latt \varphi)  \big]  =
    \Tr \big [
     \varphi_{\balitabaja +j}   ^2    - 2\varphi L_j^* \varphi_{\balitabaja+j} L _j +\varphi^2     \big ]  = -\mathfrak r_{\balitabaja, j }  +
       \Tr (
     \varphi_{\balitabaja+j}   ^2  + \varphi^2 ), \qquad
\label{gdziescostam}
     \end{align}
    where $\varphi_{\balitabaja+j}$ is understood to be $\varphi_{v_j}$ after evaluation of this expression
     at a vertex $v$.
Since, as seen above,
$\mathfrak l_{v_{+j},j} = \overline{\mathfrak {r}_{v,j}} =\mathfrak {r}_{v,j}$,
we can rename the sum indices  (which we also do for
the quadratic terms in $\varphi$ and yields an extra $2$ factor) thanks to translation invariance
and conclude that the $\mathfrak{l}$ terms just duplicate  the $\mathfrak{r}$ terms, and thanks to
Eq.~\eqref{gdziescostam}, find that
\[
\sum_v \mathfrak r_{v,j} + \mathfrak l_{v,j} = \sum_v \mathfrak
r_{v,j} +\sum_v \mathfrak r_{v_- ,j} =\sum_v 2\mathfrak r_{v,j} =
\sum_v \Big\{ 4 \Tr (\varphi^2_v) - 2\Tr_v \big [ ( D_j\latt \varphi
  (v) )^* ( D_j\latt \varphi(v) ) \big] \Big\}.\] The second term is
the gauge-Higgs kinetic term reported above, as $a \to 0^+$ and when
$m$ is taken large. This has a $\Lambda^{d-4}= a ^{4-d}$ scaling, with
the $4$ coming from the $(D/\Lambda)^4$ in the spectral action, and
the $d$ from taking the sum to an integral. \par

Next, we clean up the quadratic terms, which rewrite as
a single sum over vertices as follows
 \begin{align*}a^2 (f_2 + 16  d \cdot f_4)\textstyle
 \sum_{w\in Q_0}  \Tr _w (\varphi_w^2).
 \end{align*}
For fixed $w$, the factor $16 d $ above is composed of the explicit
initial $8d$ in the second line of \eqref{costam_phiquadrat}; another
contribution of $ 2 d\times 4 \Tr _w (\varphi_w^2)$, where the $2d =
\sum_{|j|}1 $. The whole polynomial contribution of $\varphi$ is
\begin{align*}\textstyle
 \sum_v  \Tr_v \big [ a^2 ( f_2 +16 d \cdot f_4 ) \varphi_v^2 + a^4f_4 \varphi_v^4\big] \to a^{2-d} ( f_2 +16  d\cdot  f_4 ) \textstyle\int  h^2+  a^{4-d} f_4 \int h^4\,,
\end{align*}
as $a\to 0^+$, $m\to \infty$.
Considering the contribution of a single plaquette based at $v$
with $0<i<j$,
\begin{align*} \notag
\hol_L P_{i,j} = \ee^{ \ii a A_i (v) }   \ee^{ \ii a A_j (v+ a \be_i ) }
 \ee^{ \ii a A_i (v+ a \be_i +a \be_j  ) }
  \ee^{ \ii a A_j (v+ a \be_j  ) }
  = \ee^{ \ii a A_i (v) }   \ee^{ \ii a A_j (v+ a \be_i ) }
 \ee^{- \ii a A_i (v+ a \be_j  ) }
  \ee^{ - \ii a A_j (v) }
\end{align*}
and using Backer-Campbell-Hausdorff formula to simplify the first two
factors and the last two, one finds $ \hol_L P_{i,j}(v) = \exp [ \ii
  a^2 F_{ij}(v) ] $ ignoring $O(a^3)$ in the exponent, where, letting
 $a
\Delta_i A_j(v):= A_j(v+ a\be_i ) -A_j(v)$ we defined
  $F_{ij}:= \Delta _i A _{j} -\Delta _j A _{i} + \ii [A_i,A_j]= -F_{ji}$. In terms of this discrete
version of the curvature of a connection matrix $A$ we rewrite the
plaquettes' contribution:

 \noindent
\begin{align}  \nonumber
\sum_{v \in \Z_m^d}
\sumsub{ P \in \Omega_v(T^d_m) \\ {\mathrm{
      plaquettes}}}  \Tr_v \circ \hol_L ( P)    & = 4\sum_{v  }
\sumsub{ i  , j =1 \\ i < j   } ^ d  \Tr_v\big( \hol_L P_{i,j} + \hol_L  P_{j,i}  \big)
% \qquad  \raisebox{-45pt}{
% \includegraphics{plaquettes_tlumaczenie}}  \\[-.8cm]
\\
&= 4 \sum_{v} \frac12 \sum_{i\neq j }  \Tr_v (2 - a^4 F_{ij}F_{ij}) \,.   \label{end_F}\end{align}
% &= 4 \sum_{v} \frac12 \sum_{i\neq j } \Tr_v [2 + a^4 F_{ij}F_{j,i}] \,. \end{align*}
%  \end{minipage}~
% %  \begin{minipage}{.292\textwidth}
%  \begin{align}\label{end_F}\raisebox{-28.5pt}{
%  \includegraphics{plaquettes_tlumaczenie}}\end{align}
%  \end{minipage}
\noindent%
\begin{minipage}{.79\textwidth}
In the first equality, we split the sum into anti-clockwise and
clockwise plaquettes. Also the sum on the \textsc{lhs} over plaquettes
with $|i|, | j |=1,\dots, d$ is rephrased only as plaquettes with
$i,j>0$, depicted in the right in solid colour, in terms of $A_i(v)$
and $A_j(v)$ while the neighbouring vertices `borrow' those with $i$
or $j$ negative. This in turn implies the factor of $4$, since the
same Wilson loop value $\Tr_v \circ \hol_L P_{i,j}(v)$ reappears also
when the sum over vertices hits $v +a \be_i$, $v +a \be_i$ and $v +a
\be_i + a\be_j$ and the plaquettes there are in the negative quadrants
(hatched). From Eq.~\eqref{end_F} the constant contribution is $
4d(d-1)$ and the $F_{ij}$-dependent part is promoted to the functional
of a smooth field in the limit, obtaining the pure Yang-Mills part of
\eqref{smooth_SpAct}.  From Proposition \ref{thm:YMH_nonexplicitform},
the previous constant is increased by $8d^2-2d$, yielding $12d^2-6d$.
\qedhere \end{minipage} \quad
\begin{minipage}{.192\textwidth}\centering
\begin{align*}%\label{end_F}
\raisebox{-28.5pt}{
 \includegraphics{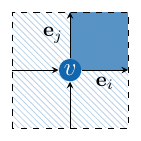}}\end{align*} \end{minipage}
\end{proof}

\begin{remark} \label{rem:uwaga_Higgs}
As the only `phenomenological' comment here, we exercise critique on
the previous model and address solutions: The coefficient for the
above pure Euclidean Yang-Mills sector should be $-1/4g^2$ for $g \in
\R^+$ the gauge $\uni(N)$-Yang-Mills coupling, which fixes
$f_4=1/8g^2$.  This creates two issues.  First, $\mathrm{sgn }f_4$ is
wrong for the gauge-Higgs kinetic term.  This is, at least for even
$d$, an easy-to-solve problem by tensoring the present Hilbert
space\footnote{Almost-commutative manifolds (see \cite{WvSbook}) do
have such Clifford module, an infinite-dimensional algebra.  The
algebra of matrix geometries, on the other hand, is
finite-dimensional, but their Hilbert spaces and Dirac operators have
a different form (in particular, the coefficients of the
gamma-matrices are commutators or anticommutators,
cf. \cite[App. A]{Barrett:2015foa}). I thank Harald Grosse for the
question that motivated this remark.}  by a Clifford $\mathbb Cl(d)
$-module $\mathbb S$ and placing the chirality operator $\gamma_o$
(`generalised $\gamma_5$') along the self-loops $o$
(i.e. $\varphi_v\to \gamma_o\otimes\varphi_v, v\in Q_0$) and flat
gamma matrices along the lattice-edges (i.e. $L_e\to \gamma_i \otimes
L_e$ for $e$ parallel to $\be_i$).  This provides the correct sign,
via anticommutation relations $\gamma_o \gamma_i + \gamma_i
\gamma_o=0$, $i=1,\ldots, d$, but requires the spectral triples to be
even (the corresponding category is beyond the scope of this article).
Secondly, the constraint $f_4=1/8g^2$ affects the Higgs sector too: a
solution is to use the graph distance $\rho(e)= a/ \tau$ ($\tau >0$)
if $s(e)=t(e)$ and else $\rho(e)=a$, for a lattice edge
$e$. Pictorially,
% a patch of $O_m^d$ looks then like
\vspace{-1.05ex}
\[
\includegraphics{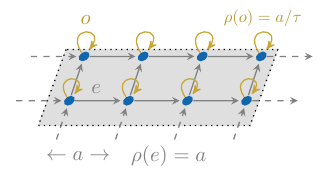} \vspace{-2ex}
\]
This rescales the Higgs as $h\to \tau \cdot h $,
whereby $\tau$ becomes a function of $g$. This is because
the kinetic gauge-Higgs term should, for hermitian $h$, have a
 prefactor $1/2$, so $\tau =  {\sqrt 2 \cdot g} $ (after
 the chirality takes care of the sign).
Since $f_2$ is still free, it controls the coefficient $\tau^2 ( f_2 +16df_4)=
2g^2f_2 +4d$
 of $h^2$, and can yield the known Higgs potential with two minima.
The analysis is left for future work.
 \end{remark}

\subsection{Improved gauge theory}
Symanzik's programme known as \textit{improved gauge theory} consists
of systematically correcting the Wilsonian action for gauge theory,
which we met above in terms of plaquettes.  One of his aims was the
enhancement of the speed of convergence from the lattice to the
continuum (for instance, if a certain observable is known to converge
with an error of O$(a^2)$ in the cutoff $a\inv$, the improved model
should achieve O$(a^4)$ or better). Some models are known to be
capable of this but, to the best of our knowledge, there is no
geometrical explanation of the terms that have to be added to achieve
such improvements. Below, Proposition \ref{thm:D6} derives directly
one model and opens a perspective for new investigations to obtain
further corrections in the framework of Connes' spectral formalism in
noncommutative geometry (which in view of Theorem
\ref{thm:YMH_explicitandsmooth} also could include a Higgs field).
With a unit lattice space, recall the formula $ L_{j} (v)= L_{ (v, v+
  \mathrm{sgn}(j) \be_{|j|}) } (v) $, which allows a convenient
extension to $j\in \{-d,1-d,\ldots,d-1,d\} \setminus \{0\}$.  We can
compute this matrix by \eqref{defD}, $L_{ {-j} }(v) =[L_{j} (v-
  \mathbf e_j)]^*$ which in terms of the gauge field $A_j$ reads
$L_{-j} (v)= \ee^{-\ii A_j(v- \mathbf e_j)} $, so the index in $A_j$
is always positive.  \par

\begin{definition}\label{def:hol_lenght6paths}
Given a representation of $T^d_m$, one lets for any of its vertices
$v$ and $i,j,l \in \{-d,1-d,\ldots,d-1,d\} \setminus \{0\}$ pairwise
different in absolute value
\begin{align*}
\hol _L(\runterhalb{\includegraphics[height=10pt]{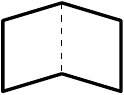}},v ; j,i, l) & :=  \Tr_v (L_j L_i L_{-j} L_l L_{-i} L_{-l} ) \\
\hol_L (\runterhalb{\includegraphics[height=14.3pt]{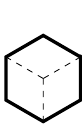}}\,,v ; i,j,l) & :=  \Tr_v (L_i L_j L_{l} L_{-i} L_{-j} L_{-l} ) \end{align*}
where we are using a shorthand notation which, in
the case of the first path, should read
\begin{align*}
\hol _L (\runterhalb{\includegraphics[height=10pt]{hol_puerta}},v ; j,i, l) & =  \Tr_v
[ L_j (v)  L_i (v+\be_j ) L_{-j}(v+\be_i+\be_{j} )
L_l (v+\be_i  )L_{-i} (v+\be_{l} +\be_i) L_{-l}  (v+ \be_{l})
] \notag
\end{align*}
when it is written in full. We can afford ourselves the abbreviation
above in the lattice by choosing the vertices to evaluate $L_j$
in the unique way that makes the path (in this case
$\runterhalb{\includegraphics[height=10pt]{hol_puerta}}$) well-defined. We also define \begin{align*} \hol_L
(\runterhalb{\includegraphics[height=7.5pt]{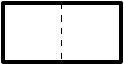}},v ; i,j) &: =
\Tr_v ( L_i L_i L_j L_{-i} L_{-i}L_{-j} ) \\ \hol_L
\big(\,\runterhalb{\includegraphics[height=11pt]{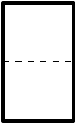}} \,,v ; i ,j \big) &
:= \Tr_v ( L_i L_j L_j L_{-i} L_{-j}L_{-j} )
\end{align*}
\end{definition}

% \subsection{Symanzik improved gauge theory}

\begin{proposition}[The Spectral Action yields the Weisz-Wohlert cells of `improved gauge theory', \cite{Weisz:1983bn}] \label{thm:D6}
Let $d\geq 3$ and $m>6$. On $Q=(T_m^d)^\star$ and $D= D_Q(L)$,

\begin{align*} \notag
\Tr(D^6 ) & = \sum_{v \in \Z^d_m} \Bigg \{
\theta_{0} (d) N+
\sumsub{| i | , |  j |  =1  \\
|i| \neq |j| }^d \Big[
 \theta_{\includegraphics[height=4pt]{hol_recthor}} (d)
\hol (\runterhalb{\includegraphics[height=7.5pt]{hol_recthor}},v ;  i,j)
+
 \theta_{\includegraphics[height=6.70pt]{hol_rectver}} (d)
\hol (\runterhalb{\includegraphics[height=11pt]{hol_rectver}},v ;  i,j)
+
\theta_{\square} (d)
\hol_L ( \Box ,v ;  i,j) \Big] \\
& \quad \quad +\sumsub{| i | , |  j | , |l|  =1  \\
 \text{pairwise different} }^d
 \Big[\theta_{\includegraphics[height=5pt]{hol_puerta}} (d)
 \hol _L(\runterhalb{\includegraphics[height=10pt]{hol_puerta}},v ; j,i, l)
 +
 \theta_{\includegraphics[height=8.3pt]{hol_hex}} (d)
 \hol_L (\runterhalb{\includegraphics[height=14.3pt]{hol_hex}},v ; i,j,l)
 \Big]
 \Bigg \}\,,\notag
\end{align*}
where the $\theta$-coefficients are polynomials in $d$ given by
\begin{align*}
\theta_{0}(d) & =4(10d^3 - 11d^2 + 6d)\\
\theta_{\square} (d)& = 12 d \\
\theta_{\includegraphics[height=5pt]{hol_puerta}} (d)& = 3 \\
\theta_{\includegraphics[height=8.3pt]{hol_hex}} (d) & = 1 =
 \theta_{\includegraphics[height=3.8pt]{hol_recthor}} (d)  =  \theta_{\includegraphics[height=6.70pt]{hol_rectver}} (d) \,.
\end{align*}
\end{proposition}
A further Symanzik-type improvement beyond that by Weisz and Wohlert
could be build by adding higher powers (say up to $2k$) of $D$. That
is why we called $ \theta_{\includegraphics[height=5pt]{hol_puerta}}
(d)$ and the other constants a `polynomial', as for such improvement
$\theta_{\includegraphics[height=8.3pt]{hol_hex}}\hp{2k}
,\theta_{\includegraphics[height=5pt]{hol_puerta}}\hp{2k} ,
\theta_{\includegraphics[height=3.8pt]{hol_recthor}}\hp{2k} $ and $
\theta_{\includegraphics[height=6.70pt]{hol_rectver}}\hp{2k} $ will
have degree $2k-3$.
 \begin{proof}
Observe the following trichotomy for any loop $p\in \Omega_v(Q^\star)$
of length $6$:
\begin{figure} \small \notag
  \begin{align*} \notag \text{Type I}&
\runter{\includegraphics[width=11.5cm]{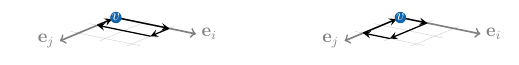}}   \\ \notag  \text{Type II$_a$ }
 &\runter{ \includegraphics[width=11.3cm]{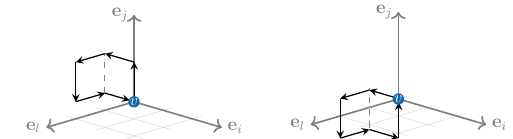 } } \\
\\[2ex] \text{Type II$_b$ }&\runter{ \includegraphics[width=5.5cm]{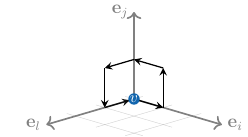} }
 \runter{ \includegraphics[width=5.6cm]{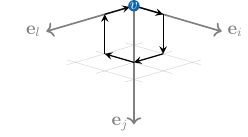} } \quad\,\,\, \text{Type III }\notag \\[2ex]
  \text{Type IV }&\runter{ \includegraphics[width=5.35cm]{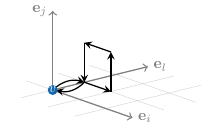} }\qquad\qquad
 \runter{ \includegraphics[width=4.0035cm]{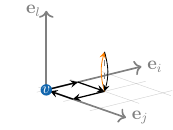} } \qquad\quad\text{Type V } \notag
 \end{align*}
 \caption{Types of length 6 paths in $T^d_m$ ($d\geq 3$) with
   nontrivial holonomy.  Additionally to the main text description:
   Type I shows both $p$ and $p'$. Type II${}_\mathrm a $ has two
   subtypes, the one on the left corresponding to $\tau$ being the
   identity, the rightmost to $(123456)$.  In Type V the insertion of
   the path with the orange arrow could occur at any of the points in
   the plaquette. Here we depict only the insertion with value
   $\alpha=2$, but all other values appearing in
   \eqref{insertions_paths} are meant too.  (As before, in case of
   ambiguity, the orange arrow is the `next one' in the path; curved
   arrows to ease visualisation.)\label{fig:typeslength6paths}}
\end{figure}
\begin{itemize}
 \itemb \textit{Case 1: If  $p$ does not have
 the holonomy of a plaquette, albeit the holonomy of $p$ is not trivial}. Then
 $p$ is any of the next path types with $|i|,|j|,|l| \in \{1,\ldots, d\}$:
 \begin{subequations}
   \begin{align} \qquad
 p_{\textsc{\scriptsize i}}(v;i,j)&= [e_i, e_i ,e_j, e_{-i},e_{-i},e_{-j}]   && &   |i| & \neq |j|  \label{p}\\
 p_{\textsc{\scriptsize i}}'(v;i,j)&= [e_i, e_j ,e_j, e_{-i},e_{-j},e_{-j}]   && &   |i| & \neq |j|\label{pprime} \\
\,\, \qquad p^\tau_{\textsc{\scriptsize ii}}(v;i,j,l)&= \tau \cdot   [e_{i}, e_{j} , e_{-i},e_{l}, e_{-j},e_{-l}]   && & |i|&,|j|,|l|  \text{ pairwise different}  \label{pII} \\
%  \qquad p_{\textsc{\scriptsize iii}}(i,j,l)&= [e_i, e_j ,e_{-i},e_{l},e_{-j}, e_{-l}]  && &  |i|&,|j|,|l|  \text{ pairwise different}\\
 \qquad p_{\textsc{\scriptsize iii}}(v;i,j,l)&= [e_i, e_j ,e_l, e_{-i},e_{-j},e_{-l}] && &  |i|&,|j|,|l|  \text{ pairwise different}\label{pIII}
  \end{align}
 \end{subequations}%
where $\tau \in \Sym(6)$, which acts by permutation of the six
arguments, is one of
\begin{align} \tau=\id_6 \quad \tau=(123456) \quad \text{ or } \quad  \tau=(135)(246)\,. \label{pemutations_allowed}\end{align}
Further, $e_\alpha$ denotes the edge parallel to $\mathrm{sgn}(\alpha)
\be_{|\alpha|}$ based at the only vertex that makes the path in
question well-defined (making outgoing sources and incoming targets
coincide, see Types I,II and III in Fig. \ref{fig:typeslength6paths})
and based at $v$.  \itemb \textit{Case 2: The path $p$ has the
  holonomy of a plaquette}. Concretely, let $(p_1) ^\vee
_{\,\alpha}(p_2)$ denote (whenever well-defined) the insertion of the
path $p_1$ into the path $p_2$ after the $(\alpha-1)^{\text{th}} $
vertex of the latter. Then $p$ has for $|i|, |j| \in \{1,\ldots, d\}$
with $|i| \neq |j|$ and arbitrary $ l$ the following types
\begin{align} \label{insertions_paths}
p_{\textsc{\scriptsize iv}}
= e_{-l}\cdot  P_{i,j} \cdot e_l
\qquad
p_{\textsc{\scriptsize v}}= [e_l,e_{-l}]^\vee_{\alpha} P_{i,j} \qquad \alpha=\{0,1,2,3,4\}
\end{align}
where $P_{j,l}$ is a plaquette based as in the Types IV and V of
Figure \ref{fig:typeslength6paths}.  \itemb \textit{Case 3: The path $p$
  has no holonomy.} We only care about the number $\theta_0(d)$ of
such paths.
\end{itemize}
We now count how many paths per type exist.
\begin{itemize}
 \itemb \textit{Type I:} Here $i,j\in\{-d,\ldots,-1,1,\ldots,d\}$ are
 the only parameters, and it is only required that $|i| \neq |j|$,
 else the path has trivial holonomy. Thus there are $2d(2d-2)$ Type I
 paths of the form \eqref{p} (that is
 \runterhalb{\includegraphics[height=8pt]{hol_recthor}} ) and the same
 number for \eqref{pprime}, or
 \runterhalb{\includegraphics[height=11pt]{hol_rectver}} in form.
 \itemb \textit{Type II:} Depending on the cycle $\tau $ in
 $p^\tau_{\textsc{\scriptsize ii}}(v;i,j,l)$, there are two subcases:
\begin{itemize}
 \item[-]\textit{Type II$_{a}$}: If $\tau=\sigma$
 or $\tau=\sigma^2$,  being $\sigma=(123456)$. See Figure \ref{fig:typeslength6paths}.
    \item[-] \textit{Type II$_{b}$}: When $\tau$ is the trivial
      permutation $\id_6$.
\end{itemize}
These three choices yield two paths that are independent in the sense
that, e.g. $p_{\textsc{\scriptsize II}}^{\id_6}$ cannot be obtained
from $p_{\textsc{\scriptsize II}}^{\sigma}$ or $p_{\textsc{\scriptsize
    II}}^{\sigma^2}$ just by a different choice of their arguments.
On the other hand, the other permutations $\sigma^q$, $q=3,4,5$
corresponding to the other three rootings of the polygon
(i.e. vertices of the path where to put $v$) are dependent from the
first three.  Since the holonomies are presented as a sum over
$i,j,l$, $p^{\sigma^q}_{\textsc{\scriptsize II}}$ for $q>2$ are
already considered in the cases for lower $q$, by symmetry
arguments. For instance, for $q=3$, the permutation
$\sigma^3=(14)(25)(36) $ yields $p_{\textsc{\scriptsize
    II}}^{(14)(25)(36)}(v;i,j,l) = p_{\textsc{\scriptsize II}}^{\id_6}
(v; l,-j,i)$.  \par

 Notice that we rewrite all the holonomies for Type II$_{a}$
 as Type II$_{b}$ at a shifted vertex  (it is
 more natural to see this path with form of `open door'
 as being based at the one of the bases of `its hinge').
 Concretely,  the holonomy of $p_{\textsc{\scriptsize II}}^{\sigma}(v;i,j,k)$
  coincides with that of $p_{\textsc{\scriptsize II}}^{\id_6}( v -\be_i ; i,j,l)$;
  similarly,
  the holonomies of $p_{\textsc{\scriptsize II}}^{\sigma^2}(v;i,j,k)$  and
  $p_{\textsc{\scriptsize II}}^{\id_6}(v-\be_i-\be_j; i,j,k)$ are equal.
    But this means also that the Type II$_{a}$ paths from the neighbours
  in precisely the opposite direction are gained back. Thus the factor $\theta_{\runterhalb{\includegraphics[height=5pt]{hol_puerta}}}$ of
  $\hol (\runterhalb{\includegraphics[height=10pt]{hol_puerta}}, v;i,j,l )$ is 3.
  All this is possible since the trace of $D^6$ is
 a sum over paths based at all vertices.

 \itemb \textit{Type III:} The parameters $i,j,l$ have pairwise
 different absolute values, there are thus $2^3d(d-1)(d-2)$ such
 paths.  \itemb \textit{Type IV:} For each of the $2d$ nearest
 neighbours of $v$ there are $4d(d-1)$ `shifted' plaquettes. They will
 contribute to the sum of all holonomies of such neighbours, and do
 not contribute to plaquette-holonomies based at $v$.  By the same
 argument, the neighbours $\{v- \be_i\}_{i, |i|=1,\ldots,d} $ provide
 shifted plaquettes at $v$, as depicted here:
 \begin{align*}
 \runter{ \includegraphics[width=5cm]{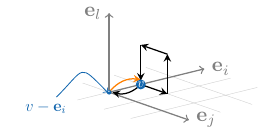} }
 \end{align*}
 So due to this shift, we anyway obtain plaquette-holonomies based at
 $v$, in total $2d$ times the usual $4d(d-1)$ different plaquettes.

 \itemb \textit{Type V:} There are $2d$ possible path insertions
 of the type $[e_l,-e_l]$  at any of the vertices $\alpha=0,1,2,3,4$, $4d(d-1)$
 of the plaquette based at $v$. This yields $5\times 4d(d-1)$.

\end{itemize}
From Types IV and V we conclude that number of times that
$\hol(\Box,v;i,j,l) $ is repeated is $\theta_{\square}=12 d$.  The
final coefficient $\theta_0(d)$ is then the total number $c_6(d)$ of
length-$6$ paths minus Case I and II.  By Lemma
\ref{thm:numberloopslattice} or Ex.~\ref{ex:numberloopslattice}, one
knows $c_6(k)$, and by the analysis above one has
\begin{align*} \nonumber
\theta_0(d) &=c_6(d)-  \bigg[\sum_{\text{X$=$I, II,\ldots, V}} \#\{\text{paths in Type X}\} \bigg] \\
&= 120d^3 - 180d^2 + 80d   -  \{ (\theta_{\includegraphics[height=4pt]{hol_recthor}} (d)+  \theta_{\includegraphics[height=4pt]{hol_rectver}}(d) +
\theta_{\square}(d) )\times 4d(d-1) \\
&\hspace{4.5cm} + [\theta_{\includegraphics[height=5pt]{hol_puerta}} (d) +  \theta_{\includegraphics[height=8.3pt]{hol_hex}}(d ) ] \times 8d(d-1)(d-2) \nonumber
\}
\\&=120d^3 - 180d^2 + 80d
-( 80d^3 - 136d^2 + 56d) \,.\qedhere
\end{align*}
\end{proof}
% \theta_{\includegraphics[height=5pt]{hol_puerta}} (d)& = 3 \\
% \theta_{\includegraphics[height=8.3pt]{hol_hex}} (d) & = 1 =
%  \theta_{\includegraphics[height=4pt]{hol_recthor}} (d)  =  \theta_{\includegraphics[height=4pt]{hol_rectver}} (d)

\subsection{Remarks on the target category}\label{sec:theMvScategory}
We now justify why we chose $\pS$ as target category.
% Mainly why we did not use directly previous constructions,
% and in Sec. \ref{sec:kernels}, why we considered faithful actions.
\subsubsection{Implications of $\mathcal S$-representations for the Higgs}
 In
\cite[Sec.~4.2]{MvS}, a spin $4$-manifold $M$ is assumed for $Q$ to
embed there, $Q\subset M$. The manifold $M$ induces a Dirac operator
on $Q$ that in fact inspired us to build $D_Q$ in
Eq.~\eqref{defD}. However, two additional differences, beyond those already
apparent before, exist: first, ours
does not assume a background manifold; further, and essentially,  the
hermitian (Higgs) field that emerges from tracing our $D_Q$ is not
forced to be constant on $Q_0$.
With the aim to justify the second remark, we give details on the
category $\mathcal S$ used by Marcolli-van Suijlekom.  Its objects are
spectral triples of finite dimension and $\mathcal S$-morphisms
$(\phi,L):(A_s,H_s,D_s)\to (A_t,H_t,D_t) $ consist of a
$\ppS$-morphism with the additional condition (here again $s,t$ are
just labels)
\begin{align}
L D_s L^*= D_t \qquad \qquad\text{\cite[Eq. 2]{MvS}}\,. \label{isospectrality}
\end{align}
Despite its naturalness\footnote{ `Naturalness' in the sense that
\textit{not} asking \eqref{isospectrality} will be not natural, since
this would imply that isomorphic objects in that (hypothetical)
category would exist with different spectra.  But this, only if the
Dirac operators play a role in the category. In the present paper, our
Dirac operator emerges from representation theory in a category
without Dirac operator (for whose objects therefore
\eqref{isospectrality} plays no role).} in the spectral noncommutative
geometry context, quiver $\mathcal S$-representations are too
restrictive to yield non-trivial local action functionals.  It is
precisely this latter condition what we want to avoid by working only
in $\pS$ and building a Dirac operator out of a quiver
$\pS$-representation. Otherwise, assuming
\eqref{isospectrality}, an $\mathcal S$-representation of a connected
quiver $Q$ forces isospectrality\footnote{The author thanks Sebastian
Steinhaus for this remark, made in private communication. Only the
proof that follows is by the author.} of the Dirac operators at all its vertices,  yielding
a constant Higgs field, as we prove:\\[-.5ex]

\noindent
Indeed, by
connectedness of $Q$, there is always at least one path $p$ in the underlying
graph $\Gamma Q$ between arbitrary vertices $v, w\in Q$.  Applying
Eq.~\eqref{isospectrality} to each edge $e_j$ of the path
$p=[e_1,\ldots,e_k]$ one has $D_{w}=U_p D_{v} U_p^*$ where
$U_p=
L_{e_k}^{\varepsilon_k} \cdots L_{e_2}^{\varepsilon_2}L_{e_1}^{\varepsilon_1}$ and signs $\varepsilon_j=\pm$ for
$1\leq j\leq k$ have the effect to correct the orientation of the edges that has been
`forgotten' by passing from $Q$ to $\Gamma Q$.  In any case, $U_p$ is
again a unitary matrix, as each $L_{e_i}$ is, and therefore the characteristic polynomials of $
D_w= U_p D_{v} U_p^*$ and of $D_v$ are the same. They then share
spectrum, so $\Tr_v (D^n_v)= \Tr_w (D_w^n)$, $n\in\N$.\\[-6ex]

\noindent
\begin{figure}[h!]\noindent
\begin{minipage}{.5384\textwidth}\noindent
This implies that if one constructs operators that are polynomial
($\mathscr O(\mathfrak h)=\sum_i a_i \mathfrak h^i$, $a_i\in \R$) in a Higgs scalar
field $\mathfrak h$ by tracing powers of a Dirac operator that representation
theory attaches to the vertex of the quiver, then the use of $\mathcal
S$-representations---from them, concretely
Eq.~\eqref{isospectrality}---prevents the construction of a local
non-trivial action, as the above implies the constancy of
$Q_0 \ni v \mapsto \Tr _{v}(D_{v}^n)$, for all $n\in \N$. In the case of
\cite[Sec. 4]{MvS}, with $Q=\Z^4$ this yields in the smooth limit  a constant Higgs field $\mathfrak h$ on $\R^4$ (and whenever defined,
$\int_{M} \mathscr O(\mathfrak h) = \vol
(M) \cdot \mathscr O[\mathfrak h(x)]$
where $x$ is \textit{any} point of
the manifold $M$, if one uses the same category $\mathcal S$).
  \\[1ex]
This ends the proof. As a side comment, in
retrospective, for a classical, discrete (or PL-)manifold made of gluings
of polygon or higher dimensional blocks, the spectra of
its different pieces need not be the same. Indeed, a discrete surface
as in Figure \ref{fig:nonreflexive} made of different blocks (therein
regular hexagons and pentagons) has pieces with different
spectra, by Weyl's law.
\end{minipage}
 \hfill
 \begin{minipage}{.436\textwidth}
 \qquad\runter{\includegraphics[width=6.592cm]{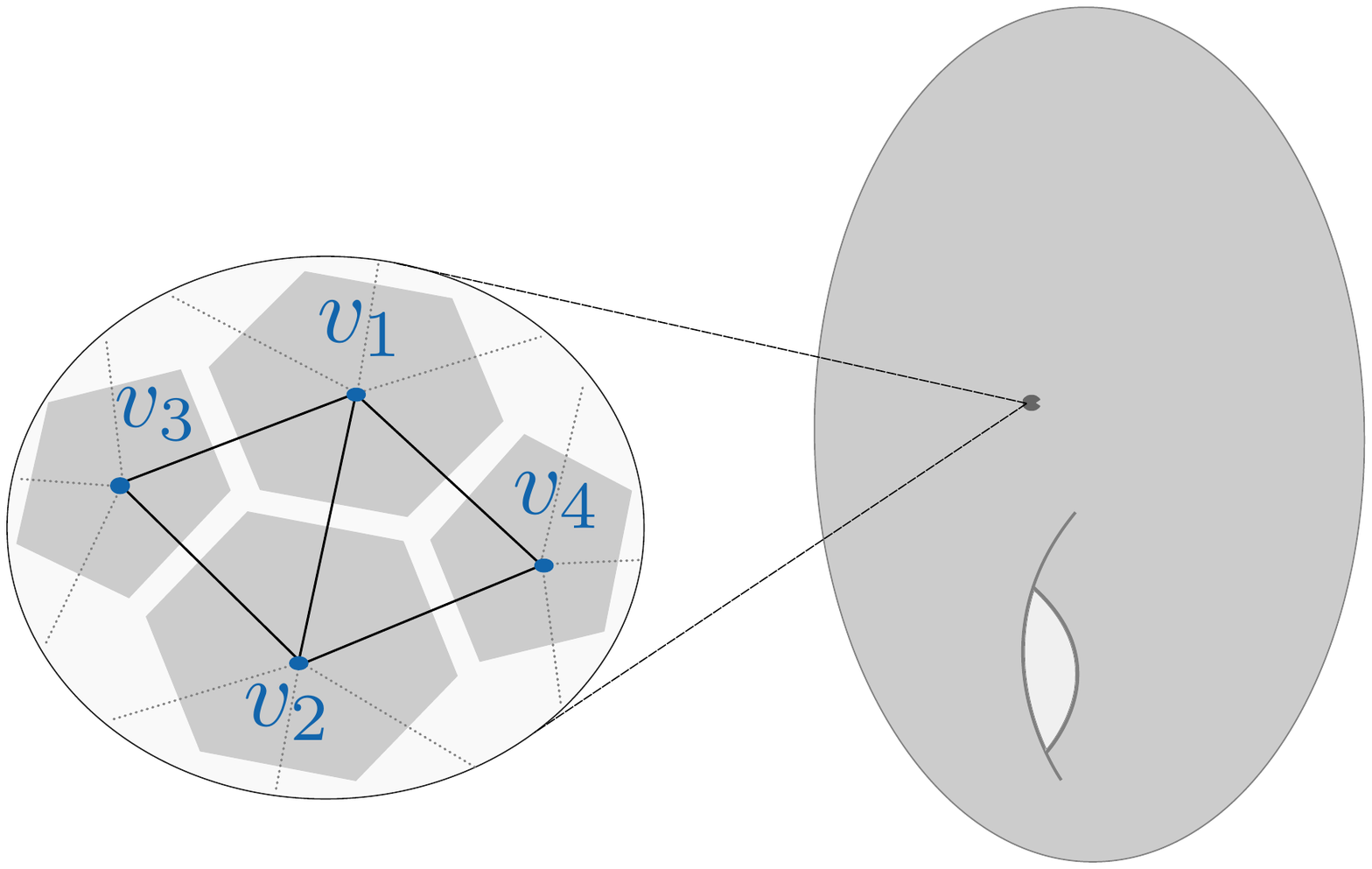}} \vspace{-2ex}
  \caption{A discretised manifold is shown on the right. When one
    looks in detail at these polygonal equilateral building blocks,
    even in a classical setting Dirac operators for the polygonal
    regions, thus those at the vertices of the dual graph or quiver,
    need not be isospectral. (This illustration is not playing a role
    in the proof.)
 \label{fig:nonreflexive}}
\end{minipage}
\end{figure}
\vspace{-2ex}

\noindent
% \begin{remark}\label{rem:lambda_weg}

\subsubsection{Kernels}\label{sec:kernels}
As a final remark, allowing $\lambda$ to have a non-zero kernel $\ker
\lambda= \{a\in A : \lambda(a)=0\}$ has been studied in \cite{MvS},
yielding some neat computations of representation spaces in terms of
homogeneous spaces.  These kernels of the action $\lambda$ are not to
be confused with the kernels of the maps associated to the edges,
which emerge in the enumeration of indecomposable quiver representations in
vector spaces; see \cite[Sec. 6.2]{Etingof}. For, in the
first situation, the kernels are those $\ker \lambda_v $ of the action
$\lambda_v:A_v \curvearrowright H_v$ associated to objects
$(A_v,\lambda_v,H_v) \in \pS$ that sit at vertices $v\in Q_0$. In
contrast, given a $\mathsf{Vect}$-representation $ (W_v, T_e)_{v\in Q_0, e
  \in Q_1}$ of $Q$ on vector spaces, the kernels that matter (in the
enumerative sense of above) are the kernels of maps associated to edges,
$\ker T_e: W_{s(e)}\to W_{t(e)}$.
% \end{remark}

\subsection{Proposal for quantisation: integration over $\Rep Q$} %\label{sec:QgaugeTh}
The tentative quiver partition function
\begin{align*}
``Z(Q)=\sum_{((A,H),(\phi,L))\in \ReppS(Q)}  \int
\ee^{-\frac1\hbar \Tr f [D_Q(L)/\Lambda ] }
\big(\textstyle\prod_{e\in Q_1}   \dif L_e\big) \text{''}
\end{align*}
is hard to state even in terms of formal series.  This was the reason
to introduce the restricted space of $N$-dimensional representations
in \eqref{rep_N}. Claim \ref{thm:boundRepQ} implies that $Z_N^f(Q)$ is given by a finite
sum, in fact indexed by Bratteli networks of dimension $N$  (being $N=\dim H_v$ for all $v$,
cf. Def. \ref{def:Bratteli_sieć}),
of finite-dimensional integrals over unitary groups. A
partition function proposal is
\begin{align*}
Z^f_N(Q) :=\sumsub { \text{Bratteli networks} \\ (\bn,\br ) \text{ on $Q$}}  \int_{\prod_{e\in Q_1}
 \uni(\bn_{t(e)} )}
\ee^{- \Tr f [D_Q(L)/\Lambda ] }
\dif \mu_{\bn,\br}(L)\in \C[[f_0,f_1,\ldots,]] \,\, \text{ for fixed } N\in \Z_{\geq 1}\, , \nonumber
 \end{align*} where  the integral is performed over all
 edge-assignments $e\mapsto L_e$ that verify $R=(\bn ,\br ; L) \in
 \Rep^N_\pS(Q)$. The measure $\dif \mu_{\bn,\br}(L)$ is a product Haar
 measure. For the evaluation of (the expectation value of) Wilson loops, integrals of the types
 Gross-Witten-Wadia and Harish-Chandra--Itzykson–Zuber are useful. In
 fact, an initial step of the planed extended analysis were the
 Makeenko-Migdal equations for this type of path integral
 \cite{MakeenkoMigdalNCG}. \par  As a final remark, when $\# Q_0$ is
 infinity, a question that remains for future work is whether the
 axioms of the spectral triple in infinite dimensions hold
 (Rem.~\ref{rem:infdims}), with or without extensions by reality and
 chirality operators (Rem.~\ref{rem:uwaga_Higgs}).  \\[-3ex]

 \addtocontents{toc}{\protect\setcounter{tocdepth}{-1}}
 \section*{Support received during this work}%
 \fontsize{9.2}{9.0}\selectfont This work was mainly supported by the
 European Research Council (ERC) under the European Union’s Horizon
 2020 research and innovation program (grant agreement No818066) and
 also by the Deutsche Forschungsgemeinschaft (DFG, German Research
 Foundation) under Germany’s Excellence Strategy EXC-2181/1-390900948
 (the Heidelberg \textsc{Structures} Cluster of Excellence).  \par An
 \textit{ESI Junior Research Fellowship} of the Erwin Schrödinger
 International Institute for Mathematics and Physics (ESI) Vienna,
 where important part of this article was written under fantastic
 conditions of work, is acknowledged.  Also the organisers of the
 OIST-Workshop `Invitation to Recursion, Resurgence and Combinatorics'
 in Okinawa, Japan, are acknowledged for travel support. For an office
 place and allowing my participation in `Quantum gravity, random
 geometry and holography', the support of the Institut Henri Poincaré
 (UAR 839 CNRS-Sorbonne Université) and LabEx CARMIN
 (ANR-10-LABX-59-01) is acknowledged.  \\[-3ex]

 {\fontsize{11.395}{14.0125}\selectfont%
  \section*{Acknowledgements}%
 }\fontsize{9.2}{9.0}\selectfont I thank both L Glaser
 (\text{U. Wien}) and S. Steinhaus (\text{U. Jena}) for fruitful
 discussions, and the former additionally for the very kind and
 long-term hospitality.  That several years ago W. van Suijlekom
 (U. Radboud) shared his enthusiasm for the core-topic was important
 to start this work. The author thanks an anonymous referee for
 corrections and numerous useful comments that led to improvements.
\\

 \fontsize{10.395}{13.0125}\selectfont%
 \addtocontents{toc}{\protect\setcounter{tocdepth}{1}}
\appendix

\section{Counting loops in a lattice} \label{app:Pathcounting}
In the main text several computations of the spectral action $\Tr
f(D)$ are presented in terms of number of paths. This appendix
estimates the growth in $\deg f$ of the contributions to the spectral
action.  Since we are on a square lattice, contributions come from an
even $\deg f$ (if we have enough lattice points in order to avoid
`straight' loops caused by cyclic boundary conditions, which explains
$m>k$ below).

\begin{lemma}[Number of loops on the lattice] \label{thm:numberloopslattice}
For even $k$ and $m>k$
the number $c_d(k)$ of length-$k$
closed paths in $(T^d_m)^\star $ based at any point reads
\begin{align}
c_{d}(k)=\!\!\!\! \sumsub{ \bmu\,\, \vdash\, k /2 \\ \quad
\,\bmu\,=\,(\mu_i )_{i=1}^d \\\,\,\,  \,\,\mathrm{ordered }
\,\,\bmu \in \Z_{\geq0}^d
}
\frac{ k! }{ \big [ \prod_{j =1}^{d}  \mu_j !  \big ]^2 }\,,
\label{pathcountingGeneral}
\end{align}
where the integer partition $\boldsymbol \mu \vdash k/2$
does allow zero-entries of the $d$-tuple.
\end{lemma}

\begin{proof} Let $\be_1,\ldots, \be_d$ be the standard basis vectors of $\Z^d$.
 A path $p$ based at any point of $(T_m^d)^\star$ is determined,
 first, by the number $\mu_i \in \{0,1,\ldots, k/2\}$ of steps in
 positive direction $\be_i$ for each $i=1,\ldots,d$. Since
 $s(p)=t(p)$, the number of steps opposite direction $-\be_i$  to the
 basis vector, is equally $\mu_i$; the second datum determining $p$ is
 a a permutation $\tau \in \Sym(k)$ that orders all the steps, which
 have to be $2\mu_1 +2\mu_2+\ldots +2\mu_d=\ell(p)=k$ in number,
 whence $\boldsymbol \mu \vdash k/2$.\par But $\tau$ is unique only up
 to $2d$ permutations: one of $\Sym(\mu_j)$ for each $j=1,\ldots, d$,
 which accounts for the multiplicity of the steps along the positive
 $j$-axis, and another independent permutation of all the steps of the
 negative $j$-axis, thus also in $\Sym(\mu_j)$. This reduces the
 symmetry to
\begin{align*}
[\tau]  \in \frac{\Sym(k)}{ \Sym(\mu_1)^2\times \Sym(\mu_2)^2  \cdots  \times \Sym(\mu_{d})^2}\,,
\end{align*}
which has as many elements as those summed in \eqref{pathcountingGeneral}.
\end{proof}

\begin{example}\label{ex:numberloopslattice}
We check Formula \eqref{pathcountingGeneral} above against explicit counting.
\begin{itemize}
 \itemb For $d=1$, $c_1(k)={ k \choose k/2 }$, since the only
 $d$-tuple partition of $k/2$ is $\mu=(k/2)$ itself.
 \itemb
 According to Lemma \ref{thm:numberloopslattice}, the number of length-$k$ paths on
 the plane rectangular lattice is
 \begin{align*} \nonumber
 \qquad \,\,\quad
 c_2(k)=\sumsub{ \mu_1,\mu_2 \geq 0\\
 \mu_1 + \mu_2 =k/2}
 \frac{k!}{(\mu_1!)^2 (\mu_2!)^2}
 =\frac{k!}{(k/2) ! (k/2)! } \times \!\!\!\! \sumsub{ \mu_1,\mu_2 \geq 0\\
 \mu_1 + \mu_2 =k/2}
 \frac{(k/2)!}{\mu_1! \mu_2!}
 \frac{(k/2)!}{\mu_1! \mu_2!}={ k \choose k/2 }\times{ k \choose k/2 }\,,
 \end{align*}
thanks to the Vandermonde identity. The pattern then breaks, $c_3(k) \neq {k \choose {k/2}}^3$.

 \itemb To count length-$6$ paths one needs the partitions of $6/2$,
 $\{1,1,1\}$, $\{2,1\} $ and $\{3\}$. For $d\geq 3$ there are $d
 \choose 3$ ways to add zeroes to the first partition to make an
 ordered $d$-tuple; $d (d-1)$ ways for $\{2,1\} $ and $d$ ways for
 $\{3\}$. Thus
  \begin{align*} c_d(6) = 6! \textstyle\Big\{ \frac{(d-2)(d-1)d}{ 3!  (1! 1! 1! 0!\cdots 0!)^2} +
   \frac{d(d-1)}{(2! 1! 0!\cdots 0!)^2} + \frac{d}{(3! 0!\cdots 0!)^2}
  \Big \} = 120d^3 - 180 d^2 + 80 d\,.   \end{align*}
\end{itemize}

\end{example}

The next bounds help to estimate the growth of paths
that contribute to $\Tr (D_Q^l)$. Nevertheless,
most of them yield constant terms in the spectral action,
it is obvious from Section \ref{sec:physics}.
\begin{lemma}\label{thm:PathsKn}
Given $l\in \Zpos$ and
any graph $\Gamma$, let $t_{\Gamma} (l )$ be the number of length-$l$
closed paths on $\Gamma$. Then
\begin{enumerate}
  \item $t_{K_n}(l) =(n-1)^l +(n-1) \cdot (-1)^l$ for the complete
    graph $K_n$.
\vspace{5pt}
\item $t_{K^\circ_n}(l)= n^l $ for the complete graph $\mathring K_n$
  enlarged by self-loops.
\vspace{5pt}
%  \item if $\nu$ is the maximum number of edges between
%  any two vertices in $Q$, . Then
\item Consider the graph $G{(n, \lambda,\nu)}$ with $n$ vertices, with
  exactly $\lambda$ self-loops at each vertex and $\nu$ edges between
  any pair of different
  vertices. Then \begin{align*}t_{G(n,\lambda,\nu)} (l) = [ (n-1) \nu +
      \lambda ]^l +(-1)^l( n-1) \cdot (\nu-\lambda)^l \,.\end{align*}
%    \vspace{3pt}
\item For a quiver $Q$, let $\Gamma Q $ denote its underlying graph. Letting
\begin{align*}\qquad \, n=\#Q_0 \text{ and } \nu =\max_{v,w\in Q_0} \# \{ e\in \Gamma Q _1 : e=(v,w)  \}\,,
 \text { it holds } t_{\Gamma \kol Q }(l) \leq   n^l  \nu^l\,. \notag
 \end{align*}
\vspace{-10pt}
  \item
 Let $\lambda=\max_{v\in Q_0} \# \{ e\in \Gamma Q _1 : e=(v,v)  \}  $, i.e. the maximum
 number of self-loops at any vertex in $Q$. Then
 \begin{align*} t_{\Gamma Q }(l) \leq    [ (n-1) \nu + \lambda ]^l +( n-1) \cdot  (\lambda -\nu)^l  \,. \end{align*}
\end{enumerate}

\end{lemma}
\begin{proof}
It is a graph theory fact---but it follows also from the proof of
Corollary \ref{thm:Wksym} by letting $A=\C$ and by choosing unit
weights $1$ on each edge therein---that the $(i,j)$-th entry of the
$l$-th power of the adjacency matrix of any graph $G$ counts the
number of length-$l$ paths (made of edges) of $G$ between $i$ and $j$.
Then the trace of the $l$-th power of the adjacency matrix counts all
length-$l$ loops in $G$. \par
% \noindent
% \begin{minipage}{.5\textwidth}
We work out the first case and all others follow.  The adjacency
matrix $\mathscr A_n$ of $K_n$ is the constant matrix with zeroes in
the diagonal (since there are no self-loops) and filled elsewhere ones
(since exactly one edge connects any two different vertices), so
$(\mathscr A_n)_{i,j}=1-\delta_{i,j}$ for $i,j\in
\{1,\ldots,n\}=(K_n)_0$. Let $E_n$ be the matrix whose entries are all
ones, or $E_n=\mathscr A_n+1_n$, where $1_n$ is the identity matrix.
Clearly $\Tr E_n=n$ and $(E_n)^k=n^{k-1} E_n$ for $k>1$, so
% \end{minipage}
% ~
% \begin{minipage}{.45\textwidth}
\allowdisplaybreaks[3]
\begin{align}
t_n(l ) &= \Tr  ( \mathscr A_n^l   ) \label{depaso}\\
& = \sum_{k=0}^l  \binom{l }{k}\Tr \big[ (E_n)^{k} (-1_n)^{l -k} \big] \notag \\
&= \sum_{k=1}^l  \binom{l }{k}\Tr \big[  n^{k-1} E_n (-1_n)^{l -k} \big]  +n  (-1)^l   \notag\\ \notag
&= \sum_{k=1}^l  \binom{l }{k} n^{k-1} (-1)^{l -k}  \cdot n +n  (-1)^l  \\ % \Tr \big[   E_n  \big]
&=  \sum_{k=0}^l  \binom{l }{k} n^{k} (-1)^{l -k} -(-1)^l   + n(-1)^l  \notag \\ \notag
&= (n-1)^l  +(n-1) \cdot (-1)^l \,.
\end{align}%
% \end{minipage}
\allowdisplaybreaks[1]%
But notice that $E_n$ is the adjacency matrix for $\kol K_n$, so
equally from $\Tr E_n=n$ and $(E_n)^l=n^{l-1} E_n$, it follows
$t_{K^\circ_n}(l)= \Tr( E_n^l) = n^l $.  \par For the fourth
statement, observe that any entry of the adjacency matrix $\mathscr
A_Q$ of $Q$ satisfies $(\mathscr A_Q)_{i,j}\leq \nu $ by definition of
$\nu$, so $ t_{\Gamma \kol Q}=\Tr [ (\mathscr A_Q)^l ]\leq \Tr[ (\nu
  E_n)^l ]=\nu^l n^l$ where the last equality follows by the second
statement.  In the third statement, $t_{G(n,\lambda,\nu)}$ is obtained
by replacement of $E_n\mapsto \nu E_n$, and $1_n\mapsto \lambda 1_n$
in the \textsc{rhs} of Eq.~\eqref{depaso} and similar
manipulation. Finally, the fifth follows from the third and by $
t_{\Gamma Q }(l) \leq t_{G(n,\lambda,\nu)}$, which is obtained by an
obvious bound entry-wise, using the definitions of $\lambda$ and $\nu$.
\end{proof}
\noindent
Lemma \ref{thm:PathsKn}
generalises the next \textsc{oeis}-entries  \cite{oeis_walks_Kn}
\begin{align*} &\text{A092297 for $n=3$, }
\{t_3(l)\}_{l=1,2,3,\ldots } = \{ 0, 6, 6, 18, 30, 66, 126, 258, 510,\ldots\} \\
&\text{A226493 for $n=4$, } \{t_4(l)\}_{l=1,2,3,\ldots }= \{0, 12, 24, 84, 240, 732, 2184, 6564,\ldots \}
\end{align*}
to arbitrary $n$ (which are unreported at \textsc{oesi} $n>4$).

\section{The size of the $N$-dimensional subspace of $\Rep Q$} \label{app:boundRepQN}

The next result estimates a rough bound for the dimension of the restricted representation space.
 Since the previous result shows that $\ReppS (Q)$ splits as a collection of
 unitary groups labelled by Bratteli networks $(\bn,\br)$ and
$ \Rep^N_{\pS}(Q)$ is finitely generated by those,
say $\mathfrak b({N,Q}) \in \N$, Bratteli networks: $(\bn^ \alpha,\br^ \alpha )$, $\alpha=1,\ldots, \mathfrak b({N,Q})$,
one can define its real dimension as
$ \Rep^N_{\pS}(Q) = \sum_{\alpha=1,\ldots, \mathfrak b({N,Q}) }  \sum_{e\in Q_1} \dim_\R \uni(\mathbf n_{t(e)} ^ \alpha)$.
\begin{claim}\label{thm:boundRepQ}
For a connected quiver $Q$, the next bound holds:
\begin{align*}\dim_\R \Rep^N_{\pS}(Q) \leq N^{ 2\cdot \# Q_1} \times  [ (N^2)_{N} ]^{\# Q_1} \,,\end{align*}
where $(n)_m=n!/(n-m)!$ denotes the Pochhammer symbol, with $m,n,n-m\in \Z_{\geq 0}$.
% \end{align*}
\end{claim}
\noindent
\begin{proof} First we bound the
number of the former.  The connectedness of $Q$ means that for any two
$v,w\in Q_0$, there is a path in the underlying graph $\Gamma Q$
connecting $v$ and $w$, yielding $ \dim H_v = \dim H_w $, since $ \dim
H\sour = \dim H\targ $ holds for each $e\in Q_1$.  It thus suffices to
show the condition for one edge $e$ (cf. Lemma
\ref{thm:WedderburnArtinConsequence}, now in edge-dependent
notation). Due to Lemma \ref{thm:characHomFerm}
 \begin{align} \label{fuerschraenke}
 \sum_{i=1}^{l\sour} q_i m_i|\sour= \bq\sour \cdot \bm\sour = \dim H\sour =& N=  \dim H\targ =\br\targ  \cdot \bn\targ  =
 \sum_{j=1}^{l\targ} r_j n_j|\targ\,.
 \end{align}
 But $N$ constrains  also  $C(\Brat_e)$,
\begin{align}N = \sum_{i=1}^{l\sour} \sum_{j=1}^{l\targ} m_i [s(e)] C_{i,j}(\Brat_e) r_j [t(e)] \geq
\sum_{i=1}^{l\sour} \sum_{j=1}^{l\targ}
C_{i,j}(\Brat_e) \label{Cconstrained}\end{align} (the latter due to
$m_i, r_j$ all being $\geq1$).  And this implies that there are at
most $N^{l\sour \times l\targ}$ such matrices (thus at most that many
$\Brat_e$'s).  Eq.  \eqref{Cconstrained} implies $N\geq
\max(l\sour,l\targ)$, since  the matrix $C(\Brat_e)$ must have at least one
nonzero in each column and in each row. Constraint
\eqref{Cconstrained} implies that the non-zeroes of $C(\Brat_e)$ are
at most $N$ in number; so the number of such Bratteli matrices is less
than $N!\times {N^2 \choose N}=(N^2)_N$, the number of ordered
embeddings of $N$ integers into a $N^2$ array (else filling with
zeroes).  This happens for each edge, so the total number of Bratteli
matrices in the quiver is $ \leq \big[\frac{{(N^2)}!}{(N^2-N)!} \big ]
^{\# Q_1 }$.  We now come to the contribution from the edge-labels.
For each $e\in Q_1$,
 $\dim_\R \uni(\bn\targ )=
 \sum_j n_{t(e),j}^2
 \leq (\sum_j n_{t(e),j})^2
\leq   N ^2$, the latter due to Eq.~\eqref{fuerschraenke}. Thus they contribute at most $(N^2) ^{\# Q_1}$.\qedhere
 \end{proof}

\fontsize{10.0}{14.0}\selectfont
\section{Notations and conventions}\label{sec:Glossary}
\fontsize{9.9}{12.3}\selectfont \noindent
In the main text, we attempted to stick to the following conventions and use of variables:\\[3ex]
\fontsize{10.0}{12.64}\selectfont
 \begin{tabularx}{.9\textwidth}{rX}
$\azul{\bullet} $ and $ \gris{\bullet}$& quiver-vertex and Bratteli-diagram-vertex, respectively\\
   $*$-$\mathsf{alg}$ & category of involutive unital algebras \\
   $A, A_v, B$ & involutive algebras \\
   $\mathscr{A}(b)$ & matrix of weights, for given $b:Q_1\to B$\\
$   b_e, b_{ij}$  & weights in $B$ for edges $e$ and $(i,j)$  \\
% bond & edge in lattice context \\
% $\Brat $  & Bratteli diagram\\
$\Brat:\mathbf m\to \mathbf n$  & Bratteli diagram compatible with $\mathbf m$ and $ \mathbf n$\\
% $\Brat_0,\Brat_1$  & vertex and edges sets of a Bratteli diagram\\
$C, C_e, C_p$ & Bratteli matrix, Bratteli matrix evaluated at $e\in Q_1$, or $p\in \mathcal PQ$\\
$\C Q$ & path algebra of a quiver $Q$\\
$r \C^n$ & abbreviation of $\C^r\otimes \C^n$, usually with a tacit action of
$\M n$  on $\C^n$     \end{tabularx}~\\ \newpage
      \begin{tabularx}{.95\textwidth}{rX}
      \hphantom{$(v_1,v_2,\ldots,v_{k+1})$} &  \hphantom{A}\\[-2ex]
diamond ($\diamond$) & $X^\diamond$ means that $X$ is spurious (cf. `spurious' below)        \\
      $D_Q(L) $ & Dirac operator for a quiver representation \\
   $h_d(k)$ & vol. of the radius-$k$, $L^1$-sphere, in dim.-$d$ lattice \\
   $H, H_v$ & Hilbert spaces\\
      $e, e_j$ & typical edge variables, $e, e_j\in Q_1$ \\
%   $e_k e_{k-1}\cdots e_2 e_1$ & a length-$k$ path $[e_1, e_2,\ldots,e_k]$ \\
  $[e_1, e_2,\ldots,e_k]$ & a length-$k$ path $p$, $ e_j\in Q_1$, $p=\big(s(e_1),t(e_1),\ldots, t(e_k) \big) $\\
     $\mathbf e_j$ & standard basis vectors (lattice context)\\
   $E_v$ & for $v\in Q_0$, the constant, length-zero path at $v$\\
   $ \mathcal G(Q)$ & gauge group of a quiver ($\prod_{R/\simeq}\Aut_{\Rep Q} R$) \\
   loop & based closed path on a quiver \\
   $\lambda$ & a  $*$-action, typically $\lambda : A \curvearrowright H$\\
%      $\pS$&  category of prespectral triples\\
     $\Phi_e =(\phi_e,L_e)$ & morphism $X_{s(e)} \to X_{t(e)}$ \\
     $\phi_e$ & involutive algebra morphism $A_{s(e)}\to A_{t(e)}$ \\
     $\textrm{Func}(\cC,\mathcal{D})$& functor category $\cC\to \mathcal{D}$\\
     $\hol_b(p)$ & holonomy of a closed path $p$ w.r.t. weights $\{b_e\}_{e\in Q_1}$ \\
     $L_e$ & unitary map $H_{s(e)} \to H_{t(e)}$  \\
     $ \ell(p)$ & length of a path $p$ \\
     $K_n$ & complete graph in $n$ vertices
\\
     $N_v, N$ & usually $\dim H_v$ ($v\in Q_0$), or $N$ if vertex-independent \\
     $\mathcal N_k(Q,v)$ & radius-$k$ sphere $\subset Q_0$ around $v$ \\
     $ \mathcal N_k^d(v)$ &  abbr. for $\mathcal N_k(Q,v)$ when $Q$ is a $d$-dimensional lattice \\
     $m$ & number of vertices per independent direction of an orthogonal lattice\\
     $(\bn_Q,\br_Q)$ & Bratteli network on $Q$, sometimes abbreviated as $(\bn, \br)$ \\
     $|\bn|$ & for $\bn\in \Z_{\geq 0}^\infty $, number of non-zero entries in $\bn$ \\
     $o_v$ & typical notation for a self-loop at vertex $v$\\
     $O_m^d$ & $(T^d_m)^\circ$, that is $T^d_m$ with added self-loops \\
     $\mathcal P Q$ &  paths on a quiver $Q$, as well as the free or path category of $Q$\\
     $p,p'$ & paths, typically on a quiver \\
     $P_e$ & parallel transport along an (embedded) edge $e$ \\
     $\ppS$ & category of \textbf{p}re\textbf{s}pectral triples\\
     $\pS$ & agrees with the objects of $\ppS$; but the morphism structure of $\pS$ is changed\\
     quiver repr. & quiver repr. on $\pS$, unless otherwise stated\\
     $Q; Q_0, Q_1$ & a quiver; its sets of vertices and of edges, respectively \\
     $R=(X_v ,\Phi_e) $ & representation $Q\to \pS$\\
      $\Rep_{\K} (Q)$ & $\textrm{Func}(\mathcal PQ,\cC)$, representations of $Q$ in a category $\K$\\
     $\Rep_{\pS} (Q)$ & $\pS$-representations of $Q$ with $\dim H_v=N$ for some (thus each) $v\in Q_0$ \\
     $\rho$ & (graph-)distance on $Q$, $\rho: Q_1 \to \R_{>0}$\\
        $\uni(\bn)$ & $\prod_j  \uni(n_j) $ with $\bn=(n_1,n_2,\ldots)$
     \\
         $v,v',w,y$ & typical variables for vertices of a quiver\\
     $(v_1,v_2,\ldots,v_{k+1})$ & a length-$k$ path $p$, given $ v_j\in Q_0$ (if $Q$ has simple edges)\\
     $s, s(e)$ & source map, source of an edge $e$\\
     self-loop & length-$1$ loop\\
%       $v_j$ & in the lattice context, abbr. for $v_j=v + a \mathrm{sng}(j) \be_{|j|}$ \\
     spuriuos & a map $X^\diamond: Q_0 \to \K $ such that  $\prod_{e\in Q_1} \hom_\K (X\se^\diamond,X\te^\diamond)=\emptyset$ {\tiny (for some cat. $\K$)}   \\
     $t, t(e)$ & target map, target of an edge $e$\\
     $\Tr_v$ & shorthand for $\Tr_{H_v}$\\
     $T^d_m$ & quiver of $m^d$ vertices in dim.-$d$ lattice\\
   % $T(v)$ &  $t\inv(v)$, set of edges $e$ whose target is $v$\\
      $\Wils(p)$ & Wilson loop of a closed path $p$   \\
%      $X$ & prespectral triples $X=(A,\lambda,H) \in \pS$\\
%       \end{tabularx} \newpage
%       \begin{tabularx}{.85\textwidth}{rX}
%       \hphantom{$(v_1,v_2,\ldots,v_{k+1})$} &  \hphantom{A}\\[-2ex]
     $X_v$ & $X_v =(A_v,\lambda_v,H_v) \in \pS$ \\
   $X^\diamond$ &  spurious vertex-labels \\
    $\Phi_e$ & morphism $X_{s(e)} \to X_{t(e)}$ \\
    $\phi, \phi_e$ &  morphism of involutive algebras (typically $\phi_e: A\sour \to A\targ$)\\
    $\varphi_v$ &  diagonal entry in $D_Q(L)$, yields the hermitian field (`Higgs')\\
%      $\psi, \psi_v$ & a vector in $H$, a vector in $H_v$ (spinors) \\
     $\Omega Q $ and $ \Omega_v Q$ & resp. cyclic paths or loops on $Q$ and those based at $v$,  $s(p)=v=t(p)$ \\
     $\coprod$ & disjoint union  \\
     $\prod$ & Cartesian product, $\prod_{i=1,\ldots,n} X_i = X_1\times \cdots \times X_n$ (whenever defined)
%     $$ &
\end{tabularx}

% ORder of Greek alphabet
% Αα	Alpha	Νν	Nu
% Ββ	Beta	Ξξ	Xi
% Γγ	Gamma	Οο	Omicron
% Δδ	Delta	Ππ	Pi
% Εε	Epsilon	Ρρ	Rho
% Ζζ	Zeta	Σσς	Sigma
% Ηη	Eta	Ττ	Tau
% Θθ	Theta	Υυ	Upsilon
% Ιι	Iota	Φφ	Phi
% Κκ	Kappa	Χχ	Chi
% Λλ	Lambda	Ψψ	Psi
% Μμ	Mu	Ωω	Omega

%
% \fontsize{11.0}{10.0}\selectfont%
% \bibliographystyle{alpha}
% \bibliography{biblio_gauge_foams}
\newpage\newcommand{\etalchar}[1]{$^{#1}$}

\end{document}